\documentclass[10pt, oneside]{article}
\usepackage{mathrsfs}
\usepackage{amsfonts}
\usepackage{amssymb}
\usepackage{amsmath,amsfonts,amssymb,amsthm}
\usepackage{caption}
\usepackage{subfigure}
\usepackage{cite}
\usepackage{color}
\usepackage{graphicx}
\usepackage{subfigure}
\usepackage{float}
\usepackage{paralist}
\usepackage{indentfirst}
\usepackage{authblk}
\usepackage{bm}
\usepackage[colorlinks,
linkcolor=red,
citecolor=blue
]{hyperref}
\usepackage{color} 

\definecolor{red}{rgb}{1.00,0.00,0.00}
\definecolor{blue}{rgb}{0.00,0.00,0.63}
\definecolor{black}{rgb}{0.00,0.00,0.00}
\definecolor{purple}{rgb}{0.00,1.00,0.00}
\definecolor{pink}{rgb}{0.95,0.01,0.08}

\usepackage[T1]{fontenc}

\def\var{\varepsilon}
\def\bv{\bm{v}}

\def\bz{\bm{z}}
\def\bZ{\bm{Z}}
\def\cC{\mathcal C}

\def\cX{\mathcal X}
\def\dB{\dot {B}}
\def\tL{\widetilde{L}}
\def\ddj{\dot \Delta_j}

\def\ddj{\dot \Delta_j}
\def\ddjj{\dot \Delta_{j'}}

\def\ddjj{\dot \Delta_{j'}}

\newcommand{\R} {\mathbb{R}}

\newcommand{\Z} {\mathbb{Z}}
\newcommand{\N} {\mathbb{N}}

\newtheorem{Thm}{Theorem}[section]
\newtheorem{Prop}{Proposition}[section]

\newtheorem{Lemma}{Lemma}[section]
\newtheorem{cor}[Lemma]{Corollary}

\usepackage{tikz}

\def\bma#1\ema{{\allowdisplaybreaks\begin{aligned}#1\end{aligned}}}

\numberwithin{equation}{section}

\begin{document}

\title{{\LARGE \textbf{Strong relaxation limit and uniform time asymptotics of the Jin-Xin model in the $L^{p}$ framework}}}

\author{Timoth\'{e}e Crin-Barat, Ling-Yun Shou \& Jianzhong Zhang}

\date{}

\renewcommand*{\Affilfont}{\small\it}
\maketitle

\begin{abstract}
We investigate the time-asymptotic stability of the Jin-Xin model and its diffusive relaxation limit toward viscous conservation laws in $\R^d$ for $d\geq 1$. First, we establish a priori estimates that are uniform with respect to both the time and the relaxation parameter $\var>0$, for initial data in hybrid Besov spaces based on $L^{p}$-norms.
This uniformity enables us to derive $\mathcal{O}(\varepsilon)$ bounds on the difference between solutions of the viscous conservation law and its associated Jin-Xin approximation, thus justifying the strong convergence of the relaxation process. Furthermore, under an additional condition on the initial data, for instance, that the low frequencies belong to $L^{p/2}(\mathbb{R}^{d})$, we show that the $L^{p}(\mathbb{R}^d)$-norm of the solution to the Jin-Xin model decays at the optimal rate $(1+t)^{-d/{2p}}$, and the $L^{p}(\mathbb{R}^d)$-norm of its difference with the solution of the associated viscous conservation law decays at the enhanced rate $\var(1+t)^{-d/{2p}-1/2}$.
\end{abstract}


\noindent{\textbf{AMS (2010) Subject Classification:} 35B25; 35L40; 35L45; 35K55}



\section{Introduction}
\subsection{Presentation of the model}
The return to equilibrium of perturbed systems, referred to as the relaxation phenomenon, occurs in a wide variety of physical situations, such as blood flows with friction, non-equilibrium gas dynamics, kinetic theory, traffic flows, etc. (see \cite{Natalini,Vicenti1,Whitham1}). Liu \cite{Liu1} first studied the relaxation of $2\times2$ hyperbolic systems in one spatial dimension. Chen, Levermore  \& Liu \cite{chen1,chen2} continued this investigation in the context of weak solutions.  In 1995, Jin and Xin \cite{JinXin1} introduced relaxation schemes for systems of conservation laws in arbitrary space dimensions that have been widely employed in numerical analysis and scientific computing, e.g., cf. \cite{Godlewski1,filbet1,Jinc1}. We will be concentrating on their relaxation procedure here. 


We investigate the following diffusely scaled version of the Jin-Xin system (cf. \cite{Jin0,JinXin1}):
\begin{equation} \label{JXSys1}
\left\{
\begin{aligned}
&\frac{\partial} {\partial t}u+\sum_{i=1}^{d} \frac{\partial} {\partial x_{i}} v_{i}=0, \\
&\var^2\frac{\partial} {\partial t} v_{i}+A_{i}\frac{\partial} {\partial x_{i}} u=-\big(v_{i}-f_{i}(u)\big),\quad\quad i=1,2,..., d,
\end{aligned}
\right.
\end{equation}
where $t>0$ and $x \in \R^d$ denote the time and space variables, $d\geq1$ is the dimension and  $\var>0$ stands for the relaxation parameter. The unknowns are $u=u(t,x)\in \R^n$ and $\bm{v}=(v_1, v_2,..., v_d )$ with $v_i =v_i(t,x)\in \R^n$ and $n\geq1$.  The nonlinear term $f(u)=(f_1(u),f_2(u),..., f_d(u))$ with
$f_i(u): \R^n\rightarrow \R^n$ depends on $u$ smoothly and satisfies $f_{i}(0)=\partial_{u_{k}} f_{i}(0)=0$ with $k=1,2,...,n$. The constant coefficient matrices $A_i$ ($i=1,2,...,d$) are taken as $A_i=a_iI_n$ with $a_i>0$ and $I_n$ the unit matrix.


\medbreak
As $\var\rightarrow 0$, the dynamics of System \eqref{JXSys1} are formally governed by the viscous conservation law
\begin{equation}\label{Thm2uheat}
\frac{\partial} {\partial t}u^{*}+\sum_{i=1}^{d} \frac{\partial} {\partial x_{i}}f_{i}(u^{*})=\sum_{i=1}^{d} \frac{\partial} {\partial x_{i}} \Big(A_{i}\frac{\partial} {\partial x_{i}} u^{*}\Big)
\end{equation}
and Darcy's law
\begin{equation}\label{Thm2vdarcy}
{\bf{v^*}}=(v^*_1,v^*_2,...,v^*_d)\quad\text{with}\quad  v_i^*=-A_i\frac{\partial}{\partial x_i}u^*+f_i(u^*),\quad \quad i=1,2,...,d.
\end{equation}
An explicit example of \eqref{Thm2uheat} is the two-dimensional Burgers equations
\begin{align}
&\frac{\partial} {\partial t} u+\frac{\partial} {\partial x_{1}} (u_{1} u)+\frac{\partial} {\partial x_{2}} (u_{2} u)-\Delta u=0,\label{burgers}
\end{align}
where $u=(u_1,u_2)\in \R^2$.
System \eqref{burgers} can also be interpreted as a pressureless incompressible Navier-Stokes model used to model unsaturated flows \cite{CKW}. The Jin-Xin approximation of \eqref{burgers} reads
\begin{equation}\label{approx1}
\left\{
\begin{aligned}
&\frac{\partial} {\partial t} u+\frac{\partial} {\partial x_{1}} v_{1}+\frac{\partial} {\partial x_{2}} v_{2}=0,\\
&\var^2\frac{\partial} {\partial t} v_{1}+\frac{\partial} {\partial x_{1}} u=-(v_{1}-u_{1} u),\\
&\var^2\frac{\partial} {\partial t} v_{2}+\frac{\partial} {\partial x_{2}} u=-(v_{2}-u_{2} u).
\end{aligned}
\right.
\end{equation}
The key point of the approximation \eqref{approx1} is that it modifies the nature of the system under study. Indeed, \eqref{burgers} is parabolic while \eqref{approx1} is purely hyperbolic. Therefore, if the approximation is valid in a sufficiently strong sense, this procedure justifies the use of hyperbolic methods to study parabolic equations.
For instance, the reader may refer to \cite{JinXin1} for details concerning the relevance of the Jin-Xin approximation for numerical analysis. Here, our goal is to extend the validity of the diffusive Jin-Xin approximation in the context of global-in-time strong solutions being perturbations of small initial data.
\medbreak
There have been a lot of studies devoted to the mathematical analysis for the Jin-Xin relaxation system in the one-dimensional setting. Chern \cite{Chern1} investigated the long-time effect of relaxation and proved that the corresponding solution tends to a diffusion wave asymptotically in time in terms of the Chapman-Enskog expansion. Natalini \cite{Natalni1} and Bianchini \cite{Bianchini0} justified the relaxation convergence of the Jin-Xin approximation of hyperbolic conservation laws. Jin and Liu \cite{Jin0} studied the relaxation limit of the Jin-Xin system under the diffusive scaling for initial data around traveling waves. Bouchut, Guarguaglini \& Natalini \cite{Bouchut1} considered the diffusive relaxation of BGK (Bhatnagar-Gross-Krook) type approximations for the Jin-Xin system. Mei and Rubino \cite{Mei1} got the time-convergence rates of solutions to the initial boundary value problem for the Jin-Xin system toward traveling waves on the half line. Orive and Zuazua \cite{Orive1} reformulated the system in a damped wave equation and derived algebraic time-decay rates of solutions on the real line. Huang, Pan \& Wang \cite{Huang2008} obtained the nonlinear stability of contact waves for the Jin-Xin model with the decay rate $(1+t)^{-\frac{1}{4}}$. Bianchini \cite{Bianchini1} derived the sharp time-decay estimates of solutions to the Jin-Xin system, in the case $f'(0)\ne0$, which are uniform with respect to $\varepsilon$ and provided the convergence to a nonlinear heat equation both asymptotically in time and in the relaxation limit. 

For the high-dimensional case, there are fewer results. Crin-Barat and Shou \cite{BaratShouJDE2023} justified the uniform convergence of the multi-dimensional Jin-Xin system \eqref{JXSys1} toward viscous conservation laws \eqref{Thm2uheat} as $\var\rightarrow0$ with an explicit rate in $L^2$-type Besov spaces.



\vspace{2mm}

To the best of our knowledge, the long-time asymptotic behavior of the multi-dimensional Jin-Xin system \eqref{JXDZ} has not been established in previous references. Moreover, since the Cauchy problem of the limiting viscous conservation law \eqref{Thm2uheat} is globally well-posed in general Besov spaces of $L^{p}$-type, it is relevant to study the relaxation limit of the Jin-Xin system in a framework adapted to the limit system. These are the two problems that we tackle in the present paper. 

\subsection{Main results}
Our main results read as follows. First, we state a global well-posedness result for viscous conservation laws in a $L^p$ framework. For the definition of Besov spaces, see Section \ref{besovdefinion}. 

\begin{Thm}\label{Thm0}
Let $d\geq1$, $n\geq1$ and $1\leq p\leq\infty$. There exists a generic constant $\eta_{0}>0$ such that if  $u_{0}^{*}$ satisfies $u_0^*\in \dot{B}^{\frac{d}{p}-1}_{p,1}\cap\dot{B}^{\frac{d}{p}}_{p,1} $ and
\begin{equation}\label{smalllimit}
\|u_{0}^*\|_{\dot{B}^{\frac{d}{p}-1}_{p,1}\cap \dot{B}^{\frac{d}{p}}_{p,1} }\leq \eta_{0},
\end{equation}
then, System \eqref{Thm2uheat} has a unique global solution $u^{*}\in \mathcal{C}(\mathbb{R}^{+};\dot{B}^{\frac{d}{p}-1}_{p,1}\cap \dot{B}^{\frac{d}{p}}_{p,1})$.
Furthermore, it holds that
\begin{equation}\label{limitur}
\begin{aligned}
&\|u^{*}\|_{\tL^\infty(\R^+;\dB_{p,1}^{\frac{d}{p}-1}\cap \dot{B}^{\frac{d}{p}}_{p,1})}+\|u^{*}\|_{L^1(\mathbb{R}^+;\dot{B}^{\frac{d}{p}+1}_{p,1}\cap \dot{B}^{\frac{d}{p}+2}_{p,1})}+\|\bv^{*}\|_{L^1(\mathbb{R}^+;\dot{B}^{\frac{d}{p}}_{p,1}\cap \dot{B}^{\frac{d}{p}+1}_{p,1})}
\leq C\|u_{0}^{*}\|_{\dot{B}^{\frac{d}{p}-1}_{p,1}\cap \dot{B}^{\frac{d}{p}}_{p,1}},
\end{aligned}
\end{equation}
where ${\bv}^*=(v_1^*,v_2^*,...,v_d^*)$ is given by Darcy's law \eqref{Thm2vdarcy}, and $C>0$ is a constant independent of time.
\end{Thm}



\vspace{2mm}

Next, we establish the global well-posedness of solutions for System \eqref{JXSys1} in a hybrid $L^p$-$L^2$ framework.

\begin{Thm}\label{Thm1}
Assume $\var>0$,  $d\geq1$ and $n\geq1$. Let $p$ satisfy
\begin{equation}\label{p}
\left\{
    \begin{aligned}
    &1\leq p\leq 4,\quad\quad\quad\quad\quad~~ \text{if}\quad  d=1,2,\\
   &\frac{6}{5}\leq p\leq 4,\quad\quad\quad\quad\quad~\text{if}\quad d=3,\\
    &\frac{2d}{d+2}\leq p\leq \frac{2d}{d-2},\quad~~\text{if}\quad d\geq4,
    \end{aligned}
    \right.
\end{equation}
and set the threshold $J_{\var}$ between low and high frequencies{\rm:}
\begin{equation}\label{Jvar}
J_\var=-[\log \var]+k_0
\end{equation}
with some generic integer $k_{0}$. There exists a constant $\eta_1>0$ independent of $\var$ such that for any initial data $(u_{0},\bv_{0})$ satisfying $u_0^{\ell}\in \dot{B}^{\frac{d}{p}-1}_{p,1},$ $\bv_0^{\ell}\in \dot{B}^{\frac{d}{p}}_{p,1}$,
$(u_0^{h},\bv_0^{h})\in \dot{B}^{\frac{d}{2}}_{2,1}$
and
\begin{equation}\label{small}
\cX_{p,0}\triangleq\|u_{0}\|_{\dot{B}^{\frac{d}{p}-1}_{p,1}\cap\dot{B}^{\frac{d}{p}}_{p,1}}^{\ell}
+\var^2\|\bv_{0}\|_{\dot{B}^{\frac{d}{p}}_{p,1}\cap\dot{B}^{\frac{d}{p}+1}_{p,1}}^{\ell}
+(1+\var) \|u_{0}\|_{\dot{B}^{\frac{d}{2}}_{2,1}}^{h}+ \var(1+\var)\|\bv_0\|_{\dot{B}^{\frac{d}{2}}_{2,1}}^{h}
\leq \eta_1,
\end{equation}
 System \eqref{JXSys1} admits a unique global strong solution $(u,\bv)$ satisfying 
$u^{\ell}\in\cC(\R^+;\dB_{p,1}^{\frac{d}{p}-1}) ,\ 
\bv^{\ell} \in\cC(\R^+;\dB_{p,1}^{\frac{d}{p}}),\  
(u^{h},\bv^{h})
\in \cC(\R^+;\dB_{2,1}^{\frac{d}{2}})$ and
\begin{equation}\label{ThmEst1}
\begin{aligned}
&\|u\|_{\tL_t^\infty(\dot B_{p,1}^{\frac{d}{p}-1}\cap\dB_{p,1}^{\frac{d}{p}})}^{\ell}
+\|u\|_{L_t^1(\dot B_{p,1}^{\frac{d}{p}+1}\cap\dB_{p,1}^{\frac{d}{p}+2})}^{\ell}+(1+\var)\|u\|_{\tL_t^\infty(\dot B_{2,1}^{\frac{d}{2}})}^{h}
+(\frac{1}{\var}+\frac{1}{\var^2})\|u\|_{L_t^1(\dot B_{2,1}^{\frac{d}{2}})}^{h}\\
&\qquad+\var^2\|\bv\|_{\tL_t^\infty(\dot B_{p,1}^{\frac{d}{p}}\cap\dB_{p,1}^{\frac{d}{p}+1})}^{\ell}
+\|\bv\|_{L_t^1(\dot B_{p,1}^{\frac{d}{p}}\cap\dB_{p,1}^{\frac{d}{p}+1})}^{\ell}+(\var+\var^2)\|\bv\|_{\tL_t^\infty(\dot B_{2,1}^{\frac{d}{2}})}^{h}
+(1+\frac{1}{\var})\|\bv\|_{L^1_t(\dot B_{2,1}^{\frac{d}{2}})}^{h}
\\
&\quad\leq C \cX_{p,0},
\end{aligned}
\end{equation}
where $C>0$ is a constant independent of time and $\var$.
\end{Thm}


Thanks to the uniform bounds obtained in Theorem \ref{Thm1}, we justify the strong relaxation limit of System \eqref{JXSys1} toward System \eqref{Thm2uheat} with an explicit rate of convergence in the $L^{p}$ framework.
\begin{Thm}\label{Thm2}
For $d\geq1$ and $\var>0$, let $(u,\bv)$ be the solution of System \eqref{JXSys1} associated with the initial data $(u_{0},\bv_{0})$ given by Theorem \ref{Thm1}, $u^*$ be the solution of System \eqref{Thm2uheat} associated with the initial data $u_{0}^{*}$ given by Theorem \ref{Thm0}, and ${\bf{v}}^*$ be given by Darcy's law \eqref{Thm2vdarcy}. Further assume that
\begin{equation}\label{p2}
\left\{
    \begin{aligned}
    & p=2,\quad\quad\quad\quad\quad~ \text{if}\quad  d=1,\\
   &2\leq p\leq 4,\quad\quad\quad~~\text{if}\quad d=2,3,\\
    &2\leq p\leq \frac{2d}{d-2},\quad~~\text{if}\quad d\geq4,
    \end{aligned}
    \right.
\end{equation}
and
\begin{align}\label{error}
\var\|\bv_0^{\ell}\|_{\dot B_{p,1}^{\frac{d}{p}}}~\text{is uniformly bounded}.
\end{align}
Then, 
the following convergence estimates hold{\rm:}
\begin{equation}\label{rate}
\begin{aligned}
&\|u-u^*\|_{\tL^{\infty}(\R^+;\dot{B}^{\frac{d}{p}-1}_{p,1})}
+\|\bv-\bv^{*}\|_{L^1(\R^+;\dot{B}^{\frac{d}{p}}_{p,1})}\leq C\|u_{0}-u_{0}^{*}\|_{\dot{B}^{\frac{d}{p}-1}_{p,1}}+C\var,
\end{aligned}
\end{equation}
where $C>0$ is a constant independent of time and $\varepsilon$.
\end{Thm}

Finally, we exhibit sharp decay estimates for the solution of \eqref{JXSys1} and for its difference with the solution of \eqref{Thm2uheat} when $u$ and $u^{*}$ are associated with the same initial data.

\begin{Thm}\label{Thm3}
For $d\geq1$ and $0<\var<1$, let $(u,\bv)$ be the global solution to \eqref{JXSys1} subject to the initial data $(u_{0},\bv_{0})$ given by Theorem \ref{Thm1}. In addition to \eqref{small}, further assume that $p$ satisfies \eqref{p2} and
\begin{equation}
\begin{aligned}
&\|(u_{0}^{\ell},\var  \bv_{0}^{\ell})\|_{\dot{B}^{\sigma_{1}}_{p,\infty}}~\text{is uniformly bounded with}~ -\frac{d}{p}\leq \sigma_{1}\leq \frac{d}{p}-1.\label{a2}
\end{aligned}
\end{equation}
Then,  for all $t>0$, it holds that
\begin{equation}\label{decaysolution}
\begin{aligned}
&\|(u,\var \bv)(t)\|_{\dot{B}^{\sigma}_{p,1}}\leq C  (1+t)^{-\frac{1}{2}(\sigma-\sigma_{1})},\quad\quad \sigma_{1}<\sigma\leq \frac{d}{p},
\end{aligned}
\end{equation}
where $C>0$ is a constant independent of time and $\var$.
 
Moreover, let $u^{*}$ be the global solution of System \eqref{Thm2uheat} supplemented with the initial data $u_{0}$ given by Theorem \ref{Thm0}. Then,  for all $t>0$, the difference $u-u^{*}$ satisfies
\begin{equation}\label{decayerror}
\|(u-u^{*})(t)\|_{\dot{B}^{\sigma}_{p,1}}\leq C 
\var  (1+t)^{-\frac{1}{2}(\sigma-\sigma_1)-\frac{1}{2}},\quad \sigma_{1}<\sigma\leq \frac{d}{p}-1.
\end{equation}

\end{Thm}

\vspace{2mm}

\subsection{Comments on the main results}
Some remarks on Theorems \ref{Thm1}-\ref{Thm3} are in order:
\begin{itemize}
\item [1.] The low-frequency regularity properties \eqref{ThmEst1} of $u$  correspond to the properties \eqref{limitur} verified by the solution $u^{*}$ of the limit system \eqref{Thm2uheat}. As the relaxation parameter $\varepsilon$ goes to zero, the low-frequency regime $|\xi|\leq 2^{J_{\var}}\sim \var^{-1}$ will cover the whole frequency regime and the high-frequency regime disappears. This highlights, qualitatively, the diffusive relaxation process from the Jin-Xin system to viscous conservation laws.

\item[2.] In Theorem \ref{Thm1}, we derive uniform regularity estimates for the solution to System \eqref{JXDZ} in $L^p$-$L^2$ hybrid Besov spaces, i.e., the low frequencies in $L^p$-based spaces and the high-frequency ones in $L^2$-based spaces. Note that in \eqref{small}, the norm of $\bv_{0}$ and the high-frequency norm of $u_{0}$ can be arbitrarily large as long as $\var$ is suitably small.  Due to the dispersive structure in the high-frequency regime, the well-posedness of the hyperbolic system \eqref{JXDZ} cannot be entirely justified in $L^{p}$-based spaces for $p \neq 2$ (see Brenner \cite{brenner1}).

\item[3.] Compared with the work {\rm\cite{BaratShouJDE2023}}, Theorem \ref{Thm1} exhibits not only a more general $L^{p}$-type functional setting but also lower regularity assumptions on $\bv$. 
Moreover,  the restriction $0<\var<1$ required in {\rm\cite{BaratShouJDE2023}} can be relaxed to the full range $\var>0$ so as to describe the so-called overdamping phenomenon for the Jin-Xin system {\rm(}see Figure \ref{fig:overdamping}{\rm)}.




\item[4.] Theorem \ref{Thm2} provides a rigorous justification of the strong relaxation limit from System \eqref{JXSys1} to System \eqref{Thm2uheat} for general ill-prepared data. We recall that initial data are called well-prepared if the limit $\bv_0^*=\lim\limits_{\var\rightarrow0}\bv_{0}^{\var}$ is well-defined and the compatibility condition holds, i.e., $(u_0^*,\bv_0^*)$ satisfies Darcy's law \eqref{Thm2vdarcy}. Our results hold for $\bv^{\ell}_0=\mathcal{O}(\var^{-2})$ without the compatibility condition.



\item[5.] By virtue of Theorem \ref{Thm3}, the solution of System \eqref{Thm2uheat} can be viewed as both the relaxation limit and the time-asymptotic profile of the solutions of System \eqref{JXSys1}. The time-decay rates in \eqref{decaysolution} coincide with the optimal rates of the heat equation. In the case $-d/p<\sigma_{1}<0$, due to the embeddings $\dot{B}^{0}_{p,1}\hookrightarrow L^{p}(\mathbb{R}^{d})$ and $L^{q}(\mathbb{R}^{d})\hookrightarrow \dot{B}^{\sigma_{1}}_{p,\infty}$ with  $q=\frac{dp}{d-p\sigma_{1}}\in [p/2,p)$, Theorem \ref{Thm3} implies that under the stronger condition that $(u_{0}^{\ell},\var  \bv_{0}^{\ell})$ is uniformly bounded in $L^{q}(\mathbb{R}^{d})$, the solution $(u,\bv)$ satisfies
\begin{equation}
\begin{aligned}
&\|(u,\var \bv)(t)\|_{L^{p}(\mathbb{R}^{d})}\lesssim (1+t)^{-\frac{d}{2}(\frac{1}{q}-\frac{1}{p})}.\nonumber
\end{aligned}
\end{equation}
Moreover, for $2\leq p\leq d$ such that $\frac{d}{p}-1\geq0$, the difference  $u-u^{*}$ verifies
\begin{equation}
\begin{aligned}
\|(u-u^{*})(t)\|_{L^{p}(\mathbb{R}^{d})}\lesssim \var (1+t)^{-\frac{d}{2}(\frac{1}{q}-\frac{1}{p})-\frac{1}{2}}.\nonumber
\end{aligned}
\end{equation}
Additionally, in our computations (see \eqref{DecaytW3}), the high frequencies of the solutions undergo faster time-decay rates than the rates obtained for the full solution in \eqref{decayerror} due to the damping behavior in this regime.

\item[6.] In Theorems \ref{Thm2}-\ref{Thm3}, we require $2\leq p\leq 2d$ due to the embedding $\dot{B}^{\frac{d}{2}}_{2,1}\hookrightarrow \dot{B}^{\frac{d}{p}}_{p,1}$ and the technical limitation when using product laws.
In the case $p<2$, using the embedding $\dot{B}^{\frac{d}{p}}_{p,1}\hookrightarrow \dot{B}^{\frac{d}{2}}_{2,1}$, one can establish similar estimates as in \eqref{error}-\eqref{decayerror} for $L^2$-type norms.


\item[7.] Different from the Green function method used for instance in {\rm\cite{Bianchini1}} concerning the 1-d case, our proof of Theorem \ref{Thm3} relies on a pure energy argument with explicit dependence on the parameter $\var$ and may be of interest in the mathematical analysis of other relaxation problems. 

\end{itemize}

\vspace{2mm}


\subsection{Strategies of proofs}
\subsubsection{Spectral behavior of the solution}
 In order to understand the behavior of the solution to \eqref{JXSys1} with respect to $\var$, we analyse the  eigenvalues of the associated linearized system of \eqref{JXSys1} as follows.  Taking the Fourier transform of the linearisation of \eqref{JXDZ}  with respect to $x$, we obtain
 \begin{equation}\nonumber
\begin{aligned}
& \partial_{t}
\left(\begin{matrix}
   \hat{u}  \\
   \var \hat{v}_{1}  \\
   \var \hat{v}_{2}\\
   \vdots\\
   \var \hat{v}_{d}
  \end{matrix}\right)
  =\widehat{\mathbb{A}}(\xi)\left(\begin{matrix}
   \hat{u}  \\
   \var \hat{v}_{1}  \\
   \var \hat{v}_{2}\\
   \vdots\\
   \var \hat{v}_{d}
     \end{matrix}\right),\quad\quad
\widehat{\mathbb{A}}(\xi)\triangleq
\left(\begin{matrix}
0                              &   -{\rm{i}}\frac{1}{\var}\xi_{1}    & -{\rm{i}}\frac{1}{\var}\xi_{2}   &\cdots    &    -{\rm{i}}\frac{1}{\var}\xi_{d}&\\
-{\rm{i}} \frac{1}{\var}a_{1} \xi_{1} & -\frac{1}{\var^2}              &  0                        &\cdots    &     0                     &\\
-{\rm{i}} \frac{1}{\var}a_{2} \xi_{2} & 0                            & -\frac{1}{\var^2}           &\cdots    &     0                     &\\
\vdots                         & \vdots                       & \vdots                    &\ddots    &     \vdots                &\\
-{\rm{i}} \frac{1}{\var}a_{d} \xi_{d} & 0                            & 0                         &\cdots    &     -\frac{1}{\var^2}       &\\
\end{matrix}\right).
\end{aligned}
\end{equation}
Then,  we compute the eigenvalues of the matrix $\widehat{\mathbb{A}}(\xi)$ from the determinant
$$
{\rm{det}}(\widehat{\mathbb{A}}(\xi)-\lambda I_{d})=(\lambda+\frac{1}{\var^2})^{d-1} (\lambda^2+\frac{1}{\var^2}\lambda+\frac{1}{\var^2}\sum_{i=1}^{d} a_{i} |\xi_{i}|^2 )=0.
$$
Solving this equation, we obtain
\begin{equation}\nonumber
\left\{
\begin{aligned}
&\lambda_{i}=-\frac{1}{\var^2},\quad\quad\quad i=1,2,...,d-1,\\
&\lambda_{d}=-\frac{1}{2\var^2}+\frac{1}{2\var}\sqrt{\frac{1}{\var^2}-4\sum_{i=1}^{d} a_{i} |\xi_{i}|^2 },\\
&\lambda_{d+1}=-\frac{1}{2\var^2}-\frac{1}{2\var}\sqrt{\frac{1}{\var^2}-4\sum_{i=1}^{d} a_{i} |\xi_{i}|^2 }.
\end{aligned}
\right.
\end{equation}
The eigenvalues have the following properties:
\begin{itemize}
\item
In the low-frequency regime $|\xi|\lesssim\var^{-1}$, all the eigenvalues are real, and we have $\lambda_{d}\sim -|\xi|^2$ and $\lambda_{d+1}\sim -\var^{-2}$. This implies that the parabolic effect and the damping effect coexist.

\item
In the high-frequency regime $|\xi|\gtrsim\var^{-1}$, the complex conjugated eigenvalues $\lambda_{i}$ ($i=d,d+1$) have $-\frac{1}{2}\var^{-2}$ as real parts.
\end{itemize}
The above spectral analysis reveals that it is suitable to split the frequencies into a low-frequency regime $|\xi|\lesssim \var^{-1}$ and a high-frequency regime $|\xi|\gtrsim \var^{-1}$ so as to study \eqref{JXSys1}. On the other hand, since no dispersive effects are present in the low-frequency regime $|\xi|\lesssim \var^{-1}$, as all the eigenvalues in low frequencies are purely real, an $L^{p}$-based functional framework is feasible in this regime.

 Such a choice of a threshold $J_{\var}$ satisfying $2^{J_{\var}}\sim \var^{-1}$ allows us to tackle the so-called overdamping phenomenon \cite{SlideZuazua} that is usually a major obstacle to studying the relaxation limit. The overdamping phenomenon refers to the fact that \textit{as the friction coefficient $\var^{-1}$ gets larger, the decay rates do not necessarily increase and achieve} the maximum at a specific threshold (cf. Figure 1 below). Here, decomposing the frequency space with a suitably chosen threshold enables us to perfectly capture the $\varepsilon$-dependency of solutions.

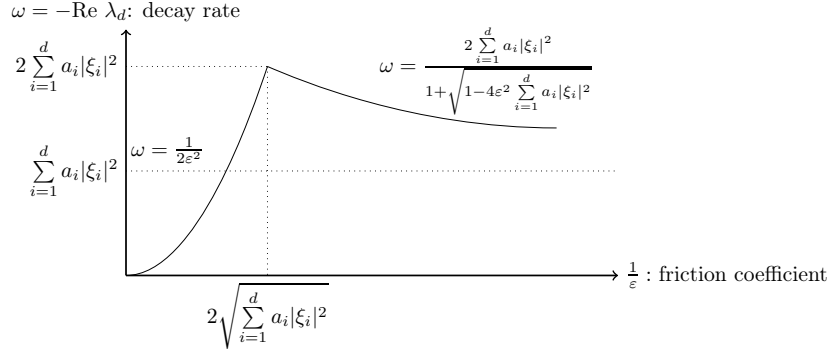
\begin{figure}[!ht]
	\centering	
 \resizebox{0.69\columnwidth}{!}{\begin{tikzpicture}
\node at (0, 3.4) [left] {${\footnotesize2\sum\limits_{i=1}^{d} a_{i} |\xi_{i}|^2}$ };
\node at (0, 1.7) [left] {${\footnotesize \sum\limits_{i=1}^{d} a_{i} |\xi_{i}|^2}$ };

\draw[->,thick] (0,0) -- (8,0) node[right] {$\frac{1}{\var}:$ friction coefficient};
\draw[->,thick] (0,0) -- (0,4) node[above] {$\omega=-{\rm Re}\,\lambda_{d}$: decay rate};

\draw[dotted] (0,1.7)--(8, 1.7);
\draw (0,0) parabola (2.3,3.4);
\draw (7,2.4)  parabola bend (7,2.4) (2.3,3.4);
\draw[dotted] (0,3.4)--(2.3, 3.4); 
\draw[dotted] (2.3,3.4)--(2.3, 0) node[below] {${\footnotesize 2\sqrt{\sum\limits_{i=1}^{d} a_{i} |\xi_{i}|^2}}$ };
\node[above left] at (1.4,1.7){$\omega=\frac{1}{2\var^2}$}; 
  \node[above right] at (4, 2.5) {$\omega=\frac{2\sum\limits_{i=1}^{d} a_{i} |\xi_{i}|^2}{1+\sqrt{1-4\varepsilon^2\sum\limits_{i=1}^{d} a_{i} |\xi_{i}|^2}}$};
\end{tikzpicture}}
	\caption{The overdamping phenomenon for System \eqref{JXSys1}.}
	\label{fig:overdamping} 
\end{figure}

\subsubsection{Effective unknowns and relaxation limit}
To achieve our results, we first introduce an effective unknown
\begin{equation}\label{effective}
\begin{aligned}
\bz=(z_1,z_2,...,z_{d})\quad\text{with}\quad z_{i}\triangleq A_{i}\frac{\partial} {\partial x_{i}} u+v_{i},\quad\quad i=1,2,...,d,
\end{aligned}
\end{equation}
which allows us to capture the sharp dissipative effects observed in the spectral behavior for low frequencies. In fact, in order to decouple System \eqref{JXSys1}, we rewrite $\eqref{JXSys1}$ in terms of $(u,\bz)$ as
\begin{equation} \label{JXDZ}
\left\{
\begin{aligned}
&\frac{\partial}{\partial t}u-\sum_{i=1}^{d}\frac{\partial}{\partial x_{i}}\Big(A_{i} \frac{\partial}{\partial x_{i}} u\Big)=-\sum_{i=1}^{d}\frac{\partial}{\partial x_{i}} z_{i},\\
&\frac{\partial}{\partial t}z_{k}+\dfrac{1}{\varepsilon^2}z_{k}= A_{k} \frac{\partial}{\partial x_{k}}\Big( \sum_{i=1}^{d}\frac{\partial}{\partial x_{i}}\Big(A_{i}  \frac{\partial}{\partial x_{i}} u \Big) \Big)-A_{k} \frac{\partial}{\partial x_{k}} \sum_{i=1}^{d}\frac{\partial}{\partial x_{i}} z_{i}+ \frac{1}{\var^2} f_{k}(u) ,\quad k=1,2,...,d.
 \end{aligned}
 \right.
\end{equation}
Note that the high-order linear terms on the right-hand side (r.h.s.) of \eqref{JXDZ} can be absorbed if the threshold $J_{\var}$ takes the form \eqref{Jvar} for a suitably small $k_{0}$ independent of $\var$. Such an unknown \eqref{effective} is different from the previous work \cite{BaratShouJDE2023} and allows us to improve the regularity assumptions by avoiding difficult terms. Meanwhile, the high-frequency analysis is based on the construction of a Lyapunov functional as in \cite{BaratShouJDE2023} but with additional parameter weights. To overcome the major difficulty caused by the quadratic nonlinear term $f(u)$ in our functional setting, some new composition estimates are established in hybrid Besov spaces with explicit dependence on the threshold (refer to Lemmas \ref{NewSmoothlow} and \ref{NewSmoothhigh}).

\vspace{2mm}

In order to obtain the convergence estimate \eqref{rate}, we introduce another effective unknown
\begin{align}
\bZ=(Z_1, Z_2, \cdots, Z_d)\quad\text{with}\quad    Z_{i}\triangleq A_{i}\frac{\partial} {\partial x_{i}} u+v_{i}-f_{i}(u).\label{effective2}
\end{align}
Then,  we are able to rewrite the equation $\eqref{JXSys1}_{1}$ of $u$ as
\begin{equation} \label{JXDZ2}
\begin{aligned}
&\frac{\partial}{\partial t}u-\sum_{i=1}^{d}\frac{\partial}{\partial x_{i}}(A_{i} \frac{\partial}{\partial x_{i}} u)
=-\sum_{i=1}^{d}\frac{\partial}{\partial x_{i}} Z_{i}
+\sum_{i=1}^{d}\frac{\partial}{\partial x_{i}}  f_{i}(u).
 \end{aligned}
\end{equation}
Note that \eqref{JXDZ} can be viewed as the structure of the viscous conservation law \eqref{Thm2uheat} coupled with the higher-order damped mode $-\sum\limits_{i=1}^{d}\frac{\partial}{\partial x_{i}} Z_{i}$. Indeed, from \eqref{Thm2uheat} and \eqref{JXDZ2}, the difference $\delta u\triangleq u-u^{*}$ satisfies
\begin{equation}\label{ErrorSys}
\begin{aligned}
&\partial_t\delta u-\sum_{i=1}^{d}\frac{\partial}{\partial x_{i}}(A_{i}\frac{\partial}{\partial x_{i}} \delta u)=-\sum_{i=1}^{d}\frac{\partial}{\partial x_{i}}Z_{i}
-\sum_{i=1}^{d}\frac{\partial}{\partial x_{i}}\big(f_{i}(u)-f_{i}(u^*) \big).
\end{aligned}
\end{equation}
To analyze \eqref{ErrorSys} and derive $\mathcal{O}(\varepsilon)$  bounds in low frequencies, we establish the following key decay-in-$\var$ estimate:
\begin{align}
&\frac{1}{\var}\|\bZ^{\ell}\|_{L^1(\mathbb{R}_{+};\dot{B}^{\frac{d}{p}}_{p,1})}\lesssim \var\|\bv_0^{\ell}\|_{\dot B_{p,1}^{\frac{d}{p}}}+\mathcal{X}_{p,0},
\end{align}
which comes essentially from the damped structure of $\bZ$ (refer to Proposition \ref{dgeq2PropRelax}). Meanwhile, the convergence rate for high frequencies follows directly from the bounds in \eqref{ThmEst1} combined with Bernstein-type estimates.

\subsubsection{Large-time asymptotics}

Finally, we explain the main ideas concerning the time asymptotics obtained in Theorem \ref{Thm2}. Since the solution of System \eqref{JXSys1} is purely damped in the high-frequency regime and that the component $Z$ is also damped in the low-frequency regime, the slow variable that will dictate the decay is $u^{\ell}$. To that matter, for the asymptotic main part of (1.21), i.e., the linear system of (1.21) neglecting the fast decay part $-\sum\limits_{i=1}^{d}\frac{\partial}{\partial x_{i}} Z_{i}$, multiplying it by $t^{\alpha}$ with any given $\alpha>1$ and using maximal regularity estimates and real interpolation, we have 
\begin{equation}\label{timme}
\begin{aligned}
\|\tau^\alpha u^{\ell}\|_{\tL^{\infty}_{t}(\dot{B}^{\frac{d}{p}}_{p,1})}+\|\tau^\alpha u^{\ell}\|_{\tL^1_{t}(\dot{B}^{\frac{d}{p}+2}_{p,1})}&\lesssim  \int_0^t  \tau^{\alpha-1}\|u^{\ell}\|_{ \dot{B}^{\frac{d}{p}}_{p,1}}d\tau+\ldots\\
&\lesssim t^{\alpha-\frac{1}{2}(\frac{d}{p}-\sigma_{1})}\|u^{\ell}\|_{\tL^{\infty}_{t}(\dot{B}^{\sigma_{1}}_{p,\infty})}+o(1) \|\tau^{\alpha}u^{\ell}\|_{\tL^1_{t}(\dot{B}^{\frac{d}{p}+2}_{p,1})}+\ldots.
\end{aligned}
\end{equation}
Under the assumption \eqref{a2}, the time-weighted estimate \eqref{timme} requires us to establish the low-frequency evolution of the $\dot{B}^{\sigma_{1}}_{p,1}$-norm: 
\begin{align}
&\|u^{\ell}\|_{\tL^{\infty}(\mathbb{R}_{+};\dot{B}^{\sigma_{1}}_{p,\infty})}+\frac{1}{\var}\|\bZ^{\ell}\|_{\tL^1(\mathbb{R}_{+};\dot{B}^{\sigma_{1}}_{p,\infty})}\lesssim \|(u_{0}^{\ell},\var  \bv_{0}^{\ell})\|_{\dot{B}^{\sigma_{1}}_{p,\infty}}+\mathcal{X}_{p,0}.\label{evolution}
\end{align}
By \eqref{timme} and \eqref{evolution}, we establish the uniform large-time behavior \eqref{decaysolution}. This approach is based on the previous works \cite{XinXu,BaratShou2,lhl1,BaratDanchin2} but requires more  elaborate weighted estimates with respect to the relaxation parameter $\var$ (see Proposition \ref{propdecay}). Moreover, we observe that when $\delta u|_{t=0}=0$, the spatial derivative of $Z_{i}$ in \eqref{ErrorSys} and the $\dot{B}^{\sigma_{1}}_{p,1}$-norm convergence of $\bZ$ in \eqref{evolution} imply the additional estimate of the difference $\delta u$:
\begin{align}
&\frac{1}{\var}\|\delta u^{\ell}\|_{\tL^{\infty}(\mathbb{R}_{+};\dot{B}^{\sigma_{1}-1}_{p,\infty})}\lesssim \frac{1}{\var}\|\bZ^{\ell}\|_{\tL^1(\mathbb{R}_{+};\dot{B}^{\sigma_{1}}_{p,\infty})}\lesssim \|(u_{0}^{\ell},\var  \bv_{0}^{\ell})\|_{\dot{B}^{\sigma_{1}}_{p,\infty}}+\mathcal{X}_{p,0}.\label{evolution1}
\end{align}
Compared with the $\dot{B}^{\sigma_{1}}_{p,\infty}$-evolution of $u^{\ell}$ in \eqref{evolution}, the key ingredient \eqref{evolution1} not only provides lower-order regularity that leads to faster decay rates but also yields an $\mathcal{O}(\varepsilon)$ bound in time-decay estimates. Performing similar time-weighted energy estimates on the difference equation \eqref{ErrorSys} and taking advantage of \eqref{evolution1} enables us to establish the enhanced decay estimates \eqref{decayerror}.

\vspace{2mm}
\subsection{Outline of the paper}
The rest of this paper is organized as follows. In Section \ref{besovdefinion}, we briefly recall the Littlewood-Paley decomposition, Besov spaces and Chemin-Lerner spaces. In Section \ref{section3}, we establish uniform {\it{a priori}} estimates and prove Theorem \ref{Thm1}. The rigorous justification of the relaxation limit from System \eqref{JXSys1} to System \eqref{Thm2uheat} is performed in Section \ref{section4}. Section \ref{section5} is devoted to the proof of the sharp decay estimates in Theorem \ref{Thm2}. Some technical lemmas and the proof of Theorem \ref{Thm0} are relegated to the appendix.

\section{Preliminaries}\label{besovdefinion}

We list some notations that are used frequently throughout the paper.

\noindent
\textbf{Notations.} For simplicity, $C$ denotes some positive constant that is independent of $\var$ and time. $A\lesssim B$ ($A\gtrsim B$) means that $A\leq C B$ ($A\geq C B$),  while $A\sim B$ means that both $A\lesssim  B$
and $A\gtrsim B$.  For a Banach space $X$, $p\in[1, \infty]$ and
$T>0$, the notation $L^p(0, T; X)$ or $L^p_T(X)$ denotes the set of measurable functions $f: [0, T]\to X$ with $t\mapsto\|f(t)\|_X$ in $L^p(0, T)$, endowed with the norm $\|\cdot\|_{L^p_{T}(X)} \triangleq\|\|\cdot\|_X\|_{L^p(0, T)}$, and $\mathcal{C}([0,T];X)$ denotes the set of continuous functions $f: [0, T]\to X$. Let $\mathcal{F}(f)=\widehat{f}$ and $\mathcal{F}^{-1}(f)=\breve{f}$ be the Fourier transform of $f$ and its inverse.

\vspace{2mm}

Then,  we recall the Littlewood-Paley decomposition and the definitions of Besov spaces. The reader can refer to Chapters 2 and 3 in \cite{Chemin} for more details. 
Choose a smooth radial non-increasing function $\chi(\xi)$  compactly supported in $B(0,\frac{4}{3})$ and satisfying $\chi(\xi)=1$ in $B(0,\frac{3}{4})$. Then, $\varphi(\xi)\triangleq\chi(\xi/2)-\chi(\xi)$ satisfies
$$
\sum_{j\in \mathbb{Z}}\varphi(2^{-j}\cdot)=1,\quad \text{{\rm{Supp}}}~ \varphi\subset \Big{\{}\xi\in \mathbb{R}^{d}~|~\frac{3}{4}\leq |\xi|\leq \frac{8}{3}\Big{\}}.
$$
For any $j\in \mathbb{Z}$, define the homogeneous dyadic block
$$
\dot{\Delta}_{j}u\triangleq\mathcal{F}^{-1}\big{(} \varphi(2^{-j}\cdot )\mathcal{F}(u) \big{)}=2^{jd}h(2^{j}\cdot)\star u,\quad\quad h\triangleq\mathcal{F}^{-1}\varphi.
$$
We also define the low-frequency cut-off operator
\begin{align}
\dot{S}_{j}\triangleq\sum_{j'\leq j-1}\dot{\Delta}_{j'}.\nonumber
\end{align}
Let $\mathcal{S}_{h}'$ stand for the set of tempered distributions $z$ on $\mathbb{R}^{d}$ such that $\dot{S}_{j}z\rightarrow0$ uniformly as $j\rightarrow\infty$ (i.e., modulo polynomials). One has
\begin{equation}\nonumber
\begin{aligned}
&u=\sum_{j\in \mathbb{Z}}\dot{\Delta}_{j}u\quad\text{in}\quad \mathcal{S}_h',\quad \quad \dot{\Delta}_{j}\dot{\Delta}_{l}u=0,\quad\text{if}\quad|j-l|\geq2.
\end{aligned}
\end{equation}

With the help of those dyadic blocks, we give the definitions of homogeneous Besov spaces and mixed space-time Besov spaces as follows. For $s\in \mathbb{R}$ and $1\leq p,r\leq \infty$, the  homogeneous Besov space $\dot{B}^{s}_{p,r}$ is defined by
$$
\dot{B}^{s}_{p,r}\triangleq\big{\{} u\in \mathcal{S}_{h}'~:~\|u\|_{\dot{B}^{s}_{p,r}}\triangleq\|\{2^{js}\|\dot{\Delta}_{j}u\|_{L^{p}}\}_{j\in\mathbb{Z}}\|_{l^{r}}<\infty \big{\}} .
$$
For $T>0$, $s\in\mathbb{R}$ and $1\leq \varrho,r, q \leq \infty$, we recall a class of mixed space-time Besov spaces $\widetilde{L}^{\varrho}(0,T;\dot{B}^{s}_{p,r})$ introduced by Chemin-Lerner \cite{chemin1}:
$$
\widetilde{L}^{\varrho}(0,T;\dot{B}^{s}_{p,r})\triangleq \big{\{} u\in L^{\varrho}(0,T;\mathcal{S}'_{h})~:~ \|u\|_{\widetilde{L}^{\varrho}_{T}(\dot{B}^{s}_{p,r})}\triangleq\|\{2^{js}\|\dot{\Delta}_{j}u\|_{L^{\varrho}_{T}(L^{p})}\}_{j\in\mathbb{Z}}\|_{l^{r}}<\infty \big{\}}.
$$
By Minkowski's inequality, it holds
\begin{equation}\nonumber
\begin{aligned}
&\|u\|_{\widetilde{L}^{\varrho}_{T}(\dot{B}^{s}_{p,r})}\leq(\geq) \|u\|_{L^{\varrho}_{T}(\dot{B}^{s}_{p,r})}\quad\text{if}~r\geq(\leq)\rho,
\end{aligned}
\end{equation}
where $\|\cdot\|_{L^{\varrho}_{T}(\dot{B}^{s}_{p,r})}$ is the usual Lebesgue-Besov norm. 

In order to restrict Besov norms to the low-frequency part and the high-frequency part, we write $\|\cdot\|_{\dot B_{q_1,r}^{s_1}}^{\ell}$ and
$\|\cdot\|_{\dot B_{q_2,r}^{s_2}}^{h}$ to denote Besov semi-norms with respect to the threshold $J_{\varepsilon}$ defined as \eqref{Jvar}, that is,
\begin{equation}\label{DefHLB}
\|u\|_{\dot B_{q_1,r}^{s_1}}^{\ell}\triangleq \Big\{\|2^{s_1 j}\|\dot \Delta_j u\|_{L^{q_1}}\}_{j\leq J_{\varepsilon}}\Big\|_{l^r},\quad 
\quad \mbox{and}\quad
\|u\|_{\dot B_{q_2,r}^{s_2}}^{h}\triangleq 
\Big\|\{2^{s_2j}\|\dot \Delta_j u\|_{L^{q_2}}\}_{j\geq J_{\varepsilon}-1}\Big\|_{l^r}.
\end{equation}
One can deduce that for all $\sigma_0>0$,
\begin{equation}\label{HLEst}
\|u\|_{\dot B_{q_1,1}^{s_1}}^{\ell}
\lesssim 2^{\sigma_0 J_\var}\|u\|_{\dot B_{q_1,\infty}^{s_1-\sigma_0}}^{\ell}\lesssim 2^{\sigma_0 J_\var}\|u\|_{\dot B_{q_1,1}^{s_1-\sigma_0}}^{\ell}
\quad \mbox{and}\quad
\|u\|_{\dot B_{q_2,1}^{s_2}}^{h}
\lesssim 
2^{-\sigma_0J_{\varepsilon}}\|u\|_{\dot B_{q_2,1}^{s_2+\sigma_0}}^{h}.
\end{equation}
We also introduce the low-high-frequency decomposition $u=u^{\ell}+u^{h}$ with
\begin{equation}\nonumber
u^\ell\triangleq\sum_{j\leq J_\var-1} \ddj u=\dot{S}_{J_\var}u
\quad\mbox{and}\quad
u^h\triangleq\sum_{j\geq J_\var} \ddj u=( {\rm{Id}}-\dot{S}_{J_\var})u.
\end{equation} 
It is easy to check that
\begin{equation}\label{HLEst11}
\|u^{\ell}\|_{\dot B_{q_1,r}^{s_1}}
\leq
\|u\|_{\dot B_{q_1,r}^{s_1}}^{\ell}
\quad\mbox{and}\quad
\|u^{h}\|_{\dot B_{q_2,r}^{s_2}}
\leq
\|u\|_{\dot B_{q_2,r}^{s_2}}^{h}.
\end{equation}
\medbreak

\section{Uniform a priori estimates and global well-posedness}\label{section3}

In this section, we give the key {\it{a priori}} estimates leading to the global existence of solutions for \eqref{JXSys1}.
\begin{Prop}\label{AprioriJXProp1}
Let $p$ satisfy \eqref{p}. For any $\var>0$ and given time $T>0$,  let $(u,\bv)$ be a solution to System \eqref{JXSys1} satisfying, for $0\leq t<T$,
\begin{equation*}
\|u\|_{L_t^\infty (L^\infty)}\leq 1.
\end{equation*}
Then,  for all $0\leq t<T$, it holds that
\begin{equation}\label{AprioriJXProp1Est}
\cX_{p}(t)\leq C \big(\cX_{p,0} +\cX_{p}(t)^2\big),
\end{equation}
where $\cX_{p}(t)$ is defined by
\begin{equation}\label{Xt}
\begin{aligned}
\cX_{p}(t)&\triangleq
\|u\|_{\tL_t^\infty(\dot B_{p,1}^{\frac{d}{p}-1}\cap\dB_{p,1}^{\frac{d}{p}})}^{\ell}
+\|u\|_{\tL_t^1(\dot B_{p,1}^{\frac{d}{p}+1}\cap\dB_{p,1}^{\frac{d}{p}+2})}^{\ell}+(1+\var)\|u\|_{\tL_t^\infty(\dot B_{2,1}^{\frac{d}{2}})}^{h}
+(\frac{1}{\var}+\frac{1}{\var^2})\|u\|_{\tL_t^1(\dot B_{2,1}^{\frac{d}{2}})}^{h}\\
&\quad+\var^2\|\bv\|_{\tL_t^\infty(\dot B_{p,1}^{\frac{d}{p}}\cap\dB_{p,1}^{\frac{d}{p}+1})}^{\ell}
+\|\bv\|_{\tL_t^1(\dot B_{p,1}^{\frac{d}{p}}\cap\dB_{p,1}^{\frac{d}{p}+1})}^{\ell}+(\var+\var^2)\|\bv\|_{\tL_t^\infty(\dot B_{2,1}^{\frac{d}{2}})}^{h}
+(1+\frac{1}{\var})\|\bv\|_{\tL^1_t(\dot B_{2,1}^{\frac{d}{2}})}^{h}.
\end{aligned}
\end{equation}
\end{Prop}

The proof of Proposition \ref{AprioriJXProp1} is divided into two cases: the low and high frequencies.



\subsection{Low-frequency analysis}\label{subsectionlow-frequency}

Let $\bz$ be the effective unknown defined by $\eqref{effective} $ and set $\bz|_{t=0}=\bz_{0}=(z_{1,0},z_{2,0},...,z_{d,0})$ with $z_{0,i}\triangleq A_{i}\frac{\partial} {\partial x_{i}} u_{0}+v_{0,i}$. Recall that $(u,\bz)$ satisfies \eqref{JXDZ}. By virtue of Lemma \ref{HeatRegulEstprop} for $\eqref{JXDZ}_{1}$, there exists a generic constant $C_{1}>0$ such that
\begin{equation}\label{LowmEst1}
\begin{aligned}
\|u\|_{\tL^{\infty}_{t}(\dot{B}^{\frac{d}{p}-1}_{p,1}\cap \dot{B}^{\frac{d}{p}}_{p,1})}^{\ell}
+\|u\|_{\tL^1_{t}(\dot{B}^{\frac{d}{p}+1}_{p,1}\cap \dot{B}^{\frac{d}{p}+2}_{p,1})}^{\ell}
\leq
C_1\Big(  \|u_{0}\|_{\dot{B}^{\frac{d}{p}-1}_{p,1}\cap \dot{B}^{\frac{d}{p}}_{p,1}}^{\ell}
+\|\bz\|_{ \tL^1_{t}(\dot{B}^{\frac{d}{p}}_{p,1}\cap \dot{B}^{\frac{d}{p}+1}_{p,1})}^{\ell}\Big).
\end{aligned}
\end{equation}
 As for $\bz$, according to \eqref{HLEst} and Lemma \ref{maximaldamped} for  $\eqref{JXDZ}_{2}$, one has
\begin{equation}\label{LowmEst1m}
\begin{aligned}
&\var^2\|\bz\|_{\tL^{\infty}_{t}(\dot{B}^{\frac{d}{p}}_{p,1}\cap \dot{B}^{\frac{d}{p}+1}_{p,1})}^{\ell}
+\|\bz\|_{\tL^1_{t}(\dot{B}^{\frac{d}{p}}_{p,1}\cap \dot{B}^{\frac{d}{p}+1}_{p,1})}^{\ell}\\
&\quad\leq C_2\Big(
\var^2\|\bz_{0}\|_{\dot{B}^{\frac{d}{p}}_{p,1}\cap \dot{B}^{\frac{d}{p}+1}_{p,1}}^{\ell}+\var^2 2^{2J_{\var}}\| u \|_{\tL^1_{t}(\dot{B}^{\frac{d}{p}+1}_{p,1}\cap \dot{B}^{\frac{d}{p}+2}_{p,1})}^{\ell}+\var^2 2^{2J_{\var}}\|\bz\|_{\tL^1_{t}(\dot{B}^{\frac{d}{p}}_{p,1}\cap \dot{B}^{\frac{d}{p}+1}_{p,1})}^{\ell}\Big)\\
&\quad\quad+C_2\| f(u)\|_{\tL^1_{t}(\dot{B}^{\frac{d}{p}}_{p,1}\cap \dot{B}^{\frac{d}{p}+1}_{p,1})}^{\ell},
\end{aligned}
\end{equation}
where $C_2$ is a universal constant. Then,  substituting \eqref{LowmEst1} into \eqref{LowmEst1m}, we deduce
\begin{equation}\label{LowzEst1}
\begin{aligned}
&\var^2\|\bz\|_{\tL^{\infty}_{t}(\dot{B}^{\frac{d}{p}}_{p,1}\cap \dot{B}^{\frac{d}{p}+1}_{p,1})}^{\ell}
+\|\bz\|_{\tL^1_{t}(\dot{B}^{\frac{d}{p}}_{p,1}\cap \dot{B}^{\frac{d}{p}+1}_{p,1})}^{\ell}\\
&\quad\leq C_2
\var^2\|\bz_{0}\|_{\dot{B}^{\frac{d}{p}}_{p,1}\cap \dot{B}^{\frac{d}{p}+1}_{p,1}}^{\ell}
+C_{1}C_{2}\var^2 2^{2J_{\var}}  \|u_{0}\|_{\dot{B}^{\frac{d}{p}-1}_{p,1}\cap \dot{B}^{\frac{d}{p}}_{p,1}}^{\ell}\\
&\quad\quad+C_{2}(C_1+1) \var^2 2^{2J_{\var}}\|\bz\|_{ \tL^1_{t}(\dot{B}^{\frac{d}{p}}_{p,1}\cap \dot{B}^{\frac{d}{p}+1}_{p,1})}^{\ell}
+C_{2}\| f(u)\|_{\tL^1_{t}(\dot{B}^{\frac{d}{p}}_{p,1}\cap \dot{B}^{\frac{d}{p}+1}_{p,1})}^{\ell}.
\end{aligned}
\end{equation}
Here  the threshold $J_{\var}$ takes the form \eqref{Jvar} such that $\var 2^{J_{\var}}=2^{k_{0}}$. Therefore, one is able to choose the integer $k_{0}$ such that 
\begin{align}
 2^{2k_0}< \frac{1}{2C_2(C_1+1)}. \label{k0}
\end{align}
Combining \eqref{LowmEst1}, \eqref{LowzEst1} and \eqref{k0} together, we obtain 
\begin{equation}\label{LmzEst1}
\begin{aligned}
&\|u\|_{\tL^{\infty}_{t}(\dot{B}^{\frac{d}{p}-1}_{p,1}\cap \dot{B}^{\frac{d}{p}}_{p,1})}^{\ell}
+\|u\|_{\tL^1_{t}(\dot{B}^{\frac{d}{p}+1}_{p,1}\cap \dot{B}^{\frac{d}{p}+2}_{p,1})}^{\ell}
+\var^2\|\bz\|_{\tL^{\infty}_{t}(\dot{B}^{\frac{d}{p}}_{p,1}\cap \dot{B}^{\frac{d}{p}+1}_{p,1})}^{\ell}
+\|\bz\|_{\tL^1_{t}(\dot{B}^{\frac{d}{p}}_{p,1}\cap \dot{B}^{\frac{d}{p}+1}_{p,1})}^{\ell}\\
&\quad\lesssim \|u_{0}\|_{\dot{B}^{\frac{d}{p}-1}_{p,1}\cap \dot{B}^{\frac{d}{p}}_{p,1}}^{\ell}
+\var^2\|\bz_{0}\|_{\dot{B}^{\frac{d}{p}}_{p,1}\cap \dot{B}^{\frac{d}{p}+1}_{p,1}}^{\ell}
+\| f(u)\|_{\tL^1_{t}(\dot{B}^{\frac{d}{p}}_{p,1}\cap \dot{B}^{\frac{d}{p}+1}_{p,1})}^{\ell}.
\end{aligned}
\end{equation}


Next, we bound the nonlinear term $f(u)$. Due to $p\geq \max\{1,\frac{2d}{d+2}\}$, it holds, by composition estimates in Lemma \ref{NewSmoothlow} with $(s,\sigma)=(d/p,d/2)$, that
\begin{equation}\label{LNolinerEst1}
\begin{aligned}
\|f(u)\|_{\tL_t^1(\dot{B}^{\frac{d}{p}}_{p,1})}^{\ell}
&\lesssim \int_0^t \Big( \big(\|u^\ell\|_{\dot B_{p,1}^{\frac{d}{p}} } 
+\|u\|_{\dot B_{2,1}^{\frac{d}{2}} }^h  \big)\|u\|_{\dot{B}^{\frac{d}{p}}_{p,1}}^\ell
+
\big(2^{J_\var}\|u^\ell\|_{\dot B_{p,1}^{\frac{d}{p}-1}}+ 
\|u\|_{\dot B_{2,1}^{\frac{d}{2}}}^h\big)
\|u\|_{\dot B_{2,1}^{\frac{d}{2}}}^{h} \Big) d\tau\\
&\lesssim
\Big(\| u\|_{\tL^2_{t}(\dot{B}^{\frac{d}{p}}_{p,1})}^{\ell} \Big)^2
+\Big(\| u\|_{\tL^2_{t}(\dot{B}^{\frac{d}{2}}_{2,1})}^{h} \Big)^2
+\| u\|_{\tL^\infty_{t}(\dot{B}^{\frac{d}{p}-1}_{p,1})}^{\ell}
\frac{1}{\var}\| u\|_{\tL^1_{t}(\dot{B}^{\frac{d}{2}}_{2,1})}^{h}.
\end{aligned}
\end{equation}
 Similarly, by Lemma \ref{NewSmoothlow} with $(s,\sigma)=(d/p+1,d/2)$ and \eqref{L2} one has
\begin{equation}\label{LNolinerEst11}
\begin{aligned}
\|f(u)\|_{\tL_t^1(\dot{B}^{\frac{d}{p}+1}_{p,1})}^{\ell}
&\lesssim \int_0^t \Big(\big(\|u\|_{\dot B_{p,1}^{\frac{d}{p}} }^\ell 
+\|u\|_{\dot B_{2,1}^{\frac{d}{2}} }^h  \big)\|u\|_{\dot{B}^{\frac{d}{p}+1}_{p,1}}^{\ell}
+
2^{J_\var}\big(2^{J_\var}\|u\|_{\dot B_{p,1}^{\frac{d}{p}-1}}^\ell+ 
\|u\|_{\dot B_{2,1}^{\frac{d}{2}}}^h\big)
\|u\|_{\dot B_{2,1}^{\frac{d}{2}}}^{h} \Big)d\tau\\
&\lesssim
\Big(\| u\|_{\tL^2_{t}(\dot{B}^{\frac{d}{p}}_{p,1}\cap\dot{B}^{\frac{d}{p}+1}_{p,1})}^{\ell} \Big)^2
+\| u\|_{\tL^{2}_{t}(\dot{B}^{\frac{d}{p}+1}_{p,1})}^{\ell}\| u\|_{\tL^2_{t}(\dot{B}^{\frac{d}{2}}_{2,1})}^{h}\\
&\quad\quad+\|u\|_{\tL^{\infty}_{t}(\dot{B}^{\frac{d}{p}-1}_{p,1})}^{\ell}\frac{1}{\var^2}\| u\|_{\tL^1_{t}(\dot{B}^{\frac{d}{2}}_{2,1})}^{h}+\frac{1}{\var}\Big(\|u\|_{\tL^{2}_{t}(\dot{B}^{\frac{d}{2}}_{2,1})}^{h}\Big)^2.
\end{aligned}
\end{equation}
From the interpolation inequalities in Lemma \ref{Interpolation} it holds that
\begin{equation}\label{L2}
\left\{
\begin{aligned}
&\| u\|_{\tL^2_{t}(\dot{B}^{\frac{d}{p}}_{p,1}\cap\dot{B}^{\frac{d}{p}+1}_{p,1})}^{\ell}\lesssim \Big( \|u\|_{\tL_t^{\infty}(\dot B_{p,1}^{\frac{d}{p}-1}\cap B_{p,1}^{\frac{d}{p}})}^{\ell}\Big)^{\frac{1}{2}}\Big( \|u\|_{\tL_t^{1}(\dot B_{p,1}^{\frac{d}{p}+1}\cap B_{p,1}^{\frac{d}{p}+2})}^{\ell}\Big)^{\frac{1}{2}}\lesssim \mathcal{X}_{p}(t),\\
&(1+\frac{1}{\var})\|u\|_{\tL_t^2(\dot B_{2,1}^{\frac{d}{2}})}^{h}\lesssim \Big( (1+\var)\|u\|_{\tL_t^{\infty}(\dot B_{2,1}^{\frac{d}{2}})}^{h}\Big)^{\frac{1}{2}}\Big( (\frac{1}{\var}+\frac{1}{\var^2})\|u\|_{\tL_t^1(\dot B_{2,1}^{\frac{d}{2}})}^{h}\Big)^{\frac{1}{2}}\lesssim \mathcal{X}_{p}(t).
\end{aligned}
\right.
\end{equation}
Collecting \eqref{LmzEst1}, \eqref{LNolinerEst1}, \eqref{LNolinerEst11} and \eqref{L2}, we have
\begin{equation}\label{LFinalEst11111}
\begin{aligned}
&\|u\|_{\tL^{\infty}_{t}(\dot{B}^{\frac{d}{p}-1}_{p,1}\cap\dot{B}^{\frac{d}{p}}_{p,1})}^{\ell}
+\|u\|_{\tL^1_{t}(\dot{B}^{\frac{d}{p}+1}_{p,1}\cap\dot{B}^{\frac{d}{p}+2}_{p,1})}^{\ell}
+\var^2\|\bz\|_{\tL^{\infty}_{t}(\dot{B}^{\frac{d}{p}}_{p,1}\cap\dot{B}^{\frac{d}{p}+1}_{p,1})}^{\ell}
+\|\bz\|_{\tL^1_{t}(\dot{B}^{\frac{d}{p}}_{p,1}\cap\dot{B}^{\frac{d}{p}+1}_{p,1})}^{\ell}\\
&\quad\lesssim
\|u_{0}\|_{\dot{B}^{\frac{d}{p}-1}_{p,1}\cap\dot{B}^{\frac{d}{p}}_{p,1}}^{\ell}
+\var^2\|\bz_{0}\|_{\dot{B}^{\frac{d}{p}}_{p,1}\cap\dot{B}^{\frac{d}{p}+1}_{p,1}}^{\ell}
+\cX_{p}(t)^2.
\end{aligned}
\end{equation}
To recover the information on $\bv$, one deduces from \eqref{effective} and \eqref{HLEst}  that
\begin{align}
\|\bv\|_{\tL^{1}_{t}(\dot{B}^{\frac{d}{p}}_{p,1}\cap \dot{B}^{\frac{d}{p}+1}_{p,1})}^{\ell}
&\lesssim
\|u\|_{\tL^{1}_t(\dot{B}^{\frac{d}{p}+1}_{p,1}\cap \dot{B}^{\frac{d}{p}+2}_{p,1})}^{\ell}
+\|\bz\|_{\tL^1_{t}(\dot{B}^{\frac{d}{p}}_{p,1}\cap \dot{B}^{\frac{d}{p}+1}_{p,1})}^{\ell}\label{Recoverv1}
\end{align}
and
\begin{align}
\var^2\| \bv\|_{\tL^{\infty}_{t}(\dot{B}^{\frac{d}{p}}_{p,1}\cap\dot{B}^{\frac{d}{p}+1}_{p,1})}^{\ell}
&\lesssim
\var^2 2^{2J_\var}\|u \|_{\tL^{\infty}_t(\dot{B}^{\frac{d}{p}-1}_{p,1}\cap\dot{B}^{\frac{d}{p}}_{p,1})}^{\ell}
+\var^2\|\bz\|_{\tL^{\infty}_{t}(\dot{B}^{\frac{d}{p}}_{p,1}\cap\dot{B}^{\frac{d}{p}+1}_{p,1})}^{\ell}\nonumber\\
&\lesssim
\| u\|_{\tL^{\infty}_t(\dot{B}^{\frac{d}{p}-1}_{p,1}\cap\dot{B}^{\frac{d}{p}}_{p,1})}^{\ell}
+\var^2\|\bz\|_{\tL^{\infty}_{t}(\dot{B}^{\frac{d}{p}}_{p,1}\cap\dot{B}^{\frac{d}{p}+1}_{p,1})}^{\ell}.\label{Recoverv2}
\end{align}
From \eqref{LFinalEst11111}, \eqref{Recoverv1} and \eqref{Recoverv2} we obtain
\begin{equation}\label{LFinalEst2}
\begin{aligned}
&\|u\|_{\tL^{\infty}_{t}(\dot{B}^{\frac{d}{p}-1}_{p,1}\cap\dot{B}^{\frac{d}{p}}_{p,1})}^{\ell}
+\|u\|_{\tL^1_{t}(\dot{B}^{\frac{d}{p}+1}_{p,1}\cap\dot{B}^{\frac{d}{p}+2}_{p,1})}^{\ell}+\var^2\|\bv\|_{\tL^{\infty}_{t}(\dot{B}^{\frac{d}{p}}_{p,1}\cap\dot{B}^{\frac{d}{p}+1}_{p,1})}^{\ell}
+\|\bv\|_{\tL^1_{t}(\dot{B}^{\frac{d}{p}}_{p,1}\cap\dot{B}^{\frac{d}{p}+1}_{p,1})}^{\ell}\\
&\quad\quad+\var^2\|\bz\|_{\tL^{\infty}_{t}(\dot{B}^{\frac{d}{p}}_{p,1}\cap\dot{B}^{\frac{d}{p}+1}_{p,1})}^{\ell}
+\|\bz\|_{\tL^1_{t}(\dot{B}^{\frac{d}{p}}_{p,1}\cap\dot{B}^{\frac{d}{p}+1}_{p,1})}^{\ell}\\
&\quad\lesssim
\cX_{p,0} +\cX_{p}(t)^2.
\end{aligned}
\end{equation}

\subsection{High-frequency analysis} \label{subsectionhigh}
For any $j\geq J_\varepsilon-1$, we localize System \eqref{JXSys1} as follows:
\begin{equation} \label{JXHD1}
\left\{
\begin{aligned}
&\frac{\partial} {\partial t}\ddj u+\sum_{i=1}^{d} \frac{\partial} {\partial x_{i}} \ddj v_{i}=0, \\
&\var^2\frac{\partial} {\partial t} \ddj v_i+ A_{i}\frac{\partial} {\partial x_{i}} \ddj u+\ddj v_{i}=  \ddj f_{i}(u) ,\quad  i=1,2,...,d.
\end{aligned}
\right.
\end{equation}
Multiplying $\eqref{JXHD1}_{1}$ by $\ddj u $ and integrating by parts, we have
\begin{equation}\label{HmEst1}
\frac{1}{2}\dfrac{d}{dt}\|\ddj u \|_{L^2}^2
-\sum_{i=1}^{d} \int_{\mathbb{R}^{d}} \ddj v_{i}\cdot  \frac{\partial} {\partial x_{i}} \ddj u\ dx=0.
\end{equation}
Meanwhile, taking the inner product of $\eqref{JXHD1}_{2}$ with $\dfrac{1}{a_{i}} \ddj v_{i} $ and summing over $1\leq i\leq d$, we get
\begin{equation}\label{HwEst2}
\begin{aligned}
&\sum_{i=1}^{d}\frac{\var^2}{2a_{i}}\dfrac{d}{dt} \|\ddj  v_{i} \|_{L^2}^2
+\sum_{i=1}^{d} \int_{\mathbb{R}^{d}} \ddj v_{i} \cdot   \frac{\partial}{\partial x_{i}} \ddj u\ dx
+\sum_{i=1}^{d}\frac{1}{a_{i}} \| \ddj v_{i} \|_{L^2}^2
\\
&\quad=\sum_{i=1}^{d}\frac{1}{a_{i}}\int_{\mathbb{R}^{d}} \ddj f_{i}(u) \cdot  \ddj v_{i}\ dx.
\end{aligned}
\end{equation}
Adding \eqref{HmEst1} and \eqref{HwEst2} together yields 
\begin{equation}\label{HmwEnergyEst}
\begin{aligned}
&\frac{1}{2}\frac{d}{dt}\sum_{i=1}^{d}\Big(
\|\ddj u \|_{L^2}^2+\frac{\var^2}{a_{i}} \|\ddj  v_{i} \|_{L^2}^2\Big)
+\sum_{i=1}^{d}\frac{1}{a_{i}} \| \ddj v_{i} \|_{L^2}^2=\sum_{i=1}^{d}\frac{1}{a_{i}}\int_{\mathbb{R}^{d}} \ddj f_{i}(u) \cdot  \ddj v_{i}\ dx.
\end{aligned}
\end{equation}
To derive the dissipation for $u$, we perform the following cross estimate:
\begin{equation}\label{HCrossEst}
\begin{aligned}
&\sum_{i=1}^{d}\frac{1}{a_{i}}\dfrac{d}{dt}\int_{\mathbb{R}^{d}} \ddj  v_{i} \cdot \frac{\partial} {\partial x_{i}} \ddj u\ dx
+\frac{1}{\var^2}\sum_{i=1}^{d}\Big\|\frac{\partial} {\partial x_{i}} \ddj  u \Big\|_{L^2}^2
\\
&\quad\quad- \sum_{i=1}^{d} \frac{1}{a_{i}} \int_{\mathbb{R}^{d}} \frac{\partial} {\partial x_{i}}  \ddj v_i \cdot  \sum_{k=1}^{d} \frac{\partial} {\partial x_{k}}  \ddj v_{k}\  dx+\frac{1}{\varepsilon^2}\sum_{i=1}^{d} \frac{1}{a_{i}}\int_{\mathbb{R}^{d}} \ddj v_{i} \cdot  \frac{\partial} {\partial x_{i}} \ddj u\ dx
\\
&\quad=\frac{1}{\var^2}\sum_{i=1}^{d} \frac{1}{a_{i}} \int_{\mathbb{R}^{d}} \ddj f_{i}(u) \cdot \frac{\partial} {\partial x_{i}} \ddj u \ dx.
\end{aligned}
\end{equation}
For any $j\geq J_{\var}-1$ and some suitably small constant  $\zeta>0$ to be determined later, the linear combination of \eqref{HmwEnergyEst} and \eqref{HCrossEst} leads to
\begin{equation}\label{HFinal}
\begin{aligned}
&\frac{d}{dt}\mathcal{L}_{j}^2(t)+\mathcal{H}^2_{j}(t)
\leq
\frac{1}{\var}\sum_{i=1}^{d}\frac{1}{a_{i}}\|\ddj f_{i}(u)\|_{L^2} \Big(\var\|\ddj v_i\|_{L^2}
+ \frac{2^{-2j}\zeta}{\varepsilon}  \Big\|\frac{\partial} {\partial x_{i}} \ddj u\Big\|_{L^2}\Big),
\end{aligned}
\end{equation}
where the Lyapunov functional $\mathcal{L}^2_j(t)$ and its dissipation rate $\mathcal{H}^2_j(t)$ are defined by
\begin{equation}\nonumber
\left\{
\begin{aligned}
&\mathcal{L}^2_j(t)\triangleq\frac{1}{2}\|\ddj u \|_{L^2}^2
+\sum_{i=1}^{d}\frac{1}{2a_{i}}\|\var\ddj v_{i}\|_{L^2}^2
+2^{-2j}\zeta\sum_{i=1}^{d}\frac{1}{a_{i}}\int_{\mathbb{R}^{d}} \ddj  v_{i} \cdot \frac{\partial} {\partial x_{i}} \ddj u\, dx ,\\
&\mathcal{H}^2_{j}(t)\triangleq\sum_{i=1}^{d}\frac{1}{a_{i}} \| \ddj v_{i}\|_{L^2}^2
+\frac{2^{-2j}\zeta}{\varepsilon^2} \sum_{i=1}^{d} \Big( \Big\|\frac{\partial} {\partial x_{i}} \ddj  u\Big\|_{L^2}^2\\
&\quad\quad\quad\quad- \frac{\var^2}{a_{i}} \int_{\mathbb{R}^{d}} \frac{\partial} {\partial x_{i}}  \ddj v_{i} \cdot \sum_{k=1}^{d} \frac{\partial} {\partial x_{k}}  \ddj v_{k} \ dx+ \frac{1}{ a_{i}}\int_{\mathbb{R}^{d}} \ddj v_{i}\cdot \frac{\partial} {\partial x_{i}} \ddj u\, dx \Big).
\end{aligned}
\right.
\end{equation}
Using $2^{-j}\lesssim \var$ for $j\geq J_\var-1$ as well as Bernstein's and Young's inequalities and choosing a  uniform constant $\zeta$, we deduce 
\begin{equation}\label{ineq}
\begin{aligned}
&\mathcal{L}^2_j(t)\sim \|\ddj(u,\var \bv)\|_{L^2}^2,\quad\quad \mathcal{H}^2_{j}(t)\gtrsim \frac{1}{\var^2} \|\ddj(u,\var \bv)\|_{L^2}^2.
\end{aligned}
\end{equation}
Thus, it follows from \eqref{HFinal} and \eqref{ineq} that
\begin{align}\label{Lyine}
\frac{d}{dt}\mathcal{L}_{j}^2(t)+\frac{1}{\var^2}\mathcal{L}_{j}^2(t)
\lesssim  \frac{1}{\varepsilon }\sum_{i=1}^{d}\|\ddj f_{i}(u)  \|_{L^2}  \mathcal{L}_{j}(t).
\end{align}
Dividing \eqref{Lyine} by $\sqrt{\mathcal{L}_{j}(t)+\nu_{0}}$, integrating the resulting inequality over $[0,t]$ and then,  taking the limit as $\nu_{0}\rightarrow 0$, we reach
\begin{align*}
&\|\ddj( u,\var \bv)\|_{L^2}+\frac{1}{\var^2}\|\ddj( u,\var\bv)\|_{L_t^1(L^2)}
\lesssim
\|\dot{\Delta}_{j}( u_{0}, \var\bv_{0})\|_{L^2}+ \frac{1}{\var}\sum_{i=1}^{d}\|\ddj f_{i}(u) \|_{L_t^1(L^2)},
\end{align*}
which yields
\begin{equation}\label{HEstFina12}
\begin{aligned}
&(1+\var)\|(u, \var \bv)\|_{\tL^{\infty}_{t}(\dot{B}^{\frac{d}{2}}_{2,1})}^{h}
+(\frac{1}{\var}+\frac{1}{\var^2})\|(u,\var\bv)\|_{\tL^1_{t}(\dot{B}^{\frac{d}{2}}_{2,1})}^{h}\\
&\quad\lesssim
(1+\var)\|( u_0,\var\bv_0)\|_{\dot{B}_{2,1}^{\frac{d}{2}}}^{h}
+(1+\frac{1}{\var})\|f(u)\|_{\tL^1_t(B^{\frac {d}{2}}_{2,1})}^{h}.
\end{aligned}
\end{equation}
The nonlinear term $f(u)$ is analyzed as follows. Owing to $2^{J_{\var}}\sim \var^{-1}$,  Lemma \ref{NewSmoothhigh} with $(s,\sigma)=(d/2,d/p+1)$ and \eqref{L2}, we have
\begin{equation}\label{HighfDifficult}
\begin{aligned}
\|f(u)\|_{\tL^1_t(B^{\frac d2}_{2,1})}^{h}&\lesssim
\int_0^t\Big( \big(\|u\|_{\dot B_{p,1}^{\frac{d}{p}} }^\ell +\|u\|_{\dot B_{2,1}^{\frac{d}{2}} }^h\big)\|u\|_{\dot{B}^{\frac{d}{2}}_{2,1}}^{h}
+2^{-J_\var}
\big(2^{J_\var}\|u\|_{\dot B_{p,1}^{\frac{d}{p}-1}}^\ell+\|u\|_{\dot B_{2,1}^{\frac{d}{2}}}^h \big)
\|u\|_{\dot B_{p,1}^{\frac{d}{p}+1}}^{\ell} \Big) d\tau\\
&\lesssim
\Big(\| u\|_{\tL^2_{t}(\dot{B}^{\frac{d}{p}}_{p,1})}^{\ell} \Big)^2
+\Big(\| u\|_{\tL^2_{t}(\dot{B}^{\frac{d}{2}}_{2,1})}^{h} \Big)^2
+\big(\|u\|_{\tL_t^\infty(\dot B_{p,1}^{\frac{d}{p}-1})}^{\ell} 
+\frac{1}{\var}\|u\|_{\tL_t^\infty(\dot B_{2,1}^{\frac{d}{2}})}^{h} \big)
\|u\|_{\tL_t^1(\dot B_{p,1}^{\frac{d}{p}+1})}^{\ell}\\
&\lesssim \cX_{p}(t)^2.
\end{aligned}
\end{equation}
Similarly, one gets
\begin{align}
\frac{1}{\var}\|f(u)\|_{\tL^1_t(B^{\frac d2}_{2,1})}^{h}&\lesssim
\frac{1}{\var}\int_0^t \Big( \big(\|u\|_{\dot B_{p,1}^{\frac{d}{p}} }^\ell +\|u\|_{\dot B_{2,1}^{\frac{d}{2}} }^h\big)\|u\|_{\dot{B}^{\frac{d}{2}}_{2,1}}^{h}
+2^{-2J_\var}
\big(2^{J_\var}\|u\|_{\dot B_{p,1}^{\frac{d}{p}-1}}^\ell+\|u\|_{\dot B_{2,1}^{\frac{d}{2}}}^h \big)
\|u\|_{\dot B_{p,1}^{\frac{d}{p}+2}}^{\ell} \Big) d\tau\nonumber\\
&\lesssim
\big( \|u\|_{\tL_t^\infty(\dot B_{p,1}^{\frac{d}{p}})}^{\ell} 
+\|u\|_{\tL_t^\infty(\dot B_{2,1}^{\frac{d}{2}})}^{h} \big)
\frac{1}{\var}\|u\|_{\tL_t^1(\dot B_{2,1}^{\frac{d}{2}})}^{h} 
+\big(\|u\|_{\tL_t^\infty(\dot B_{p,1}^{\frac{d}{p}-1})}^{\ell} 
+\var\|u\|_{\tL_t^\infty(\dot B_{2,1}^{\frac{d}{2}})}^{h} \big)
\|u\|_{\tL_t^1(\dot B_{p,1}^{\frac{d}{p}+2})}^{\ell}\nonumber\\\
&\lesssim\cX_{p}(t)^2.\label{HighfDifficult.2}
\end{align}
It thus holds by  \eqref{HEstFina12}, \eqref{HighfDifficult} and \eqref{HighfDifficult.2} that
\begin{equation}\label{HEstFinal2lll}
\begin{aligned}
(1+\var)\|(u,\var \bv)\|_{\tL^{\infty}_{t}(\dot{B}^{\frac{d}{2}}_{2,1})}^{h}
+(\frac{1}{\var}+\frac{1}{\var^2})\|(u,\var\bv)\|_{\tL^1_{t}(\dot{B}^{\frac{d}{2}}_{2,1})}^{h}
\lesssim \cX_{p,0}
+\cX_{p}(t)^2.
\end{aligned}
\end{equation}
The combination of estimates \eqref{LFinalEst2} and \eqref{HEstFinal2lll} gives rise to \eqref{AprioriJXProp1Est}, and concludes the proof of Proposition \ref{AprioriJXProp1}.

\subsection{Proof of global existence and uniqueness }

In order to prove Theorem  \ref{Thm1}, we first need to justify the existence and uniqueness of local-in-time solutions for System \eqref{JXSys1}. Define the space
\begin{equation}\nonumber
\begin{aligned}
E(T)\triangleq\Big\{(u,\bv)~:~& u^{\ell}\in\cC([0,T];\dB_{p,1}^{\frac{d}{p}-1}),~ 
\bv^{\ell} \in\cC([0,T];\dB_{p,1}^{\frac{d}{p}}), ~
(u^{h},\bv^{h})
\in \cC([0,T];\dB_{2,1}^{\frac{d}{2}})\Big\}
\end{aligned}
\end{equation}
and its associated norm
$$
\|(u,\bv)\|_{E(T)}\triangleq \sup_{t\in[0,T]}\Big(\|u(t)\|_{\dB_{p,1}^{\frac{d}{p}-1}\cap \dB_{p,1}^{\frac{d}{p}}}^{\ell}+(1+\var)\|u(t)\|_{\dB_{2,1}^{\frac{d}{2}}}^{h}+\var^2\|\bv(t)\|_{\dB_{p,1}^{\frac{d}{p}}\cap\dB_{p,1}^{\frac{d}{p}+1}}^{\ell}+\var(1+\var)\|\bv(t)\|_{\dB_{2,1}^{\frac{d}{2}}}^{h}\Big).
$$
We also denote the space of initial data by
\begin{equation}\nonumber
\begin{aligned}
E_{0}\triangleq\Big\{(u_{0},\bv_{0})~:~u_{0}^{\ell}\in \dB_{p,1}^{\frac{d}{p}-1},~\bv_{0}^{\ell}\in \dB_{p,1}^{\frac{d}{p}},~ (u_{0}^{h},\bv_{0}^{h})\in \dB_{2,1}^{\frac{d}{2}}\Big\},
\end{aligned}
\end{equation}
equipped with the norm $\|(u_{0},\bv_{0})\|_{E_{0}}\triangleq\cX_{p,0}$ with $\cX_{p,0}$ defined in \eqref{small}.

\begin{Thm}\label{thmlocal}
{\rm(}Local well-posedness{\rm)} Let $p$ satisfy \eqref{p}, and the threshold $J_{\var}$ be given by \eqref{Jvar}. Assume  $(u_{0},\bv_{0})\in E_{0}$. Then,  there exists a time $T_{*}>0$ such that System \eqref{JXSys1} associated with the initial data $(u_{0},\bv_{0})$ admits a unique strong solution $(u,\bv)\in  E(T_{*})$.
\end{Thm}

\vspace{2mm}

With Theorem \ref{thmlocal} in hand, we can conclude the proof of Theorem \ref{Thm1}.
 \smallbreak

 \noindent
\textbf{\emph{Proof of Theorem \ref{Thm1}}.}  Under the assumption \eqref{small}, Theorem \ref{thmlocal} implies the existence and uniqueness of the solution  $(u,\bv)$ to System \eqref{JXSys1} on $[0,T_{\max})$ with a maximal time $T_{\max}>0$. In light of the uniform {\it{a priori}} estimate \eqref{AprioriJXProp1Est} obtained in Proposition \ref{AprioriJXProp1} and a standard bootstrap argument, one can justify $T_{\max}=\infty$ and verify that the global solution  $(u,\bv)$ fulfills the property \eqref{ThmEst1}. \qed

\vspace{2mm}

\noindent
\textbf{\emph{Proof of Theorem \ref{thmlocal}}.} The proof of Theorem \ref{thmlocal} is divided into the following four steps.
\begin{itemize}
\item \underline{Step 1: Construction of the approximate sequence.} 
\end{itemize}

For $k=1,2,..$, define the regularized initial data
\begin{eqnarray}\nonumber
(u_{0}^{k},\bv_{0}^{k})\triangleq
\begin{cases}
\chi(L_{k}x)\sum\limits_{|j|\leq k}\dot{\Delta}_{j}(u_{0},\bv_{0}),
& \mbox{if $p\geq 2,$ } \\
\sum\limits_{|j|\leq k}\dot{\Delta}_{j}(u_{0},\bv_{0}),
& \mbox{if $p<2,$}
\end{cases}
\end{eqnarray}
where $\chi(x) \in \mathcal{S}(\mathbb{R}^{d})$ satisfies  $\chi(0)=1$ and Supp $\mathcal{F}(\chi)(\xi)\subset B(0,1)$, and $L_{k}>0$ is a constant which fulfills $\lim\limits_{k\rightarrow\infty}L_{k}=0$ and will be chosen later.

We claim that $(u_{0}^{k},\bv_{0}^{k})$ belongs to $H^{s_{0}}(\mathbb{R}^{d})$ for $s_{0}\geq[\frac{d}{2}]+1$ and fixed $k\geq 0$. Indeed, in the case $p\geq2$, one deduces from Bernstein's inequality and the embedding $\dot{B}^{\frac{d}{p}}_{p,1}\hookrightarrow L^{\infty}(\mathbb{R}^{d}) $ that
\begin{equation}\nonumber
\begin{aligned}
\|(u_{0}^{k},\bv_{0}^{k})\|_{H^{s_{0}}(\mathbb{R}^{d})}&\lesssim 2^{k s_{0}}\|(u_{0}^{k},\bv_{0}^{k})\|_{L^2(\mathbb{R}^{d})}\lesssim 2^{k s_{0}} \|\chi(L_{k} \cdot)\|_{L^2(\mathbb{R}^{d})}\|(u_{0},v_{0})\|_{L^{\infty}(\mathbb{R}^{d})}<\infty.
\end{aligned}
\end{equation}
As for the case $p<2$, due to Bernstein's inequality and Sobolev embeddings, one also has $(u_{0}^{k},\bv_{0}^{k})\in W^{s_{0}+\frac{d}{p}-\frac{d}{2},p}(\mathbb{R}^{d})\hookrightarrow H^{s_{0}}(\mathbb{R}^{d})$.

Then, we explain that there exists a suitably large integer $k_{0}$ such that for all $k\geq k_{0}$, $(u_{0}^{k},\bv_{0}^{k})$ has the following uniform bound:
\begin{equation}\label{appdatauniform}
\begin{aligned}
\|(u^{k}_{0},\bv^{k}_{0})\|_{E_{0}}\lesssim \|(u_{0},\bv_{0})\|_{E_{0}}.
\end{aligned}
\end{equation}
We only justify \eqref{appdatauniform} when $p\geq2$ as the case $p<2$ is clear. Recalling the invariance of the norm $\dot{B}^{\frac{d}{p}}_{p,1}$ by spatial dilation (see, e.g.,  \cite[Proposition 2.18]{Chemin}), we have $\|\chi(\frac{\cdot}{k})\|_{\dot{B}^{d/p}_{p,1}}\sim \|\chi\|_{\dot{B}^{d/p}_{p,1}}\lesssim 1$. Thus, for $p\geq 2$, the classical product law \eqref{ClassicalProductLawEst2} and Bernstein's inequality ensure that
\begin{equation}\nonumber
\begin{aligned}
&\|u_{0}^{k}\|_{\dot{B}^{\frac{d}{p}}_{p,1}}^{\ell}\lesssim \|\chi\|_{\dot{B}^{\frac{d}{p}}_{p,1}}\|u_{0}\|_{\dot{B}^{\frac{d}{p}}_{p,1}}\lesssim \|u_{0}\|_{\dot{B}^{\frac{d}{p}}_{p,1}}^{\ell}+\|u_{0}\|_{\dot{B}^{\frac{d}{2}}_{2,1}}^{h}.
\end{aligned}
\end{equation}
Employing the hybrid product law \eqref{newproduct}, we also have
\begin{equation}\nonumber
\begin{aligned}
\|u^{k}_{0}\|_{\dot{B}^{\frac{d}{2}}_{2,1}}^{h}\lesssim(\|\chi\|_{\dot{B}^{\frac{d}{p}}_{p,1}}+\|\chi\|_{\dot{B}^{\frac{d}{2}}_{2,1}})(\|u_{0}\|_{\dot{B}^{\frac{d}{p}}_{p,1}}^{\ell}+\|u_{0}\|_{\dot{B}^{\frac{d}{2}}_{2,1}}^{h})\lesssim \|u_{0}\|_{\dot{B}^{\frac{d}{p}}_{p,1}}^{\ell}+\|u_{0}\|_{\dot{B}^{\frac{d}{2}}_{2,1}}^{h}.
\end{aligned}
\end{equation}
Similarly, one can obtain the desired bounds for $\bv_{0}^{k}$ in \eqref{appdatauniform}. Then, it suffices to show 
\begin{align}
&\lim_{k\rightarrow\infty}\|u^{k}_{0}-u_{0}\|_{\dot{B}^{\frac{d}{p}-1}_{p,1}}^{\ell}=0.\label{u0strong}
\end{align}
In fact, inspired by \cite{abidi1}[Lemma 4.2] we decompose 
\begin{align}
&u^{k}_{0}-u_{0}= (\lambda(L_{k} x)-1)\sum_{|j|\leq k}\dot{\Delta}_{j}u_{0}-               \sum_{|j|\geq k+1}\dot{\Delta}_{j}u_{0}.\nonumber
\end{align}
For any $\eta>0$, one can find a suitably large integer $k_{0}^{*}=k_{0}^{*}(\eta)$ such that for all $k\geq k_{0}^{*}$,
\begin{equation}\nonumber
\begin{aligned}
&\|\sum_{|j|\geq k+1}\dot{\Delta}_{j}u_{0}\|_{\dot{B}^{\frac{d}{p}-1}_{p,1}}^{\ell}\leq C \sum_{j\leq -k+1,~j\leq J_{\var}}2^{(\frac{d}{p}-1)j}\|\dot{\Delta}_{j}u_{0}\|_{L^{p}}<\eta.
\end{aligned}
\end{equation}
On the other hand, since Supp $\mathcal{F}(\chi(L_{k} \cdot))\subset B(0,L^{k})$ and
$$
\text{Supp}~\mathcal{F}\Big(\sum_{|j|\leq k}\dot{\Delta}_{j}u_{0}\Big)\subset \{\xi\in\mathbb{R}^{d} : \frac{3}{4}2^{-k}\leq |\xi|\leq \frac{8}{3}2^{k}\},
$$
we have 
$$
\text{Supp}~\mathcal{F}(u_{0}^{k})\subset \{\xi\in\mathbb{R}^{d} : \frac{3}{8}2^{-k}\leq |\xi|\leq \frac{11}{3}2^{k}\}~~ \text{for}~~ L_{k}\leq \frac{3}{8}2^{-k}, 
$$
so  $\dot{\Delta}_{j'}u_{0}^{k}=0$ if $|j'|\geq k+3$. Hence, from the definition of $\chi$ it follows that
\begin{equation}\nonumber
\begin{aligned}
\|(\lambda(L_{k} x)-1)\sum_{|j|\leq k}\dot{\Delta}_{j}u_{0}\|_{\dot{B}^{\frac{d}{p}-1}_{p,1}}^{\ell}&\lesssim 2^{k} \|(\lambda(L_{k} \cdot)-1)\sum_{|j|\leq k}\dot{\Delta}_{j}u_{0}\|_{\dot{B}^{\frac{d}{p}}_{p,1}}^{\ell}\\
&\lesssim 2^{k} \|\lambda(L_{k} \cdot)-1\|_{\dot{B}^{\frac{d}{p}}_{p,1}}\|u_{0}\|_{\dot{B}^{\frac{d}{p}}_{p,1}}\rightarrow 0\quad\text{as}\quad k\rightarrow \infty,
\end{aligned}
\end{equation}
as long as we choose a suitable constant $L_{k}$. Therefore, we have \eqref{u0strong} and complete the proof of \eqref{appdatauniform}.





According to the classical local well-posedness theorem for hyperbolic systems (see, e.g., \cite{Majdalocal,Dafermos1}), for fixed $k\geq k_{0}$ and any $s_{0}\geq[\frac{d}{2}]+1$, there exists a time $T_{k}$ such that the Cauchy problem of System \eqref{JXDZ} associated with the initial data $(u_{0}^{k},\bv_{0}^{k})$ has a unique solution $(u^k,\bv^k)\in \cC([0,T_{k}];H^{s_{0}}(\mathbb{R}^{d}))$.

\begin{itemize}
\item \underline{Step 2: Uniform estimates.} 
\end{itemize}

Performing similar computations as in Section \ref{section3} and using \eqref{appdatauniform}, for all $k\geq k_{0}$ and $0<T<T_{k}$, we have
\begin{equation}
\begin{aligned}
&\|(u^{k},\bv^{k})\|_{E(T)}\lesssim \|(u_{0},v_{0})\|_{E_{0}}+\| f(u^{k})\|_{L^1_{T}(\dot{B}^{\frac{d}{p}}_{p,1}\cap \dot{B}^{\frac{d}{p}+1}_{p,1})}^{\ell}+(1+\frac{1}{\var})\|f(u^{k})\|_{L^1_T(B^{\frac {d}{2}}_{2,1})}^{h}.
\end{aligned}
\end{equation}
Then, Lemmas \ref{LemNewSmoothlow} and \ref{NewSmoothhigh} guarantee that
\begin{equation}\nonumber
\begin{aligned}
\| f(u^{k})\|_{L^1_{T}(\dot{B}^{\frac{d}{p}}_{p,1}\cap \dot{B}^{\frac{d}{p}+1}_{p,1})}^{\ell}&\lesssim T(1+\frac{1}{\var})\| f(u^{k})\|_{L^{\infty}_{T}(\dot{B}^{\frac{d}{p}}_{p,1})}^{\ell}\\
&\lesssim
T(1+\frac{1}{\var})\Big(\| u^{k}\|_{L^\infty_{T}(\dot{B}^{\frac{d}{p}}_{p,1})}^{\ell} \Big)^2
+T(1+\frac{1}{\var})\Big(\| u^{k}\|_{L^{\infty}_{T}(\dot{B}^{\frac{d}{2}}_{2,1})}^{h} \Big)^2\\
&\quad+\frac{T}{\var}(1+\frac{1}{\var})\| u^{k}\|_{L^\infty_{T}(\dot{B}^{\frac{d}{p}-1}_{p,1})}^{\ell}
\| u^{k}\|_{L^{\infty}_{T}(\dot{B}^{\frac{d}{2}}_{2,1})}^{h},
\end{aligned}
\end{equation}
and
\begin{equation}\nonumber
\begin{aligned}
\|f(u^{k})\|_{L^1_T(B^{\frac d2}_{2,1})}^{h}&\lesssim
T\Big(\| u^{k}\|_{L^{\infty}_{T}(\dot{B}^{\frac{d}{p}}_{p,1})}^{\ell} \Big)^2
+T\Big(\| u^{k}\|_{L^{\infty}_{T}(\dot{B}^{\frac{d}{2}}_{2,1})}^{h} \Big)^2\\
&\quad+\frac{T}{\var}\big(\|u^{k}\|_{L_T^\infty(\dot B_{p,1}^{\frac{d}{p}-1})}^{\ell} 
+\frac{1}{\var}\|u^{k}\|_{L_T^\infty(\dot B_{2,1}^{\frac{d}{2}})}^{h} \big)
\|u^{k}\|_{L_T^{\infty}(\dot B_{p,1}^{\frac{d}{p}})}^{\ell}.
\end{aligned}
\end{equation}
Consequently, one can find a generic constant $C_{*}>0$ such that
\begin{equation}\label{localuniform}
\begin{aligned}
&\|(u^{k},\bv^{k})\|_{E(T)}\leq C_{*} \|(u_{0},\bv_{0})\|_{E_{0}}+C_{*}T(1+\frac{1}{\var})^2 \|(u^{k},\bv^{k})\|_{E(T)}^2,\quad\quad k\geq k_{0}.
\end{aligned}
\end{equation}
We define the time
\begin{align}
T_{*} \triangleq \frac{1}{5C_{*}^2(1+\frac{1}{\var})^2\|(u_{0},\bv_{0})\|_{E_{0}}}\label{T}
\end{align}
and the time set
\begin{equation}\label{In}
\begin{split}
&I^{k}:=\Big{\{}t\in [0,T_{*}]~:~\|(u^{k},\bv^{k})\|_{E(t)}\leq 2 C_{*} \|(u_{0},\bv_{0})\|_{E_{0}}\Big\}.
\end{split}
\end{equation}
 Due to the time continuity of $(u^{k},\bv^{k})$ in $E(t)$ for all $t\in [0,T_{*}]$,  $I^{k}$ is a nonempty closed subset of $[0,T]$ for every $k\geq k_{0}$. By \eqref{localuniform} and the definition (\ref{In}) we have
\begin{equation}
\begin{aligned}
&\|(u^{k},\bv^{k})\|_{E(t)}< 2C_{*} \|(u_{0},\bv_{0})\|_{E_{0}},\quad\quad t\in I^{k},\quad k\geq k_{0}.\nonumber
\end{aligned}
\end{equation}
Again, using the time continuity of $(u^{k},\bv^{k})$ in $E(t)$ for all $t\in [0,T_{*}]$, one can show that there exists a ball $B(t,\eta_{*})$ for a suitably small constant $\eta_{*}>0$ such that $[0,T_*]\cap B(t,\eta_{*})\subset I^{k}$, which implies that $I^{k}$ is also an open subset of $[0,T_*]$. Hence, we have $I^{k}=[0,T_*]$, and the approximate sequence $(u^{k},\bv^{k})$ satisfies the estimate 
$$
\|(u^{k},\bv^{k})\|_{E(T_*)}\leq 2C_{*} \|(u_{0},\bv_{0})\|_{E_{0}},
$$
which is uniform with respect to $k\geq k_{0}$. With this bound, one can establish the higher-order $H^{s_0}(\mathbb{R}^d)$-estimate for $(u^{k},\bv^{k})$. This implies that, in fact, $T_{*}<T_{k}$ holds for all $k\geq k_0$.

\begin{itemize}
\item \underline{Step 3: Convergence of the approximate sequence}
\end{itemize}

The uniform estimate established in Step 2 implies that there exists $(u,\bv)$ such that as $k\rightarrow \infty$, it holds up to a subsequence that
\begin{equation}
\begin{aligned}
&(u^{k},\bv^{k})\overset{\ast}{\rightharpoonup} (u,\bv) \quad\quad \text{in}\quad L^{\infty}(0,T_{*};L^{\infty}(\mathbb{R}^{d}) ).\nonumber
\end{aligned}
\end{equation}
In order to justify the convergence of $f(u^{k})$ in $\eqref{JXDZ}_{2}$, one needs to show the strong compactness of $\{u^{k}\}_{k\geq k_{0}}$ in a suitable sense. From $\eqref{JXDZ}_{1}$ and the uniform estimate of $\bv^{k}$ we have
$$
\|\partial_{t}u^{k}\|_{L^{\infty}(0,T;\dot{B}^{\frac{d}{\max\{2,p\}}-1}_{\max\{2,p\},1})}\lesssim \|v^{k}\|_{L^{\infty}(0,T;\dot{B}^{\frac{d}{p}}_{p,1})}^{\ell}+\|v^{k}\|_{L^{\infty}(0,T;\dot{B}^{\frac{d}{2}}_{2,1})}^{h}\lesssim \|(u_{0},\bv_{0})\|_{E_{0}}.
$$
Gathering this and the compact embedding $\dot{B}^{\frac{d}{\max\{2,p\}}}_{\max\{2,p\},1}\hookrightarrow L^1_{{\rm{loc}}}(\mathbb{R}^{d})$, one infers from the Aubin-Lions lemma and the Cantor diagonal argument that, as $k\rightarrow \infty$, for any bounded set $K\subset \mathbb{R}^{d},$
\begin{equation}\nonumber
\begin{aligned}
&u^{k}\rightarrow u \quad\quad \text{in}\quad L^{1}(0,T_{*};L^{1}(K)),
\end{aligned}
\end{equation}
which implies that
\begin{equation}\nonumber
\begin{aligned}
&\|f(u^{k})-f(u)\|_{L^{1}(0,T_{*};L^{1}(K) )}\\
&\quad\leq \sup_{\tau\in[0,1]}\|\nabla_{u}f(u+\tau(u^{k}-u))\|_{L^{\infty}(0,T_{*};L^{\infty}(\mathbb{R}^{d}))}\|u^{k}-u\|_{L^{1}(0,T_{*};L^{1}(K) )}\rightarrow 0.
\end{aligned}
\end{equation}
Therefore, $(u,\bv)$ indeed satisfies System \eqref{JXDZ} in the sense of distributions. Finally, we justify the time continuity of the solution. Taking advantage of $\eqref{JXDZ}_{1}$, for any $0\leq t_{1},t_{2}\leq T_{*}$, we have
\begin{equation}\nonumber
\begin{aligned}
&\|u^{\ell}(t_{1})-u^{\ell}(t_{2})\|_{\dot{B}^{\frac{d}{p}-1}_{p,1}}\lesssim \|v\|_{L^{\infty}(0,T_{*};\dot{B}^{\frac{d}{p}}_{p,1})}^{\ell}|t_{1}-t_{2}|.
\end{aligned}
\end{equation}
This implies that $u^{\ell}\in \cC([0,T_{*}];\dot{B}^{\frac{d}{p}-1}_{p,1})$. A similar argument leads to $\bv^{\ell}\in \cC([0,T_{*}];\dot{B}^{\frac{d}{p}}_{p,1})$. To deal with the high-frequency part, we consider the decomposition $(u,\bv)^{h}=S_{N_{0}}(u,\bv)^{h}+({\rm Id}-S_{N_{0}})(u,\bv)^{h}$  for some integer $N_{0}$. From $\eqref{JXDZ}$ and the given bounds on  $(u,\bv)$, one can show that  the high-frequency part satisfies $(u,\bv)^{h}\in \cC([0,T_{*}];\dot{B}^{\frac{d}{2}-1}_{2,1})$. It thus follows that $S_{N_{0}}(u,\bv)^{h}\in \cC([0,T_{*}];\dot{B}^{\frac{d}{2}}_{2,1})$ due to Bernstein's inequality. On the other hand, following similar arguments as in Subsection \ref{subsectionhigh}, we have $(u,\bv)^{h}\in \widetilde{L}^{\infty}(0,T_{*};\dot{B}^{\frac{d}{2}}_{2,1})$, and therefore 
\begin{equation}\nonumber
\begin{aligned}
&\|({\rm Id}-S_{N_{0}})(u,\bv)\|_{L^{\infty}_{T_{*}}(\dot{B}^{\frac{d}{2}}_{2,1})}^{h}\lesssim \sum_{j\geq \max\{J_{\var},N_{0}\}-1}2^{\frac{d}{2}j}\sup_{t\in[0,T_{*}]}\|\dot{\Delta}_{j}(u,\bv)\|_{L^2}
\end{aligned}
\end{equation}
can be arbitrarily small as long as $N_{0}$ is chosen to be large enough. Consequently, $(u,\bv)^{h}\in \cC([0,T_{*}];\dot{B}^{\frac{d}{2}}_{2,1})$ holds.

\begin{itemize}
\item \underline{Step 4: Proof of the uniqueness}
\end{itemize}
For a given time $T>0$, let $(u_1,\bv_1)$ and $(u_2,\bv_2)$ be two solutions of System \eqref{JXSys1} in the space $E(T)$ with the same data $(u_0,\bv_0)$. Then, $(U, \bm{V})=(u_1-u_2, \bv_1-\bv_2)$ satisfies 
\begin{equation} \nonumber
\left\{
   \begin{aligned}
&\frac{\partial} {\partial t}U+\sum_{i=1}^{d} \frac{\partial} {\partial x_{i}} V_{i}=0, \\
&\varepsilon^2\frac{\partial} {\partial t} V_{i}
+ A_{i}\frac{\partial} {\partial x_{i}} U+V_{i}= \big(f_{i}(u_1)-f_{i}(u_2)\big),\quad i=1,2,...,d.
    \end{aligned}
  \right.
\end{equation}
\begin{itemize}
\item Case 1: $p\geq 2$.
\end{itemize}

Arguing similarly as in Subsections \ref{subsectionlow-frequency}-\ref{subsectionhigh}, for $t\in(0,T]$, one can infer that
\begin{equation}\label{UninessEst1}
\begin{aligned}
&\|U\|_{L^{\infty}_{t}(\dot{B}^{\frac{d}{p}-1}_{p,1})}^{\ell}
+\|U\|_{L^1_{t}(\dot{B}^{\frac{d}{p}+1}_{p,1})}^{\ell}+\var\|(U, \var \bm{V})\|_{L^{\infty}_{t}(\dot{B}^{\frac{d}{2}}_{2,1})}^{h}
+\frac{1}{\var}\|(U,\var\bm{V})\|_{L^1_{t}(\dot{B}^{\frac{d}{2}}_{2,1})}^{h}\\
&\quad\lesssim \| f(u_1)-f(u_2)\|_{L^1_{t}(\dot{B}^{\frac{d}{p}}_{p,1})}^{\ell}
+\|f(u_1)-f(u_2)\|_{L^1_t(B^{\frac {d}{2}}_{2,1})}^{h}.
\end{aligned}
\end{equation}
We now bound the nonlinear part of \eqref{UninessEst1}. Since $p\geq2$, we have $u_{1},u_{2}\in L^{\infty}(0,T;\dot{B}^{\frac{d}{p}}_{p,1})$ by virtue of Bernstein's inequality. It follows from Corollary \ref{cor0.1} that
\begin{equation}\label{UnipGeq2LowEst1}
\begin{aligned}
\| f(u_1)-f(u_2)\|_{L^1_{t}(\dot{B}^{\frac{d}{p}}_{p,1})}^{\ell}
&\lesssim \int_{0}^{t}\|U\|_{\dot{B}^{\frac{d}{p}}_{p,1}}
\|(u_1,u_2)\|_{\dot{B}^{\frac{d}{p}}_{p,1}}d\tau \\
&\lesssim \int_{0}^{t}  \|(u_1,u_2)\|_{\dot{B}^{\frac{d}{p}}_{p,1}} 
 \Big(\frac{1}{\var}\|U\|_{\dot{B}^{\frac{d}{p}-1}_{p,1}}^\ell
+\|U\|_{\dot{B}^{\frac{d}{2}}_{2,1}}^h\Big)d\tau.
\end{aligned}
\end{equation}
To bound the nonlinear term in the high-frequency regime, we rewrite $f_{i}(u_1)-f_{i}(u_2)=\sum_{j=1}^{d}U_{j} b_{i,j}$ with $b_{i,j}\triangleq \int_0^1 \frac{\partial{f}}{\partial{u_{j}}}(u_2 + \tau (u_1-u_2))\ d\tau$. Therefore,  the product law in Lemma \ref{NonClassicalProLaw1} and the composition estimates in Lemmas \ref{NewSmoothlow} and \ref{NewSmoothhigh} yield
\begin{equation}\label{UniHighpGeq2Est1}
\begin{aligned}
&\|f(u_1)-f(u_2)\|_{L^1_t(B^{\frac {d}{2}}_{2,1})}^{h}\\
&\quad \lesssim 
\int_{0}^{t}\Big(\|U\|_{\dot B_{p,1}^{\frac{d}{p}}}
\sum_{i,j}^{d}\|b_{i,j}\|_{\dot B_{2,1}^{\frac{d}{2}}}^h
+\sum_{i,j}^{d}\|b_{i,j}\|_{\dB_{p,1}^{\frac{d}{p}}}
\|U\|_{\dB_{2,1}^{\frac{d}{2}}}^{h}
+
\|U\|_{\dB_{p,1}^{\frac{d}{p}}}^\ell\Big)d\tau\\
&\quad\lesssim (1+\frac{1}{\var})\int_{0}^{T} \Big(\|(u_1,u_2)\|_{\dB_{p,1}^{\frac{d}{p}}}^{\ell}+\|(u_1,u_2)\|_{\dB_{2,1}^{\frac{d}{2}}}^{h}\Big)\Big(\|U\|_{\dot B_{p,1}^{\frac{d}{p}-1}}^{\ell}+\|U\|_{\dot B_{2,1}^{\frac{d}{2}}}^{h}\Big)d\tau.
\end{aligned}
\end{equation}
Inserting \eqref{UnipGeq2LowEst1} and \eqref{UniHighpGeq2Est1} into \eqref{UninessEst1} and employing Gr\"onwall's inequality yields $(u_1,\bv_1)=(u_2,\bv_2)$ for a.e. $(t,x)\in [0,T]\times \mathbb{R}^{d}$.
\begin{itemize}
\item Case 2: $p< 2$.
\end{itemize}
As $p<2$, it is clear that $\dot B_{p,1}^{\frac{d}{p}}\hookrightarrow \dot B_{2,1}^{\frac{d}{2}}$. Therefore, we only need to prove the uniqueness in the $L^2$ framework. The $L^2$ energy method from Subsections \ref{subsectionlow-frequency}-\ref{subsectionhigh} leads to
\begin{equation}\label{UninessEst2mm}
\begin{aligned}
&\|U\|_{L^{\infty}_{t}(\dot{B}^{\frac{d}{2}-1}_{2,1} )}^{\ell}
+\|U\|_{L^1_{t}(\dot{B}^{\frac{d}{2}+1}_{p,1})}^{\ell}
+\var\|(U, \var \bm{V})\|_{L^{\infty}_{t}(\dot{B}^{\frac{d}{2}}_{2,1})}^{h}
+\frac{1}{\var}\|(U,\var\bm{V})\|_{L^1_{t}(\dot{B}^{\frac{d}{2}}_{2,1})}^{h}\\
&\quad\lesssim \| f(u_1)-f(u_2)\|_{L^1_{t}(\dot{B}^{\frac{d}{2}}_{2,1})}.
\end{aligned}
\end{equation}
The composition estimate in Corollary \ref{cor0.1} gives
\begin{equation}\label{UnipLeq2LowEst1}
\begin{aligned}
\| f(u_1)-f(u_2)\|_{L^1_{t}(\dot{B}^{\frac{d}{2}}_{2,1})}
&\lesssim
\int_{0}^{t}
\|(u_1,u_2)\|_{\dot{B}^{\frac{d}{2}}_{2,1}}\|U\|_{\dot{B}^{\frac{d}{2}}_{2,1}}d\tau\\
& \lesssim \int_{0}^{t}
\|(u_1,u_2)\|_{\dot{B}^{\frac{d}{2}}_{2,1}}(\frac{1}{\var}\|U\|_{\dot{B}^{\frac{d}{2}-1}_{2,1}}^{\ell}+\|U\|_{\dot{B}^{\frac{d}{2}}_{2,1}}^{h})d\tau.
\end{aligned}
\end{equation}
Hence, Gr\"onwall's inequality implies the uniqueness in the case $p<2$. \qed

\section{Strong relaxation limit}\label{section4}

In this section, we prove Theorem \ref{Thm2}. For clarity, we divide the proof into two steps. First, we establish additional regularity estimates of the effective unknown $\bZ$.



\begin{Prop}\label{dgeq2PropRelax}
Let $(u,\bv)$ be the solution to System \eqref{JXSys1} obtained in Theorem \ref{Thm1}. If $p$ is given by \eqref{p2}, then
\begin{equation}\label{additionu}
\begin{aligned}
\|u\|_{\tL^{\infty}(\R^+;\dot{B}^{\frac{d}{p}-1}_{p,1}\cap\dot{B}^{\frac{d}{p}}_{p,1})}+\|u\|_{\tL^2(\R^+;\dot{B}^{\frac{d}{p}}_{p,1})}
\lesssim\cX_{p,0},
\end{aligned}
\end{equation}
where $\cX_{p,0}$ is given by \eqref{small}.

Moreover, under the additional assumption \eqref{error}, we have
\begin{equation}\label{additionz}
\begin{aligned}
\var\|\bZ^{\ell}\|_{\tL^\infty(\R^+; \dot B_{p,1}^{\frac{d}{p}})}
+\frac{1}{\var}\|\bZ^{\ell}\|_{\tL^1(\R^+; \dot B_{p,1}^{\frac{d}{p}})}
\lesssim
\var\|\bv_0^{\ell}\|_{\dot B_{p,1}^{\frac{d}{p}}}
+\cX_{p,0},
\end{aligned}
\end{equation}
where $\bZ$ is defined by \eqref{effective2}.
\end{Prop}
\begin{proof}
Thanks to  \eqref{ThmEst1}, \eqref{HLEst} and Bernstein's inequality, it holds that
\begin{equation}\nonumber
\begin{aligned}
&\|u\|_{\tL^{\infty}(\R^+;\dot{B}^{\frac{d}{p}-1}_{p,1}\cap\dot{B}^{\frac{d}{p}}_{p,1})}\lesssim \|u\|_{\tL^{\infty}(\R^+;\dot{B}^{\frac{d}{p}-1}_{p,1}\cap\dot{B}^{\frac{d}{p}}_{p,1})}^{\ell}+(1+\var)\|u\|_{\tL^{\infty}(\R^+;\dot{B}^{\frac{d}{2}}_{2,1})}^{h}\lesssim \mathcal{X}_{p,0},\\
&\| u \|_{\tL^2(\R^+;\dot{B}^{\frac{d}{p}}_{p,1})}\lesssim \| u \|_{\tL^2(\R^+;\dot{B}^{\frac{d}{p}}_{p,1})}^{\ell}+\| u \|_{\tL^2(\R^+;\dot{B}^{\frac{d}{2}}_{2,1})}^{h}\lesssim \mathcal{X}_{p,0},
\end{aligned}
\end{equation}
where we used that $p\geq2$ and $2^{J_{\var}}\lesssim \var^{-1}$. Consequently, we get \eqref{additionu}.

Next, to establish the low-frequency estimate \eqref{additionz} of $\bZ$, we observe that $\bZ$ has a damping effect in
\begin{equation} \label{JXDZ22}
\begin{aligned}
&\frac{\partial}{\partial t}Z_{k}+\dfrac{1}{\varepsilon^2}Z_{k}
= A_{k} \frac{\partial}{\partial x_{k}}\Big( \sum_{i=1}^{d}\frac{\partial}{\partial x_{i}}(A_{i}  \frac{\partial}{\partial x_{i}} u) \Big)
-A_{k} \frac{\partial}{\partial x_{k}} \sum_{i=1}^{d}\frac{\partial}{\partial x_{i}} Z_{i}\\
&\quad\quad
- A_{k} \frac{\partial}{\partial x_{k}} \Big(\sum_{i=1}^{d}\frac{\partial}{\partial x_{i}} f_{i}(u)\Big)
+\sum_{j=1}^n\frac{\partial}{\partial_{u_{j}}}f_{k}(u) \sum_{i=1}^d \frac{\partial}{\partial x_i}v_{i,j},\quad k=1,2,...,d.
 \end{aligned}
\end{equation}
Applying the low-frequency cut-off operator $\dot{S}_{J_{\var}}$ to \eqref{JXDZ22}, and making use of Lemma \ref{maximaldamped} and the fact that $2^{J_{\var}}\lesssim \var^{-1}$, we obtain
\begin{equation}\label{AdditionalzEst1}
\begin{aligned}
\var\|\bZ^{\ell}\|_{\tL^{\infty}(\R^+;\dot{B}^{\frac{d}{p}}_{p,1})}
+\frac{1}{\var}\|\bZ^{\ell}\|_{\tL^1(\R^+;\dot{B}^{\frac{d}{p}}_{p,1})}&\lesssim 
\var\|\bZ_{0}^{\ell}\|_{\dot{B}^{\frac{d}{p}}_{p,1}}
+ \var\| u^{\ell} \|_{\tL^1(\R^+;\dot{B}^{\frac{d}{p}+3}_{p,1})}
+\var\|f(u)^{\ell}\|_{\tL^1(\R^+;\dot{B}^{\frac{d}{p}+1}_{p,1})}\\
&\quad+\sum_{k,i=1}^{d}\sum_{j=1}^n\var \Big\|\Big(\frac{\partial}{\partial_{u_{j}}}f_{k}(u)  \frac{\partial}{\partial x_i}v_{i,j} \Big)^{\ell}\Big\|_{\tL^1(\R^+;\dot{B}^{\frac{d}{p}}_{p,1})}.
\end{aligned}
\end{equation}
Here $\bZ_{0}=(Z_{0,1}, Z_{0,2},...,Z_{0,d})$ with 
$Z_{0,k}\triangleq A_k\frac{\partial }{\partial x_k} u_{0}+ v_{0,k}+f_k(u_{0})$. By the classical composition estimate in Lemma \ref{DifferComposition}, the first term on the r.h.s of \eqref{AdditionalzEst1} is controlled by
\begin{equation}\label{AdditionalzEst11}
\begin{aligned}
\var\|\bZ_{0}^{\ell}\|_{\dot{B}^{\frac{d}{p}}_{p,1}}&\lesssim \var \|u_{0}\|_{\dot{B}^{\frac{d}{p}+1}_{p,1}}^{\ell}+\var \|\bv_{0}\|_{\dot{B}^{\frac{d}{p}}_{p,1}}^{\ell}+\var \|f(u_{0})\|_{\dot{B}^{\frac{d}{p}+1}_{p,1}}^{\ell}\\
&\lesssim  \|u_{0}\|_{\dot{B}^{\frac{d}{p}}_{p,1}}^{\ell}+\var \|\bv_{0}\|_{\dot{B}^{\frac{d}{p}}_{p,1}}^{\ell}+\|f(u_{0})\|_{\dot{B}^{\frac{d}{p}}_{p,1}}^{\ell}\\
&\lesssim \|u_{0}\|_{\dot{B}^{\frac{d}{p}}_{p,1}}^{\ell}+ \|u_{0}\|_{\dot{B}^{\frac{d}{2}}_{2,1}}^{h}+\var \|\bv_{0}\|_{\dot{B}^{\frac{d}{p}}_{p,1}}^{\ell}.
\end{aligned}
\end{equation}
From \eqref{ThmEst1}, \eqref{HLEst}, \eqref{LNolinerEst11} and $2^{J_{\var}}\lesssim \var^{-1}$, one gets
\begin{equation}\label{AdditionalzEst12}
\begin{aligned}
&\var\| u^{\ell} \|_{\tL^1(\R^+;\dot{B}^{\frac{d}{p}+3}_{p,1})}+\var\|f(u)^{\ell}\|_{\tL^1(\R^+;\dot{B}^{\frac{d}{p}+2}_{p,1})}\\
&\quad\quad\quad\lesssim \|u\|_{\tL^1(\R^+;\dot{B}^{\frac{d}{p}+2}_{p,1})}^{\ell}+\|f(u)\|_{\tL^1(\R^+;\dot{B}^{\frac{d}{p}+1}_{p,1})}^{\ell}\lesssim \mathcal{X}_{p,0}^2.
\end{aligned}
\end{equation}
By virtue of \eqref{HLEst}, \eqref{additionu}, Lemma \ref{ClassicalProductLaw}, Lemma \ref{DifferComposition} and $-\frac{d}{p}\leq \frac{d}{p}-1$ due to $2\leq p\leq 2d$, one also has
\begin{equation}\label{AdditionalzEst13}
\begin{aligned}
\varepsilon \Big\|\Big(\frac{\partial}{\partial_{u_{j}}}f_{k}(u)  \frac{\partial}{\partial x_i}v_{i,j} \Big)^{\ell} \Big\|_{\tL^1(\R^+;\dot{B}^{\frac{d}{p}}_{p,1})}&\lesssim  \Big\|\frac{\partial}{\partial_{u_{j}}}f_{k}(u)  \frac{\partial}{\partial x_i}v_{i,j} \Big\|_{\tL^1(\R^+;\dot{B}^{\frac{d}{p}-1}_{p,\infty})}^{\ell}\\
&\lesssim \|\frac{\partial}{\partial_{u_{j}}}f_{k}(u)\|_{\tL^{\infty}(\R^+;\dot{B}^{\frac{d}{p}}_{p,1})}\|\frac{\partial}{\partial x_i}v_{i,j}\|_{\tL^1(\R^+;\dot{B}^{\frac{d}{p}-1}_{p,\infty})}\\
&\lesssim \|u\|_{\tL^{\infty}(\R^+;\dot{B}^{\frac{d}{p}}_{p,1})}
\Big(\| \bf{v}\|_{\tL^1(\R^+;\dot{B}^{\frac{d}{p}}_{p,1})}^{\ell}
+\| \bf{v}\|_{\tL^1(\R^+;\dot{B}^{\frac{d}{2}}_{2,1})}^{h} \Big)\lesssim  \mathcal{X}_{p,0}^2.
\end{aligned}
\end{equation}
Inserting \eqref{AdditionalzEst11}, \eqref{AdditionalzEst12} and \eqref{AdditionalzEst13} into \eqref{AdditionalzEst1}, we obtain  \eqref{additionz}.
\end{proof}

We are now ready to prove Theorem \ref{Thm2}.

\noindent
\textbf{\emph{Proof of Theorem \ref{Thm2}}.} Let the assumptions \eqref{p}, \eqref{small} and \eqref{error} be in force. To derive the convergence rate,  we shall estimate the difference of the solutions $(u,\bv)$ and $u^{*}$ to Systems \eqref{JXSys1} and \eqref{Thm2uheat}, respectively. One has the estimate  $\|\bZ^{\ell}\|_{\tL^1(\R^+;\dot{B}^{\frac{d}{p}}_{p,1})}\lesssim \var$ due to  \eqref{additionz}. Defining $(\delta u, \delta \bv) \triangleq (u-u^*, \bv-\bv^*)$ and using \eqref{JXDZ}, we have
\begin{equation} \label{diff1}
\left\{
\begin{aligned}
&\partial_t\delta u-\sum_{i=1}^{d}\frac{\partial}{\partial x_{i}}(A_{i}\frac{\partial}{\partial x_{i}} \delta u)=-\sum_{i=1}^{d}\frac{\partial}{\partial x_{i}}Z_{i}
-\sum_{i=1}^{d}\frac{\partial}{\partial x_{i}}\big(f_{i}(u)-f_{i}(u^*) \big),\\
&\delta v_{i}=-A_{i}\dfrac{\partial} {\partial x_{i}}  \delta u+f_{i}(u)-f_{i}(u^*)+Z_{i},\quad i=1,2,...,d
\end{aligned}
\right.
\end{equation}
with $\bZ=(Z_1, Z_2, \cdots, Z_d)$ defined in \eqref{effective2}. In the high-frequency regime, due to $2^{-J_{\var}}\lesssim \var$, the uniform bounds \eqref{limitur} and \eqref{additionu} directly give
\begin{equation}\label{errorhigh}
\left\{
\begin{aligned}
&\|\delta u\|_{\tL^{\infty}(\R^+;\dot{B}^{\frac{d}{p}-1}_{p,1})}^{h}\lesssim \var\|(u,u^{*})\|_{\tL^{\infty}(\R^+;\dot{B}^{\frac{d}{p}}_{p,1})}^{h}\lesssim \var (\|u_{0}^{*}\|_{\dot{B}^{\frac{d}{p}-1}_{p,1}\cap \dot{B}^{\frac{d}{p}}_{p,1}}+\mathcal{X}_{p,0}),\\
&\|\delta \bv\|_{\tL^1(\R^+;\dot{B}^{\frac{d}{p}}_{p,1})}^{h}\lesssim\var\Big( \frac{1}{\var}\|\bf{v}\|_{\tL^{1}(\R^+;\dot{B}^{\frac{d}{p}}_{p,1})}^{h}+\|\bf{v}^{*}\|_{\tL^{1}(\R^+;\dot{B}^{\frac{d}{p}+1}_{p,1})}^{h}\Big)\lesssim \var (\|u_{0}^{*}\|_{\dot{B}^{\frac{d}{p}-1}_{p,1}\cap \dot{B}^{\frac{d}{p}}_{p,1}}+\mathcal{X}_{p,0}).
\end{aligned}
\right.
\end{equation}
In the low-frequency regime, applying Lemma \ref{HeatRegulEstprop} to $\eqref{diff1}_{1}$, we deduce
\begin{equation*}
\begin{aligned}
&\|\delta u^{\ell}\|_{\tL^{\infty}(\R^+;\dot{B}^{\frac{d}{p}-1}_{p,1})}
+\|\delta u^{\ell}\|_{\tL^1(\R^+;\dot{B}^{\frac{d}{p}+1}_{p,1})}\\
&\quad\lesssim \|u_{0}-u_{0}^{*}\|_{\dot{B}^{\frac{d}{p}-1}_{p,1}}
+\|\bZ^{\ell}\|_{\tL^1(\R^+;\dot{B}^{\frac{d}{p}}_{p,1})}
+\| \big(f(u)-f(u^*)\big)^{\ell}\|_{\tL^1(\R^+;\dot{B}^{\frac{d}{p}}_{p,1})}
,
\end{aligned}
\end{equation*}
which, together with \eqref{HLEst11} and the key uniform bound $\|\bZ^{\ell}\|_{\tL^1(\R^+;\dot{B}^{\frac{d}{p}}_{p,1})}\lesssim \var$ in  \eqref{additionz}, leads to
\begin{equation}\label{deltauEst}
\begin{aligned}
&\|\delta u^{\ell}\|_{\tL^{\infty}(\R^+;\dot{B}^{\frac{d}{p}-1}_{p,1})}
+\|\delta u^{\ell}\|_{\tL^1(\R^+;\dot{B}^{\frac{d}{p}+1}_{p,1})}
\lesssim \|u_{0}-u_{0}^{*}\|_{\dot{B}^{\frac{d}{p}-1}_{p,1}}+\var
+\| f(u)-f(u^*)\|_{\tL^1(\R^+;\dot{B}^{\frac{d}{p}}_{p,1})}.
\end{aligned}
\end{equation}
It follows from \eqref{limitur}, $\eqref{L2}_{2}$, \eqref{additionu}, the interpolation in Lemma \ref{Classical Interpolation} and the composition estimate \eqref{corIneq1} in Corollary \ref{cor0.1} that
\begin{equation}\label{RelaxfEst}
\begin{aligned}
&\| f(u)-f(u^{*})\|_{\tL^1(\R^+;\dot{B}^{\frac{d}{p}}_{p,1})}\\
&\quad \lesssim
\|(u,u^{*})\|_{\tL^2(\R^+;\dot{B}^{\frac{d}{p}}_{p,1})}\| \delta u \|_{\tL^2(\R^+;\dot{B}^{\frac{d}{p}}_{p,1})}\\
&\quad\lesssim
\|(u,u^{*})\|_{\tL^2(\R^+;\dot{B}^{\frac{d}{p}}_{p,1})}\Big(\| \delta u^{\ell} \|_{\tL^2(\R^+;\dot{B}^{\frac{d}{p}}_{p,1})}+\| u \|_{\tL^2(\R^+;\dot{B}^{\frac{d}{2}}_{2,1})}^{h}+\var \| u^{*} \|_{\tL^2(\R^+;\dot{B}^{\frac{d}{p}+1}_{p,1})}^{h}\Big)\\
&\quad\lesssim \Big(\cX_{p,0}+\|u_{0}^{*}\|_{\dot{B}^{\frac{d}{p}-1}_{p,1}\cap\dot{B}^{\frac{d}{p}}_{p,1}}\Big)\Big(\|\delta u^{\ell}\|_{\tL^{\infty}(\R^+;\dot{B}^{\frac{d}{p}-1}_{p,1})}
+\|\delta u^{\ell}\|_{\tL^1(\R^+;\dot{B}^{\frac{d}{p}}_{p,1})}\Big)\\
&\quad\quad+\Big(\cX_{p,0}+\|u_{0}^{*}\|_{\dot{B}^{\frac{d}{p}-1}_{p,1}\cap\dot{B}^{\frac{d}{p}}_{p,1}}\Big)^2\var.
\end{aligned}
\end{equation}
Inserting \eqref{RelaxfEst} into \eqref{deltauEst} and using that both $\cX_{p,0}$ and $\|u_{0}^{*}\|_{\dot{B}^{\frac{d}{p}-1}_{p,1}\cap\dot{B}^{\frac{d}{p}}_{p,1}}$  are suitably small, we obtain 
\begin{equation}\label{fffffffff}
\begin{aligned}
&\|\delta u^{\ell}\|_{\tL^{\infty}(\R^+;\dot{B}^{\frac{d}{p}-1}_{p,1})}
+\|\delta u^{\ell}\|_{\tL^1(\R^+;\dot{B}^{\frac{d}{p}+1}_{p,1})}
\lesssim\|u_{0}-u_{0}^{*}\|_{\dot{B}^{\frac{d}{p}-1}_{p,1}}+
\var.
\end{aligned}
\end{equation}
Thanks to \eqref{additionz}, $\eqref{diff1}_{2}$, \eqref{RelaxfEst} and \eqref{fffffffff}, we  recover the information on $\delta \bv$ in low frequencies as follows:
\begin{equation}\label{ffff}
\begin{aligned}
\|\delta \bv^{\ell}\|_{\tL^1(\R^+;\dot{B}^{\frac{d}{p}}_{p,1})}
&\lesssim
\|\delta u^{\ell}\|_{\tL^1(\R^+;\dot{B}^{\frac{d}{p}+1}_{p,1})}
+ \|f(u)-f(u^*)\|_{\tL^1(\R^+;\dot{B}^{\frac{d}{p}}_{p,1})}
+ \|\bZ^{\ell}\|_{\tL^1(\R^+;\dot{B}^{\frac{d}{p}}_{p,1})}\\
&\lesssim\|u_{0}-u_{0}^{*}\|_{\dot{B}^{\frac{d}{p}-1}_{p,1}}+
\var.
\end{aligned}
\end{equation}
Hence, by \eqref{errorhigh}, \eqref{fffffffff} and \eqref{ffff}, we obtain $\eqref{rate}$. With the help of $\eqref{rate}$, one can show the global convergence from \eqref{JXSys1} to \eqref{Thm2uheat}-\eqref{Thm2vdarcy} in $\mathcal{S}'(\mathbb{R}_+\times\mathbb{R}^d)$ as $\var\rightarrow 0$. \qed

\section{Uniform time asymptotics}\label{section5}

\subsection{Time-decay of the solution}\label{Decay of the solution}

\begin{Prop}\label{propdecay}
Assume that $p$ satisfies \eqref{p2} and $(u_{0},  \bv_{0})$ satisfies \eqref{small} and \eqref{a2}. Let $(u,\bv)$ be the global solution to System \eqref{JXSys1} subject to the initial data $(u_{0},  \bv_{0})$ obtained in Theorem \ref{Thm1}. Then,  it holds that
\begin{equation}\label{decayuv}
\begin{aligned}
&\|(u,\var \bv)(t)\|_{\dot{B}^{\sigma}_{p,1}}\lesssim (1+t)^{-\frac{1}{2}(\sigma-\sigma_{1})}\mathcal{D}_{p,0},\quad\quad \sigma_{1}<\sigma\leq \frac{d}{p}
\end{aligned}
\end{equation}
for $\mathcal{D}_{p,0}=\|(u_{0}^{\ell},\var  \bv_{0}^{\ell})\|_{\dot{B}^{\sigma_{1}}_{p,\infty}}+\mathcal{X}_{p,0}$.
\end{Prop}

The proof of Proposition \ref{propdecay} is divided into the following four steps.

\begin{itemize}
\item \underline{{\emph{ Step 1: Estimates of the solution in $\dot{B}^{\sigma_{1}}_{p,\infty}$}}}
\end{itemize}

We first establish the uniform evolution of the $\dot{B}^{\sigma_{1}}_{p,\infty}$-regularity.
    
\begin{Lemma}\label{Decayprop1}
 Let $p$ be given by \eqref{p2}, and $(u,\bv)$ be the global solution to \eqref{JXSys1} satisfying \eqref{ThmEst1} for $\var\in(0,1)$. In addition to \eqref{small}, further assume that \eqref{a2} holds. Then,  we have
\begin{equation}\label{sigma1}
\begin{aligned}
&\|u^{\ell}\|_{\tL^{\infty}(\R^+;\dot{B}^{\sigma_1}_{p,\infty})}
+\|u^{\ell}\|_{\tL^1(\R^+;\dot{B}^{\sigma_1+2}_{p,\infty})}
+\var\|\bZ^{\ell}\|_{\tL^{\infty}(\R^+;\dot{B}^{\sigma_1}_{p,\infty})}
+\frac{1}{\var}\|\bZ^{\ell}\|_{\tL^1(\R^+;\dot{B}^{\sigma_1}_{p,\infty})}\lesssim \mathcal{D}_{p,0}
\end{aligned}
\end{equation}
with $\bZ$ defined by \eqref{effective2} and $\mathcal{D}_{p,0}=\|(u_{0}^{\ell},\var \bv_{0}^{\ell})\|_{\dot{B}^{\sigma_1}_{p,\infty}}
+\cX_{p,0}$. 
\end{Lemma}
\begin{proof}
We recall that $(u,\bZ)$ satisfies \eqref{JXDZ2} and \eqref{JXDZ22}. Applying the low-frequency cut-off operator $\dot{S}_{J_{\var}}$ to $\eqref{JXDZ2}$ and \eqref{JXDZ22} yields
\begin{equation} \label{JXDZ2low}
\left\{
\begin{aligned}
&\frac{\partial}{\partial t}u^{\ell}-\sum_{i=1}^{d}\frac{\partial}{\partial x_{i}}(A_{i} \frac{\partial}{\partial x_{i}} u^{\ell})
=-\sum_{i=1}^{d}\frac{\partial}{\partial x_{i}} Z^{\ell}_{i}
+\sum_{i=1}^{d}\frac{\partial}{\partial x_{i}}  \dot{S}_{J_{\var}}f_{i}(u),\\
&\frac{\partial}{\partial t}Z^{\ell}_{k}+\dfrac{1}{\varepsilon^2}Z^{\ell}_{k}
= A_{k} \frac{\partial}{\partial x_{k}}\Big( \sum_{i=1}^{d}\frac{\partial}{\partial x_{i}}(A_{i}  \frac{\partial}{\partial x_{i}} u^{\ell}) \Big)
-A_{k} \frac{\partial}{\partial x_{k}} \sum_{i=1}^{d}\frac{\partial}{\partial x_{i}} Z^{\ell}_{i}\\
&\quad\quad\quad\quad\quad\quad\quad
-A_{k} \frac{\partial}{\partial x_{k}}\dot{S}_{J_{\var}} \Big(\sum_{i=1}^{d}\frac{\partial}{\partial x_{i}} f_{i}(u)\Big)+ \dot{S}_{J_{\var}}\Big(\sum_{j=1}^n\frac{\partial}{\partial_{u_{j}}}f_{k}(u) \sum_{i=1}^d \frac{\partial}{\partial x_i}v_{i,j}\Big),\quad k=1,2,...,d.
 \end{aligned}
 \right.
\end{equation}
Applying Lemma \ref{HeatRegulEstprop} to $\eqref{JXDZ2low}_{1}$ yields
\begin{equation}\nonumber
\begin{aligned}
&\|u^{\ell}\|_{\tL^{\infty}(\R^+;\dot{B}^{\sigma_1}_{p,\infty})}
+\|u^{\ell}\|_{\tL^1(\R^+;\dot{B}^{\sigma_1+2}_{p,\infty})}
\lesssim
\|u_0^{\ell}\|_{\dot{B}^{\sigma_1}_{p,\infty}}
+
2^{J_{\var}}\|\bZ^{\ell}\|_{\tL^1(\R^+;\dot{B}^{\sigma_1}_{p,\infty})}
+
\|f(u)\|_{\tL^1(\R^+;\dot{B}^{\sigma_1+1}_{p,\infty})}^{\ell}.
\end{aligned}
\end{equation}
In addition, one deduces from applying Lemma \ref{maximaldamped} to
$\eqref{JXDZ2low}_2$ that
\begin{equation}\nonumber
\begin{aligned}
&\var\|\bZ^{\ell}\|_{\tL^{\infty}(\R^+;\dot{B}^{\sigma_1}_{p,\infty})}
+\frac{1}{\var}\|\bZ^{\ell}\|_{\tL^1(\R^+;\dot{B}^{\sigma_1}_{p,\infty})}\\
&\quad\lesssim
\var\|\bZ_{0}^{\ell}\|_{\dot{B}^{\sigma_1}_{p,\infty}}
+\var 2^{J_{\var}}
\|u^{\ell}\|_{\tL^1(\R^+;\dot{B}^{\sigma_1+2}_{p,\infty})}
+\var 2^{J_{\var}}
\|\bZ^{\ell}\|_{\tL^1(\R^+;\dot{B}^{\sigma_1}_{p,\infty})}
\\
&\quad\quad+
\var\|f(u)\|_{\tL^1(\R^+;\dot{B}^{\sigma_1+2}_{p,\infty})}^{\ell}
+ 
\sum_{k,j=1}^{d}\sum_{j=1}^n \var \Big\|\frac{\partial}{\partial_{u_{j}}}f_{k}(u)\frac{\partial}{\partial x_i}v_{i,j}\Big\|_{\tL^1(\R^+;\dot{B}^{\sigma_1}_{p,\infty})}^{\ell},
\end{aligned}
\end{equation}
where $\bZ_{0}=(Z_{0,1},Z_{0,2},...,Z_{0,d})$ with 
$Z_{0,k}\triangleq A_k\frac{\partial }{\partial x_k} u_{0}+ v_{0,k}+f_k(u_{0})$. Since $2^{J_{\var}}\leq 2^{k_{0}}\var^{-1}$ with $k_{0}$ small enough, we have
\begin{equation}\label{DecayNegNormEst3}
\begin{aligned}
&\|u^{\ell}\|_{\tL^{\infty}(\R^+;\dot{B}^{\sigma_1}_{p,\infty})}
+\|u^{\ell}\|_{\tL^1(\R^+;\dot{B}^{\sigma_1+2}_{p,\infty})}
+\var\|\bZ^{\ell}\|_{\tL^{\infty}(\R^+;\dot{B}^{\sigma_1}_{p,\infty})}
+\frac{1}{\var}\|\bZ^{\ell}\|_{\tL^1(\R^+;\dot{B}^{\sigma_1}_{p,\infty})}\\
&\quad\lesssim
\|u_0^{\ell}\|_{\dot{B}^{\sigma_1}_{p,\infty}}
+
\var\|\bZ_{0}^{\ell}\|_{\dot{B}^{\sigma_1}_{p,\infty}}
+
\|f(u)\|_{\tL^1(\R^+;\dot{B}^{\sigma_1+1}_{p,\infty})}^{\ell}
+
\sum_{k,i=1}^{d}\sum_{j=1}^n \var \Big\|\frac{\partial}{\partial_{u_{j}}}f_{k}(u)\frac{\partial}{\partial x_i}v_{i,j}\Big\|_{\tL^1(\R^+;\dot{B}^{\sigma_1}_{p,\infty})}^{\ell}.
\end{aligned}
\end{equation}
Now, we bound the terms on the r.h.s of \eqref{DecayNegNormEst3}. We denote by $\widetilde{f}_{i}$ the smooth function such that $f_{i}(u)=\widetilde{f}_{i}(u)u$ and $\widetilde{f}_{i}(0)=0$. Due to the facts that $2\leq p\leq 2d$, $ -\frac{d}{p}\leq \sigma_1\leq \frac{d}{p}-1$, $\var\leq 1$ and $2^{J_{\var}}\sim \var^{-1}$, one  deduces from  \eqref{HLEst}, \eqref{ClassicalProductLawEst3} and Lemma \ref{DifferComposition}  that
\begin{equation}\nonumber
\begin{aligned}
\|f_{i}(u_{0})\|_{\dot{B}^{\sigma_1}_{p,\infty}}&\lesssim \|\widetilde{f}_{i}(u_{0})\|_{\dot{B}^{\frac{d}{p}}_{p,1}}\|u_{0}\|_{\dot{B}^{\sigma_1}_{p,\infty}}\lesssim (\|u_{0}\|_{\dot{B}^{\frac{d}{p}}_{p,1}}^{\ell}+ \|u_{0}\|_{\dot{B}^{\frac{d}{2}}_{2,1}}^{h})(\|u_{0}\|_{\dot{B}^{\sigma_1}_{p,\infty}}^{\ell}+\var^{\frac{d}{2}-\sigma_{1}}\|u_{0}\|_{\dot{B}^{\frac{d}{2}}_{2,1}}^{h}).
\end{aligned}
\end{equation}
This implies that
\begin{equation}\label{pinfty1}
\begin{aligned}
\var\|\bZ_{0}^{\ell}\|_{\dot{B}^{\sigma_1}_{p,\infty}}\lesssim \var \|\bv^{\ell}_{0}\|_{\dot{B}^{\sigma_1}_{p,\infty}}+\var\|\nabla u^{\ell}_{0}\|_{\dot{B}^{\sigma_1}_{p,\infty}}+\var\|f(u_{0})\|_{\dot{B}^{\sigma_1}_{p,\infty}}
^{\ell}\lesssim \|(u_{0}^{\ell},\var \bv_{0}^{\ell})\|_{\dot{B}^{\sigma_1}_{p,\infty}}+\mathcal{X}_{p,0}.
\end{aligned}
\end{equation}
It also holds from \eqref{ThmEst1}, \eqref{HLEst}, Lemma \ref{DifferComposition} and \eqref{ClassicalProductLawEst3} with $2\leq p\leq 2d$ that
\begin{equation}\nonumber
\begin{aligned}
\| f(u)\|_{\tL^1(\R^+;\dot{B}^{\sigma_1+1}_{p,\infty})}^{\ell}&\sim \|\nabla f(u)\|_{\tL^1(\R^+;\dot{B}^{\sigma_1}_{p,\infty})}\\
&\lesssim \sum_{j=1}^{n}\Big\|\frac{\partial}{\partial u_{j}}f(u) \Big\|_{\tL^2(\R^+;\dot{B}^{\frac{d}{p}}_{p,1})}\|\nabla u\|_{\tL^2(\R^+;\dot{B}^{\sigma_1}_{p,\infty})}\\
&\lesssim \|u\|_{\widetilde{L}^2(\R^+;\dot{B}^{\frac{d}{p}}_{p,1})}\Big(\| u^{\ell}\|_{\tL^2(\R^+;\dot{B}^{\sigma_1+1}_{p,\infty})}+\var^{\frac{d}{p}-1-\sigma_{1}}\|u\|_{\tL^2(\R^+;\dot{B}^{\frac{d}{2}}_{2,1})}^{h}\Big),
\end{aligned}
\end{equation}
which, together with the uniform bound \eqref{ThmEst1} and Corollary \ref{Interpolation}, yields
\begin{equation}\label{pinfty2}
\begin{aligned}
\var\| f(u)\|_{\tL^1(\R^+;\dot{B}^{\sigma_1+2}_{p,\infty})}^{\ell}\lesssim \mathcal{X}_{p,0}\Big(\|u^{\ell}\|_{\tL^{\infty}(\R^+;\dot{B}^{\sigma_1}_{p,\infty})}
+\|u^{\ell}\|_{\tL^1(\R^+;\dot{B}^{\sigma_1+2}_{p,\infty})}\Big)+\mathcal{X}_{p,0}.
\end{aligned}
\end{equation}
Similarly, the last term on the r.h.s of  \eqref{DecayNegNormEst3} is estimated as follows
\begin{align}
&\var \Big\|\frac{\partial}{\partial_{u_{j}}}f_{k}(u)\frac{\partial}{\partial x_i}v_{i,j}\Big\|_{\tL^1(\R^+;\dot{B}^{\sigma_1}_{p,\infty})}^{\ell}\nonumber\\
&\quad\lesssim \var \Big\|\frac{\partial}{\partial_{u_{j}}}f_{k}(u)\frac{\partial}{\partial x_i}(v_{i,j})^{\ell}\Big\|_{\tL^1(\R^+;\dot{B}^{\sigma_1}_{p,\infty})}^{\ell}+\var \Big\|\frac{\partial}{\partial_{u_{j}}}f_{k}(u)\frac{\partial}{\partial x_i}(v_{i,j})^{h}\Big\|_{\tL^1(\R^+;\dot{B}^{\sigma_1}_{p,\infty})}^{\ell}\nonumber\\
&\quad\lesssim \var\|u\|_{\tL^{\infty}(\R^+;\dot{B}^{\sigma_{1}}_{p,\infty})}\|\nabla v\|_{\tL^1(\R^+;\dot{B}^{\frac{d}{p}}_{p,1})}^{\ell}+\var\|u\|_{\tL^{\infty}(\R^+;\dot{B}^{\frac{d}{p}}_{p,1})}\|\nabla v\|_{\tL^1(\R^+;\dot{B}^{\sigma_{1}}_{p,\infty})}^{h}\nonumber\\
&\quad\lesssim \var\Big(\|u^{\ell}\|_{\tL^{\infty}(\R^+;\dot{B}^{\sigma_{1}}_{p,\infty})}+\var^{\frac{d}{p}-\var}\|u\|_{\tL^{\infty}(\R^+;\dot{B}^{\frac{d}{2}}_{2,1})}^{h}\Big)\|\bf{v}\|_{\tL^1(\R^+;\dot{B}^{\frac{d}{p}+1}_{p,1})}^{\ell}\nonumber\\
&\quad\quad+\|u\|_{\tL^{\infty}(\R^+;\dot{B}^{\frac{d}{p}}_{p,1})}\var^{\frac{d}{p}-1}\| \bf{v}\|_{\tL^1(\R^+;\dot{B}^{\frac{d}{2}}_{2,1})}^{h}\nonumber\\
&\quad\lesssim \mathcal{X}_{p,0}\|u^{\ell}\|_{\tL^{\infty}(\R^+;\dot{B}^{\sigma_{1}}_{p,\infty})}+\mathcal{X}_{p,0},\label{pinfty3}
\end{align}
where one used $\var\leq1$. Hence, it follows from \eqref{DecayNegNormEst3}-\eqref{pinfty3} that
\begin{equation}\nonumber
\begin{aligned}
&\|u^{\ell}\|_{\tL^{\infty}(\R^+;\dot{B}^{\sigma_1}_{p,\infty})}
+\|u^{\ell}\|_{\tL^1(\R^+;\dot{B}^{\sigma_1+2}_{p,\infty})}
+\var\|\bZ^{\ell}\|_{\tL^{\infty}(\R^+;\dot{B}^{\sigma_1}_{p,\infty})}
+\frac{1}{\var}\|\bZ^{\ell}\|_{\tL^1(\R^+;\dot{B}^{\sigma_1}_{p,\infty})}\\
&\quad\lesssim \mathcal{X}_{p,0}+\mathcal{X}_{p,0}(\|u^{\ell}\|_{\tL^{\infty}(\R^+;\dot{B}^{\sigma_1}_{p,\infty})}
+\|u^{\ell}\|_{\tL^1(\R^+;\dot{B}^{\sigma_1+2}_{p,\infty})}),
\end{aligned}
\end{equation}
which, combined with the smallness of $\mathcal{X}_{p,0}$, gives \eqref{sigma1}. The proof of Lemma \ref{Decayprop1} is complete.
\end{proof}

\begin{itemize}
\item \underline{{\emph{ Step 2: {Time-weighted estimates in the low-frequency regime}}}}
\end{itemize}

Inspired by \cite{XinXu,BaratShou2,lhl1}, we perform time-weighted energy estimates for $(u,\bZ)$. Let $\alpha>1$ be any given constant. Multiplying $\eqref{JXDZ2low}$  by $t^{\alpha}$, we obtain
\begin{equation} \label{JXDZt}
\left\{
\begin{aligned}
&\frac{\partial}{\partial t}(t^\alpha u^{\ell})-\sum_{i=1}^{d}\frac{\partial}{\partial x_{i}}\Big(A_{i} \frac{\partial}{\partial x_{i}} (t^{\alpha} u^{\ell})\Big)=\alpha t^{\alpha-1}u^{\ell}-\sum_{i=1}^{d}\frac{\partial}{\partial x_{i}} (t^{\alpha} Z_{i})
+t^{\alpha}\sum_{i=1}^{d}\frac{\partial}{\partial x_{i}}  \dot{S}_{J_{\var}}f_{i}(u),\\
&\frac{\partial}{\partial t}(t^{\alpha} Z_{k}^{\ell})+\dfrac{1}{\varepsilon^2}(t^{\alpha} Z_{k}^{\ell})= \alpha t^{\alpha-1}Z_{k}^{\ell}+ A_{k} \frac{\partial}{\partial x_{k}}\Big( \sum_{i=1}^{d}\frac{\partial}{\partial x_{i}}\Big(A_{i}  \frac{\partial}{\partial x_{i}} (t^{\alpha} u^{\ell}) \Big) \Big)\\
&\quad\quad-A_{k} \frac{\partial}{\partial x_{k}} \sum_{i=1}^{d}\frac{\partial}{\partial x_{i}} (t^{\alpha} Z^{\ell}_{i})- t^{\alpha} A_{k} \frac{\partial}{\partial x_{k}} \Big(\sum_{i=1}^{d}\frac{\partial}{\partial x_{i}} \dot{S}_{J_{\var}}f_{i}(u)\Big)
\\
&\quad\quad+t^{\alpha} \dot{S}_{J_{\var}}\Big(\sum_{j=1}^n\frac{\partial}{\partial_{u_{j}}}f_{k}(u) \sum_{i=1}^d \frac{\partial}{\partial x_i}v_{i,j}\Big)
 \end{aligned}
 \right.
\end{equation}
for $k=1,2,...,d$.

By similar arguments as in Subsection \ref{subsectionlow-frequency}, one deduces from Lemmas \ref{HeatRegulEstprop}-\ref{maximaldamped} applied to \eqref{JXDZt} and $2^{J_{\var}}\leq 2^{k_{0}}\var^{-1}$ with $2^{k_{0}}\ll1$ that 
\begin{equation}\label{DecayLow}
\begin{aligned}
&\|\tau^\alpha u^{\ell}\|_{\tL^{\infty}_{t}(\dot{B}^{\frac{d}{p}}_{p,1})}
+\|\tau^\alpha u^{\ell}\|_{\tL^1_{t}(\dot{B}^{\frac{d}{p}+2}_{p,1})}
+\var\|\tau^\alpha \bZ^{\ell}\|_{\tL^{\infty}_{t}( \dot{B}^{\frac{d}{p}}_{p,1})}
+\frac{1}{\var}\|\tau^\alpha \bZ^{\ell}\|_{\tL^1_{t}(\dot{B}^{\frac{d}{p}}_{p,1})}\\
&\quad\lesssim \int_0^t  \tau^{\alpha-1}\|u^{\ell}\|_{ \dot{B}^{\frac{d}{p}}_{p,1}}\ d\tau+\var\int_0^t  \tau^{\alpha-1}\|\bZ^{\ell}\|_{ \dot{B}^{\frac{d}{p}}_{p,1}}\ d\tau\\
&\quad\quad+\|\tau^\alpha  f(u)\|_{\tL^1_{t}( \dot{B}^{\frac{d}{p}+1}_{p,1})}^{\ell}+
\sum_{k,i=1}^{d}\sum_{j=1}^n\var \Big\|\tau^{\alpha}\frac{\partial}{\partial_{u_{i}}}f_{k}(u)\frac{\partial}{\partial x_i}v_{i,j}\Big\|_{\tL^1_{t}(\dot{B}^{\sigma_1}_{p,1})}^{\ell}.
\end{aligned}
\end{equation}
The key ingredient for the gain of decay rates is to control the first term on the r.h.s of \eqref{DecaytW} by time-space interpolation arguments. Let 
$\theta=\frac{2}{d/p+2-\sigma_{1}}$ such that $d/p=\sigma_{1}\theta+(d/p+2)(1-\theta)$. Then, applying the interpolation inequality in Lemma \ref{Classical Interpolation} yields
\begin{equation}\label{DecayInter1}
\begin{aligned}
 \int_{0}^{t}\tau^{\alpha-1}\|u^{\ell}\|_{\dot{B}^{\frac{d}{p}}_{p,1}}\ d\tau
 &\lesssim \int_{0}^{t}\tau^{\alpha-1} (\|u^{\ell}\|_{\dot{B}^{\sigma_{1}}_{p,\infty}})^{\theta}( \|u^{\ell}\|_{\dot{B}^{\frac{d}{p}+2}_{p,\infty}})^{1-\theta}
 \ d\tau\\
 &\lesssim \Big(t^{\alpha-\frac{1}{2}(\frac{d}{p}-\sigma_{1})}\|u^{\ell}\|_{L^{\infty}_{t}(\dot{B}^{\sigma_{1}}_{p,\infty})}\Big)^{\theta} \|\tau^{\alpha}u^{\ell}\|_{L^1_{t}(\dot{B}^{\frac{d}{p}+2}_{p,1})}^{1-\theta}\\
 &\lesssim \kappa \|\tau^{\alpha}u^{\ell}\|_{\tL^1_{t}(\dot{B}^{\frac{d}{p}+2}_{p,1})}
 + \frac{1}{\kappa} t^{\alpha-\frac{1}{2}(\frac{d}{p}-\sigma_{1})}\|u^{\ell}\|_{\tL^{\infty}_{t}(\dot{B}^{\sigma_{1}}_{p,\infty})}
\end{aligned}
\end{equation}
for some small constant $\kappa>0$ to be chosen.  One also has
\begin{equation}\label{DecayInter2}
\begin{aligned}
\var \int_{0}^{t}\tau^{\alpha-1}\|\bZ^{\ell}\|_{\dot{B}^{\frac{d}{p}}_{p,1}}\ d\tau
&\lesssim \var \Big( t^{\alpha-\frac{1}{2}(\frac{d}{p}-\sigma_{1})}\|\bZ^{\ell}\|_{\tL^{\infty}_{t}(\dot{B}^{\frac{d}{p}}_{p,1})} \Big)^{\theta}\|\tau^{\alpha}\bZ^{\ell}\|_{\tL^1_{t}(\dot{B}^{\frac{d}{p}}_{p,1})}^{1-\theta}\\
&\lesssim \var^{2\theta}\Big( t^{\alpha-\frac{1}{2}(\frac{d}{p}-1-\sigma_{1})}\|\bZ^{\ell}\|_{\tL^{\infty}_{t}(\dot{B}^{\sigma_{1}}_{p,\infty})} \Big)^{\theta}\Big(\frac{1}{\var}\|\tau^{\alpha}\bZ^{\ell}\|_{\tL^1_{t}(\dot{B}^{\frac{d}{p}}_{p,1})}\Big)^{1-\theta}\\
 &\lesssim  \frac{\kappa}{\var}\|\tau^{\alpha}\bZ^{\ell}\|_{\tL^1_{t}(\dot{B}^{\frac{d}{p}}_{p,1})}+ \frac{1}{\kappa} t^{\alpha-\frac{1}{2}(\frac{d}{p}-\sigma_{1})}\var \|\bZ^{\ell}\|_{\tL^{\infty}_{t}(\dot{B}^{\sigma_{1}}_{p,\infty})},
\end{aligned}
\end{equation}
where we used $\var\leq 1$ and 
$$
\|\bZ^{\ell}\|_{\tL^{\infty}_{t}(\dot{B}^{\frac{d}{p}}_{p,1})} \lesssim 2^{J_{\var}(\frac{d}{p}-\sigma_{1})}\|\bZ^{\ell}\|_{\tL^{\infty}_{t}(\dot{B}^{\sigma_{1}}_{p,\infty})}\lesssim \var^{-(\frac{d}{p}-\sigma_{1})} \|\bZ^{\ell}\|_{\tL^{\infty}_{t}(\dot{B}^{\sigma_{1}}_{p,\infty})}.
$$
Concerning the estimation of the nonlinear terms, from \eqref{ThmEst1} and Lemma \ref{LemNewSmoothlow} with $(s,\sigma)=(d/p+1,d/2)$, we have
\begin{equation}\label{LNolinerEst1t}
\begin{aligned}
&\|\tau^\alpha  f(u)\|_{\tL_t^1(\dot{B}^{\frac{d}{p}+1}_{p,1})}^{\ell}\\
&\quad\lesssim \int_0^t \Big( \big(\tau^\alpha\|u^{\ell}\|_{\dot B_{p,1}^{\frac{d}{p}} } 
+\tau^\alpha\|u\|_{\dot B_{2,1}^{\frac{d}{2}} }^h  \big)\|u\|_{\dot{B}^{\frac{d}{p}+1}_{p,1}}^{\ell}
+
2^{J_{\var}}\big(2^{J_\var}\|u^{\ell}\|_{\dot B_{p,1}^{\frac{d}{p}-1}}+ 
\|u\|_{\dot B_{2,1}^{\frac{d}{2}}}^h\big)
\tau^\alpha\|u\|_{\dot B_{2,1}^{\frac{d}{2}}}^{h} \Big)\  d\tau\\
&\quad\lesssim
\| \tau^\alpha u^{\ell}\|_{\tL^{\infty}_{t}(\dot{B}^{\frac{d}{p}}_{p,1})}+\|  \tau^\alpha u\|_{\tL^{\infty}_{t}(\dot{B}^{\frac{d}{2}}_{2,1})}^{h})\| u\|_{\tL^1_{t}(\dot{B}^{\frac{d}{p}+1}_{p,1})}^{\ell}\\
&\quad\quad+\|u\|_{\tL^{\infty}_{t}(\dot{B}^{\frac{d}{p}-1}_{p,1})}^{\ell}\frac{1}{\var^2}\|\tau^{\alpha} u\|_{\tL^1_{t}(\dot{B}^{\frac{d}{2}}_{2,1})}^{h}+\|  u\|_{\tL^{\infty}_{t}(\dot{B}^{\frac{d}{2}}_{2,1})}^{h}\frac{1}{\var}\|\tau^{\alpha} u\|_{\tL^1_{t}(\dot{B}^{\frac{d}{2}}_{2,1})}^{h}\\
&\quad\lesssim \mathcal{X}_{p,0}(\| \tau^\alpha u^{\ell}\|_{\tL^{\infty}_{t}(\dot{B}^{\frac{d}{p}}_{p,1})}+\|  \tau^\alpha u\|_{\tL^{\infty}_{t}(\dot{B}^{\frac{d}{2}}_{2,1})}^{h}+\frac{1}{\var^2}\|\tau^{\alpha} u\|_{\tL^1_{t}(\dot{B}^{\frac{d}{2}}_{2,1})}^{h}).
\end{aligned}
\end{equation}
Due to $p\geq 2$ and $d/p-1\geq -d/p$, one infers from \eqref{ThmEst1}, \eqref{HLEst11}, \eqref{ClassicalProductLawEst3}, Lemma \ref{DifferComposition}  and Bernstein's inequality that
\begin{equation}\label{LNolinerEst1t1}
\begin{aligned}
\var \Big\|\tau^{\alpha}\frac{\partial}{\partial_{u_{j}}}f_{k}(u)\frac{\partial}{\partial x_i}v_{i,j}\Big\|_{\tL^1_{t}(\dot{B}^{\frac{d}{p}}_{p,1})}^{\ell}&\lesssim \Big\|\tau^{\alpha}\frac{\partial}{\partial_{u_{j}}}f_{k}(u)\frac{\partial}{\partial x_i}v_{i,j}\Big\|_{\tL^1_{t}(\dot{B}^{\frac{d}{p}-1}_{p,\infty})}^{\ell}\\
&\lesssim \|\tau^{\alpha} u\|_{\tL^{\infty}_{t}(\dot{B}^{\frac{d}{p}}_{p,1})} \|\nabla \bf{v}\|_{\tL^{1}_{t}(\dot{B}^{\frac{d}{p}-1}_{p,\infty})}\\
&\lesssim ( \|\tau^{\alpha} u^{\ell}\|_{\tL^{\infty}_{t}(\dot{B}^{\frac{d}{p}}_{p,1})}+\|\tau^{\alpha} u\|_{\tL^{\infty}_{t}(\dot{B}^{\frac{d}{2}}_{2,1})}^{h})(\|\bf{v}\|_{\tL^{1}_{t}(\dot{B}^{\frac{d}{p}}_{p,1})}^{\ell}+\|\bf{v}\|_{\tL^{1}_{t}(\dot{B}^{\frac{d}{2}}_{2,1})}^{h})\\
&\lesssim \mathcal{X}_{p,0}(\|\tau^{\alpha} u^{\ell}\|_{\tL^{\infty}_{t}(\dot{B}^{\frac{d}{p}}_{p,1})}+\|\tau^{\alpha} u\|_{\tL^{\infty}_{t}(\dot{B}^{\frac{d}{2}}_{2,1})}^{h}).
\end{aligned}
\end{equation}
Inserting \eqref{DecayInter1}-\eqref{LNolinerEst1t1} into \eqref{DecayLow} and making use of the evolution of the low-frequency $\dot{B}^{\sigma_{1}}_{p,\infty}$-regularity in Lemma \ref{Decayprop1}, we obtain 
\begin{equation}\label{DecaytWLow}
\begin{aligned}
&\|\tau^\alpha  u^{\ell}\|_{\tL^{\infty}_{t}(\dot{B}^{\frac{d}{p}}_{p,1})}
+\|\tau^\alpha u^{\ell}\|_{\tL^1_{t}(\dot{B}^{\frac{d}{p}+2}_{p,1})}
+\var\|\tau^\alpha \bZ^{\ell} \|_{\tL^{\infty}_{t}(\dot{B}^{\frac{d}{p}}_{p,1})}
+\frac{1}{\var}\|\tau^\alpha \bZ^{\ell}\|_{\tL^1_{t}(\dot{B}^{\frac{d}{p}}_{p,1})}\\
&\quad\lesssim (\mathcal{X}_{p,0}+\kappa)(\| \tau^\alpha u^{\ell}\|_{\tL^{\infty}_{t}(\dot{B}^{\frac{d}{p}}_{p,1})}+\|  \tau^\alpha u\|_{\tL^{\infty}_{t}(\dot{B}^{\frac{d}{2}}_{2,1})}^{h}+\frac{1}{\var^2}\|\tau^{\alpha} u\|_{\tL^1_{t}(\dot{B}^{\frac{d}{2}}_{2,1})}^{h}+\frac{1}{\var}\|\tau^\alpha \bZ^{\ell}\|_{\tL^1_{t}(\dot{B}^{\frac{d}{p}}_{p,1})}) \\
&\quad\quad+\frac{1}{\kappa} t^{\alpha-\frac{1}{2}(\frac{d}{p}-\sigma_{1})}(\|u^{\ell}\|_{\tL^{\infty}_{t}(\dot{B}^{\sigma_{1}}_{p,\infty})}+\var\|\bZ^{\ell}\|_{\tL^{\infty}_{t}(\dot{B}^{\sigma_{1}}_{p,\infty})}).
\end{aligned}
\end{equation}

\begin{itemize}
\item \underline{{\emph{ Step 3: {Time-weighted estimates in the high-frequency regime}}}}
\end{itemize}

We perform similar computations as in Subsection \ref{subsectionhigh}. Multiplying \eqref{JXSys1} with $t^{\alpha}$ leads to
\begin{equation} \nonumber
\left\{
\begin{aligned}
&\frac{\partial} {\partial t} (t^\alpha u)+\sum_{i=1}^{d} \frac{\partial} {\partial x_{i}} (t^\alpha v_{i})=\alpha t^{\alpha-1} u, \\
&\var^2\frac{\partial} {\partial t} (t^{\alpha} v_{i})
+ A_{i}\frac{\partial} {\partial x_{i}} (t^\alpha u)+ v_{i}=\var^2 t^{\alpha-1}v_i+ t^\alpha f_{i}(u),\quad i=1,2,...,d, \\
&(t^\alpha u, t^\alpha v_i)|_{t=0}=(0,0).
\end{aligned}
\right.
\end{equation}
Repeating the same argument as in \eqref{JXHD1}-\eqref{HEstFina12}, we get
\begin{equation}\label{DecayInter311111}
\begin{aligned}
\|\tau^{\alpha}(u, \var\bv)\|_{\tL^{\infty}_{t}(\dot{B}^{\frac{d}{2}}_{2,1})}^{h}
+\frac{1}{\var^2}\|\tau^{\alpha}(u, \var\bv)\|_{\tL^1_{t}(\dot{B}^{\frac{d}{2}}_{2,1})}^{h}\lesssim
\int_0^t \tau^{\alpha-1} \|(u,\var \bv)\|_{\dot{B}^{\frac{d}{2}}_{2,1}}^{h}\ d\tau+\frac{1}{\var}\|\tau^{\alpha} f(u)\|_{\tL^1_t(B^{\frac d2}_{2,1})}^{h}.
\end{aligned}
\end{equation}
Let $\theta$ be given in Step 2.  It is easy to deduce 
\begin{equation}\label{DecayInter3}
\begin{aligned}
\int_0^t \tau^{\alpha-1} \|(u,\var\bv)\|_{\dot{B}^{\frac{d}{2}}_{2,1}}^{h}\ d\tau
&\lesssim
\int_0^t \tau^{\alpha-1} (\|(u,\var\bv)\|_{\dot{B}^{\frac{d}{2}}_{2,1}}^{h})^{\theta}(\|(u,\var\bv)\|_{\dot{B}^{\frac{d}{2}}_{2,1}}^{h})^{1-\theta}\ d\tau\\
&\lesssim
\Big(t^{\alpha-\frac{1}{2}(\frac{d}{p}-1-\sigma_{1})}\|(u,\var\bv)\|_{\tL^{\infty}_{t}(\dot{B}^{\frac{d}{2}}_{2,1})}^{h}\Big)^{\theta}\Big(\|\tau^{\alpha}(u, \var\bv)\|_{\tL^1_{t}(\dot{B}^{\frac{d}{2}}_{2,1})}^{h}\Big)^{1-\theta}\\
&\lesssim
 \frac{\kappa}{\var^2} \|\tau^{\alpha}(u, \var\bv)\|_{\tL^1_{t}(\dot{B}^{\frac{d}{2}}_{2,1})}^{h}
+\frac{1}{\kappa}t^{\alpha-\frac{1}{2}(\frac{d}{p}-1-\sigma_{1})}\|(u, \var\bv)\|_{\tL^1_{t}(\dot{B}^{\frac{d}{2}}_{2,1})}^{h}.
\end{aligned}
\end{equation}
Thanks to the composition estimate in Lemma \ref{NewSmoothhigh}, it holds that
\begin{equation}\label{DecayInter31}
\begin{aligned}
\|\tau^{\alpha} f(u)\|_{\tL^1_t(B^{\frac d2}_{2,1})}^{h}
&\lesssim
\big( \|u\|_{\tL_t^\infty(\dot B_{p,1}^{\frac{d}{p}})}^{\ell} 
+\|u\|_{\tL_t^\infty(\dot B_{2,1}^{\frac{d}{2}})}^{h} \big)
\frac{1}{\var}\|\tau^{\alpha} u\|_{\tL_t^1(\dot B_{2,1}^{\frac{d}{2}})}^{h} \\
&\quad+\big(\|u\|_{\tL_t^\infty(\dot B_{p,1}^{\frac{d}{p}-1})}^{\ell} 
+\var\|u\|_{\tL_t^\infty(\dot B_{2,1}^{\frac{d}{2}})}^{h} \big)
\|\tau^{\alpha}u\|_{\tL_t^1(\dot B_{p,1}^{\frac{d}{p}+2})}^{\ell}\\
&\lesssim
\cX_{p,0} \Big( \|\tau^{\alpha}u^{\ell}\|_{\tL_t^1(\dot B_{p,1}^{\frac{d}{p}+2})}+\frac{1}{\var^2}\|\tau^{\alpha}u\|_{\tL_t^1(\dot B_{2,1}^{\frac{d}{2}})}^{h}\Big),
\end{aligned}
\end{equation}
where we used that
$$
\begin{aligned}
\|\tau^{\alpha}u\|_{\tL_t^1(\dot B_{p,1}^{\frac{d}{p}+2})}^{\ell}&\lesssim \|\tau^{\alpha}u^{\ell}\|_{\tL_t^1(\dot B_{p,1}^{\frac{d}{p}+2})}^{\ell}+\|\tau^{\alpha}u^{h}\|_{\tL_t^1(\dot B_{p,1}^{\frac{d}{p}+2})}^{\ell}\lesssim\|\tau^{\alpha}u^{\ell}\|_{\tL_t^1(\dot B_{p,1}^{\frac{d}{p}+2})}+\frac{1}{\var^2}\|\tau^{\alpha}u\|_{\tL_t^1(\dot B_{2,1}^{\frac{d}{2}})}^{h}.
\end{aligned}
$$
Hence, from \eqref{DecayInter311111}-\eqref{DecayInter31} we have
\begin{equation}\label{DecaytWHigh}
\begin{aligned}
&\|\tau^{\alpha}(u, \var\bv)\|_{\tL^{\infty}_{t}(\dot{B}^{\frac{d}{2}}_{2,1})}^{h}
+\frac{1}{\var^2}\|\tau^{\alpha}(u, \var\bv)\|_{\tL^1_{t}(\dot{B}^{\frac{d}{2}}_{2,1})}^{h}
\\
&\quad\lesssim(\kappa+\cX_{p,0}) \Big( \|\tau^{\alpha}u^{\ell}\|_{\tL_t^1(\dot B_{p,1}^{\frac{d}{p}+2})}+\frac{1}{\var^2}\|\tau^{\alpha}u\|_{\tL_t^1(\dot B_{2,1}^{\frac{d}{2}})}^{h}\Big)+ \frac{1}{\kappa}t^{\alpha-\frac{1}{2}(\frac{d}{p}-1-\sigma_{1})}\|(u,\var\bv)\|_{\tL^{\infty}_{t}(\dot{B}^{\frac{d}{2}}_{2,1})}^{h}.
\end{aligned}
\end{equation}

\begin{itemize}
\item \underline{{\emph{Step 4: {Time-decay rates}}}}
\end{itemize}

Adding \eqref{DecaytWLow} and \eqref{DecaytWHigh} together, we derive 
\begin{equation}\nonumber
\begin{aligned}
&\|\tau^\alpha  u^{\ell}\|_{\tL^{\infty}_{t}(\dot{B}^{\frac{d}{p}}_{p,1})}
+\|\tau^\alpha u^{\ell}\|_{\tL^1_{t}(\dot{B}^{\frac{d}{p}+2}_{p,1})}
+\var\|\tau^\alpha \bZ^{\ell} \|_{\tL^{\infty}_{t}(\dot{B}^{\frac{d}{p}}_{p,1})}
+\frac{1}{\var}\|\tau^\alpha \bZ^{\ell}\|_{\tL^1_{t}(\dot{B}^{\frac{d}{p}}_{p,1})}\\
&\quad\quad+\|\tau^{\alpha}(u, \var\bv)\|_{\tL^{\infty}_{t}(\dot{B}^{\frac{d}{2}}_{2,1})}^{h}
+\frac{1}{\var^2}\|\tau^{\alpha}(u, \var\bv)\|_{\tL^1_{t}(\dot{B}^{\frac{d}{2}}_{2,1})}^{h}\\
&\lesssim (\mathcal{X}_{p,0}+\kappa)\Big(\| \tau^\alpha u^{\ell}\|_{\tL^{\infty}_{t}(\dot{B}^{\frac{d}{p}}_{p,1})}+\|  \tau^\alpha u\|_{\tL^{\infty}_{t}(\dot{B}^{\frac{d}{2}}_{2,1})}^{h}+\frac{1}{\var^2}\|\tau^{\alpha} u\|_{\tL^1_{t}(\dot{B}^{\frac{d}{2}}_{2,1})}^{h}+\frac{1}{\var}\|\tau^\alpha \bZ^{\ell}\|_{\tL^1_{t}(\dot{B}^{\frac{d}{p}}_{p,1})}\Big) \\
&\quad\quad+\frac{1}{\kappa} t^{\alpha-\frac{1}{2}(\frac{d}{p}-\sigma_{1})}\Big(\|u^{\ell}\|_{\tL^{\infty}_{t}(\dot{B}^{\sigma_{1}}_{p,\infty})}+\var\|\bZ^{\ell}\|_{\tL^{\infty}_{t}(\dot{B}^{\sigma_{1}}_{p,\infty})}+\|(u, \var\bv)\|_{\tL^1_{t}(\dot{B}^{\frac{d}{2}}_{2,1})}^{h}\Big).
\end{aligned}
\end{equation}
Choosing the constant $\kappa$ small enough and noticing that $\cX_{p,0}$ satisfies \eqref{small}, we arrive at
\begin{equation}\label{DecaytW}
\begin{aligned}
&\|\tau^\alpha  u^{\ell}\|_{\tL^{\infty}_{t}(\dot{B}^{\frac{d}{p}}_{p,1})}
+\|\tau^\alpha u^{\ell}\|_{\tL^1_{t}(\dot{B}^{\frac{d}{p}+2}_{p,1})}
+\var\|\tau^\alpha \bZ^{\ell} \|_{\tL^{\infty}_{t}(\dot{B}^{\frac{d}{p}}_{p,1})}
+\frac{1}{\var}\|\tau^\alpha \bZ^{\ell}\|_{\tL^1_{t}(\dot{B}^{\frac{d}{p}}_{p,1})}\\
&\quad\quad+\|\tau^{\alpha}(u, \var\bv)\|_{\tL^{\infty}_{t}(\dot{B}^{\frac{d}{2}}_{2,1})}^{h}
+\frac{1}{\var^2}\|\tau^{\alpha}(u, \var\bv)\|_{\tL^1_{t}(\dot{B}^{\frac{d}{2}}_{2,1})}^{h}\\
&\quad\lesssim  \mathcal{D}_{p,0}t^{\alpha-\frac{1}{2}(\frac{d}{p}-\sigma_{1})},
\end{aligned}
\end{equation}
where we have used the uniform bound \eqref{ThmEst1} and the low-frequency evolution \eqref{sigma1}. In addition, thanks to $Z_i=A_{i}\dfrac{\partial} {\partial x_{i}} u+v_{i}-f_i(u)$, \eqref{DecaytW} and Lemma \ref{DifferComposition}, we recover the time-weighted estimate of  $\bv^{\ell}$ as follows:
\begin{equation}\label{DecaytW1}
\begin{aligned}
\var\|\tau^\alpha \bv^{\ell}\|_{\tL^{\infty}_{t}(\dot{B}^{\frac{d}{p}}_{p,1})}
&\lesssim
\var\|\tau^\alpha \bZ^{\ell}\|_{\tL^{\infty}_{t}(\dot{B}^{\frac{d}{p}}_{p,1})}
+\|\tau^\alpha u^{\ell}\|_{\tL^{\infty}_{t}(\dot{B}^{\frac{d}{p}}_{p,1})}
+\var\|\tau^\alpha f(u)\|_{\tL^{\infty}_{t}(\dot{B}^{\frac{d}{p}}_{p,1})}\\
&\lesssim
\var\|\tau^\alpha \bZ^{\ell}\|_{\tL^{\infty}_{t}(\dot{B}^{\frac{d}{p}}_{p,1})}
+\|\tau^\alpha u^{\ell}\|_{\tL^{\infty}_{t}(\dot{B}^{\frac{d}{p}}_{p,1})}
+\var\|\tau^\alpha u\|_{\tL^{\infty}_{t}(\dot{B}^{\frac{d}{2}}_{2,1})}^{h}\\
&\lesssim \mathcal{D}_{p,0}t^{\alpha-\frac{1}{2}(\frac{d}{p}-\sigma_{1})}.
\end{aligned}
\end{equation}
Dividing \eqref{DecaytW} and \eqref{DecaytW1} by $t^{\alpha}$ and making use of \eqref{ThmEst1}, we have
\begin{equation}\label{DecaytW3}
\begin{aligned}
\|(u^{\ell}, \var\bv^{\ell})(t)\|_{\dot{B}^{\frac{d}{p}}_{p,1}} 
+\|( u, \var \bv)(t)\|_{\dot{B}^{\frac{d}{2}}_{2,1}}^{h}
\lesssim
\mathcal{D}_{p,0}(1+t)^{-\frac{1}{2}(\frac{d}{p}-\sigma_{1})},\quad\quad t>0,
\end{aligned}
\end{equation}
which, together with the real interpolation between \eqref{sigma1} and \eqref{DecaytW3}, implies
\begin{equation}\nonumber
\begin{aligned}
\|(u^{\ell}, \var\bv^{\ell})(t)\|_{\dot{B}^{\sigma}_{p,1}}&\lesssim \|(u^{\ell}, \var\bv^{\ell})(t)\|_{\dot{B}^{\sigma_{1}}_{p,\infty}}^{\frac{\frac{d}{p}-\sigma}{\frac{d}{p}-\sigma_{1}}}\|(u^{\ell}, \var\bv^{\ell})(t)\|_{\dot{B}^{\frac{d}{p}}_{p,1}}^{\frac{\sigma-\sigma_{1}}{\frac{d}{p}-\sigma_{1}}}\\
&\lesssim
\mathcal{D}_{p,0}(1+t)^{-\frac{1}{2}(\sigma-\sigma_{1})},\quad \quad t>0,\quad \sigma_{1}<\sigma<\frac{d}{p}.
\end{aligned}
\end{equation}
Therefore, as $p\geq2$, we obtain
\begin{equation}\nonumber
\begin{aligned}
\|(u, \var\bv)(t)\|_{\dot{B}^{\sigma}_{p,1}}
&\lesssim
\|(u^{\ell}, \var\bv^{\ell})(t)\|_{\dot{B}^{\sigma}_{p,1}}
+\var^{\frac{d}{p}-\sigma}\|(u, \var\bv)(t)\|_{\dot{B}^{\frac{d}{2}}_{2,1}}^{h}\\
&\lesssim
\mathcal{D}_{p,0}(1+t)^{-\frac{1}{2}(\sigma-\sigma_{1})},\quad\quad t>0,\quad \sigma_{1}<\sigma\leq \frac{d}{p}.
\end{aligned}
\end{equation}
The proof of Proposition \ref{propdecay} is complete.

\subsection{Enhanced decay of the difference}

 Similar to Lemma \ref{Decayprop1}, we first establish the evolution of the lower-order $\dot{B}^{\sigma_{1}-1}_{p,\infty}$-regularity of the difference.

\begin{Lemma}\label{ErrorDecayprop1}
 Let $(u,\bv)$ be the global solution to \eqref{JXSys1} subject to the initial data $(u_{0},\bv_{0})$, and $u^*$ be the global solution to \eqref{Thm2uheat} with the initial data $u_{0}$. Then,  under the assumptions \eqref{small} and \eqref{a2}, 
it holds that
\begin{equation}\label{Errorsigma1}
\begin{aligned}
&\|u-u^*\|_{\tL^{\infty}(\R_{+};\dot{B}^{\sigma_1-1}_{p,\infty})}
\lesssim \var \mathcal{D}_{p,0}
\end{aligned}
\end{equation}
with  $\mathcal{D}_{p,0}=\|(u_{0}^{\ell},\var \bv_{0}^{\ell})\|_{\dot{B}^{\sigma_1}_{p,\infty}}
+\cX_{p,0}$. 
\end{Lemma}
\begin{proof}
We recall that the difference $\delta u=u-u^{*}$ satisfies \eqref{ErrorSys}. Applying the low-frequency cut-off operator $\dot{S}_{J_{\var}}$ to \eqref{ErrorSys} and taking advantage of Lemma \ref{HeatRegulEstprop}, we deduce 
\begin{equation}\label{ErrorDecaySigma}
\begin{aligned}
&\|\delta u\|^{\ell}_{\tL^{\infty}(\R_{+};\dot{B}^{\sigma_1-1}_{p,\infty})}+\|\delta u\|^{\ell}_{\tL^2(\R_{+};\dot{B}^{\sigma_1}_{p,\infty})}
+\|\delta u\|^{\ell}_{\tL^1(\R_{+};\dot{B}^{\sigma_1+1}_{p,\infty})}\\
&\quad \lesssim 
\|\bZ\|^{\ell}_{\tL^1(\R_{+};\dot{B}^{\sigma_1}_{p,\infty})}
+\| f(u)-f(u^*)\|_{\tL_t^1(\dot{B}^{\sigma_1}_{p,\infty})}^{\ell}.
\end{aligned}
\end{equation}
Recalling that \eqref{sigma1} holds, we have the key observation:
\begin{align}
\|\bZ\|^{\ell}_{\tL^1(\R_{+};\dot{B}^{\sigma_1}_{p,\infty})}\lesssim \var \mathcal{D}_{p,0}.\label{ErrorDecaySigma1}
\end{align}
In light of \eqref{limitur}, \eqref{additionu}, \eqref{HLEst} and Corollary \ref{cor0.1}, the nonlinear term in \eqref{ErrorDecaySigma} can be estimated as follows
\begin{equation}\label{ErrorDecaySigma2}
\begin{aligned}
\| f(u)-f(u^*)\|^{\ell}_{\tL^1(\R_{+};\dot{B}^{\sigma_1}_{p,\infty})}&\lesssim \|(u,u^{*})\|_{\tL^2(\R_{+};\dot{B}^{\frac{d}{p}}_{p,1})} \|\delta u\|_{\tL^2(\R_{+};\dot{B}^{\sigma_1}_{p,\infty})}\lesssim  \mathcal{X}_{p,0} \|\delta u\|_{\tL^2(\R_{+};\dot{B}^{\sigma_1}_{p,\infty})}.
\end{aligned}
\end{equation}
According to \eqref{limitur}, \eqref{ThmEst1} and $p\geq2$, the high frequencies of $\delta u$ also satisfy
\begin{equation}\label{ErrorDecaySigma3}
    \begin{aligned}
    &\|\delta u\|_{\tL^{\infty}(\R_{+};\dot{B}^{\sigma_1-1}_{p,\infty})}^{h}+\|\delta u\|_{\tL^{2}(\R_{+};\dot{B}^{\sigma_1}_{p,\infty})}^{h}\\
    &\quad\lesssim \var^{\frac{d}{p}-\sigma_{1}+1}\|(u,u^{*})\|_{\tL^{\infty}(\R_{+};\dot{B}^{\frac{d}{p}}_{p,1})}^{h}+ \var^{\frac{d}{p}-\sigma_{1}}\|(u,u^{*})\|_{\tL^{2}(\R_{+};\dot{B}^{\frac{d}{p}}_{p,1})}^{h} \lesssim  \var \mathcal{D}_{p,0},
    \end{aligned}
\end{equation}
where we noted $\frac{d}{p}-\sigma_{1}> 1$. Combining \eqref{ErrorDecaySigma}-\eqref{ErrorDecaySigma2} together and using the smallness of $\mathcal{X}_{p,0}$, we end up with \eqref{Errorsigma1}.
\end{proof}

Finally, to complete the proof of Theorem \ref{Thm3}, we prove enhanced time-decay rates for the difference.

\begin{Prop}\label{prop42}
Let $p$ be given by \eqref{p2}. Assume that $(u_{0},\bv_{0})$ satisfies \eqref{small} and \eqref{a2}. Let $(u,\bv)$ be the global solution to \eqref{JXSys1} subject to the initial data $(u_{0},\bv_{0})$, and $u^{*}$ be the global solution to \eqref{Thm2uheat} subject to the initial data $u_{0}$. Then, it holds that
\begin{equation}\nonumber
\|(u-u^{*})(t)\|_{\dot{B}^{\sigma}_{p,1}}\lesssim 
\var (1+t)^{-\frac{1}{2}(\sigma-\sigma_1+1)} \mathcal{D}_{p,0},\quad \sigma_{1}<\sigma\leq \frac{d}{p}-1
\end{equation}
with $\mathcal{D}_{p,0}=\|(u_{0}^{\ell},\var \bv_{0}^{\ell})\|_{\dot{B}^{\sigma_1}_{p,\infty}}
+\cX_{p,0}$. 
\end{Prop}

\begin{proof}
Recall that $(u,\bv)$ satisfies the heat-like decay estimates \eqref{decayuv}. Following the low-frequency analysis of Proposition \ref{propdecay}, we have
\begin{equation}\label{decayulimit}
\begin{aligned}
&\|u^{*}(t)\|_{\dot{B}^{\sigma}_{p,1}}\lesssim (1+t)^{-\frac{1}{2}(\sigma-\sigma_{1})}\mathcal{D}_{p,0},\quad\quad t>0,\quad \sigma_{1}<\sigma\leq \frac{d}{p}.
\end{aligned}
\end{equation}
Hence, using $p\geq2$, Bernstein's inequality, the decay estimates \eqref{decayuv} and \eqref{decayulimit} enable us to derive the enhanced decay of $\delta u=u-u^*$ in the high-frequency regime:
\begin{align}
&\|\delta u(t)\|_{\dot{B}^{\sigma}_{p,1}}^{h}\lesssim \var^{\frac{d}{p}-\sigma} \|(u,u^{*})(t)\|_{\dot{B}^{\frac{d}{2}}_{2,1}}^{h}\lesssim \var (1+t)^{-\frac{1}{2}(\sigma-\sigma_{1}+1)},\quad t>0,\quad \sigma_{1}<\sigma\leq \frac{d}{p}-1.\label{highhhhhh}
\end{align}

In order to derive decay rates in low frequencies, from \eqref{ErrorSys} we have
\begin{equation}\label{ErrorSys2}
\begin{aligned}
&\partial_t(t^\alpha \delta u^{\ell})-\sum_{i=1}^{d}\Big( t^\alpha \frac{\partial}{\partial x_{i}}(A_{i}\frac{\partial}{\partial x_{i}} \delta u^{\ell})\Big)=\alpha (t^{\alpha-1} \delta u^{\ell})\\
&\quad\quad -\sum_{i=1}^{d}\Big(t^\alpha \frac{\partial}{\partial x_{i}}Z_{i}^{\ell}\Big)
-\sum_{i=1}^{d}\Big(t^\alpha\frac{\partial}{\partial x_{i}}\dot{S}_{J_{\var}}\big(f_{i}(u)-f_{i}(u^*) \big)\Big),
\end{aligned}
\end{equation}
where $\alpha>1$ is a given constant. Since $ \delta u|_{t=0}=0$, using \eqref{ErrorDecayEst100} and applying Lemma \ref{HeatRegulEstprop} to \eqref{ErrorSys2}, we derive that
\begin{equation}\label{ErrorDecayEst1}
\begin{aligned}
&\|\tau^\alpha \delta u^{\ell}\|_{\tL_t^{\infty}(\dot{B}^{\frac{d}{p}-1}_{p,1})}+\|\tau^\alpha \delta u^{\ell}\|_{\tL_t^{2}(\dot{B}^{\frac{d}{p}}_{p,1})}
+\|\tau^\alpha \delta u^{\ell}\|_{\tL_t^1(\dot{B}^{\frac{d}{p}+1}_{p,1})}\\
&\quad\lesssim \int_{0}^{t}\tau^{\alpha-1} \|\delta u^{\ell}\|_{\dot{B}^{\frac{d}{p}-1}_{p,1}}d\tau
+\|\tau^\alpha\bZ^{\ell}\|_{\tL_t^{1}(\dot{B}^{\frac{d}{p}}_{p,1})}+\|\tau^\alpha (f(u)-f(u^*))\|_{\tL_t^1(\dot{B}^{\frac{d}{p}}_{p,1})}^{\ell}.
\end{aligned}
\end{equation}
Let $\theta\in(0,1)$ be given by $d/p-1=\theta(\sigma_{1}-1)+(1-\theta)(d/p+1)$. In view of real interpolation, the key estimate \eqref{Errorsigma1}  guarantees that
\begin{equation}\label{ErrorDecayEst101}
\begin{aligned}
\int_{0}^{t}\tau^{\alpha-1} \|\delta u^{\ell}\|_{\dot{B}^{\frac{d}{p}-1}_{p,1}}d\tau&\lesssim \int_{0}^{t}\tau^{\alpha-1} \|\delta u^{\ell}\|_{\dot{B}^{\sigma_{1}-1}_{p,\infty}}^{\theta}\|\delta u^{\ell}\|_{\dot{B}^{\frac{d}{p}+1}_{p,1}}^{1-\theta}d\tau\\
&\lesssim \Big( t^{\alpha-\frac{1}{2}(\frac{d}{p}-\sigma_{1})}\|\delta u^{\ell}\|_{\widetilde{L}^{\infty}_t(\dot{B}^{\sigma_{1}-1}_{p,\infty})}\Big)^{\theta}\|\delta u^{\ell}\|_{\tL^{1}_{t}(\dot{B}^{\frac{d}{p}+1}_{p,1})}^{1-\theta}\\
&\lesssim \kappa_{1}\|\delta u^{\ell}\|_{\tL^{1}_{t}(\dot{B}^{\frac{d}{p}+1}_{p,1})}+\frac{\var}{\kappa_{1}}t^{\alpha-\frac{1}{2}(\frac{d}{p}-\sigma_{1})} \mathcal{D}_{p,0}
\end{aligned}
\end{equation}
for some constant $\kappa_{1}$ to be chosen later. To control the second term on the r.h.s of $\eqref{ErrorDecayEst1}$, from \eqref{DecaytW}, we have
\begin{equation}\label{ErrorDecayEst100}
\begin{aligned}
\|\tau^\alpha \bZ^{\ell}\|_{\tL_t^1(\dot{B}^{\frac{d}{p}}_{p,1})}&\lesssim \var t^{\alpha-\frac{1}{2}(\frac{d}{p}-\sigma_{1})}\mathcal{D}_{p,0}.
\end{aligned}
\end{equation}
The bounds in Theorems \ref{Thm0}-\ref{Thm1} as well as \eqref{additionu}, \eqref{DecaytWHigh}, Corollary \ref{cor0.1} and $u_0=u_0^*$ ensure that
\begin{equation}\label{ErrorDecayEst102}
\begin{aligned}
\|\tau^\alpha (f(u)-f(u^*))\|_{\tL_t^1(\dot{B}^{\frac{d}{p}}_{p,1})}^{\ell}&\lesssim \|(u,u^{*})\|_{\tL_t^{2}(\dot{B}^{\frac{d}{p}}_{p,1})}\|\tau^\alpha\delta u\|_{\tL_t^{2}(\dot{B}^{\frac{d}{p}}_{p,1})}\\
&\lesssim \mathcal{X}_{p,0}\|\tau^\alpha\delta u^{\ell}\|_{\tL_t^{2}(\dot{B}^{\frac{d}{p}}_{p,1})}+\mathcal{X}_{p,0}\|\tau^\alpha\delta u\|_{\tL_t^{2}(\dot{B}^{\frac{d}{2}}_{2,1})}^h\\
&\lesssim \mathcal{X}_{p,0} \|\tau^\alpha\delta u^{\ell}\|_{\tL_t^{2}(\dot{B}^{\frac{d}{p}}_{p,1})}+\var t^{\alpha-\frac{1}{2}(\frac{d}{p}-\sigma_{1})}\mathcal{D}_{p,0}.
\end{aligned}
\end{equation}
From \eqref{ErrorDecayEst1}-\eqref{ErrorDecayEst102} and the smallness of $\mathcal{X}_{p,0}$, one infers, for some suitable  $\kappa_{1}>0$, that
\begin{equation}\label{ErrorDecayEst0}
\begin{aligned}
&\|\tau^\alpha \delta u^{\ell}\|_{\tL_t^{\infty}(\dot{B}^{\frac{d}{p}-1}_{p,1})}+\|\tau^\alpha \delta u^{\ell}\|_{\tL_t^{2}(\dot{B}^{\frac{d}{p}}_{p,1})}
+\|\tau^\alpha \delta u^{\ell}\|_{\tL_t^1(\dot{B}^{\frac{d}{p}+1}_{p,1})}\lesssim \var t^{\alpha-\frac{1}{2}(\frac{d}{p}-\sigma_{1})} \mathcal{D}_{p,0}.
\end{aligned}
\end{equation}
Since $\alpha>1$ is any given constant, dividing both sides of \eqref{ErrorDecayEst0} by $t^{\alpha}$ and using the bounds \eqref{limitur}, \eqref{additionu} and \eqref{highhhhhh}, we infer
\begin{equation}\label{ErrorDecayEst6}
\begin{aligned}
\|\delta u^{\ell}(t) \|_{\dot{B}^{\frac{d}{p}-1}_{p,1}}
\lesssim
\var (1+t)^{-\frac{1}{2}(\frac{d}{p}-\sigma_1)}\mathcal{D}_{p,0},\quad  t>0.
\end{aligned}
\end{equation}
By virtue of the real interpolation between \eqref{Errorsigma1} and \eqref{ErrorDecayEst6}, we get
\begin{equation}\label{1111}
\begin{aligned}
\|\delta u^{\ell}(t)\|_{\dot{B}^{\sigma}_{p,1}}
&\lesssim \|\delta u^{\ell}(t)\|_{\dot{B}^{\sigma_1-1}_{p,\infty}}^{\frac{\frac{d}{p}-1-\sigma}{\frac{d}{p}-\sigma_1}}\|\delta u^{\ell}(t)\|_{\dot{B}^{\frac{d}{p}-1}_{p,1}}^{\frac{\sigma-\sigma_1+1}{\frac{d}{p}-\sigma_1}}
\lesssim \var (1+t)^{-\frac{1}{2}(\sigma-\sigma_1+1)}\mathcal{D}_{p,0}
\end{aligned}
\end{equation} 
for $t>0$ and $-\frac{d}{p}<\sigma<\frac{d}{p}-1$. By \eqref{highhhhhh}, \eqref{ErrorDecayEst6} and \eqref{1111}, the proof of Proposition \ref{prop42} is concluded.
\end{proof}

\vspace{2mm}

\noindent
\textbf{\emph{Proof of Theorem \ref{Thm3}}.}
Under  the assumptions of Theorem \ref{Thm3}, we conclude from Propositions \ref{propdecay} and \ref{ErrorDecayprop1} that $(u,\bv)$ satisfies the time-decay estimates \eqref{decaysolution}, and the difference $u-u^{*}$ verifies the enhanced time-decay estimates \eqref{decayerror}. \qed

\vspace{8mm}

\textbf{Acknowledgments}
 T. Crin-Barat has been funded by the Alexander von Humboldt-Professorship program and the Transregio 154 Project “Mathematical Modelling, Simulation and Optimization Using the Example of Gas Networks” of the DFG. L.-Y. Shou was supported by the National Natural Science Foundation of China (12301275). J. Zhang was supported by the Shandong Province Natural Science Foundation, China (ZR2024QA003) and the Doctoral Scientific Research Foundation of Shandong Technology and Business University (BS202339).





\appendix
\section{Appendix}\label{section6}

\subsection{Some analysis tools in Besov spaces}
We state some properties of Besov spaces and related estimates which we repeatedly used in the paper. 
The first lemma is devoted to the classical Bernstein's inequality.
\begin{Lemma}[\cite{Chemin}]\label{lemma21}
Let $0<r<R, 1\leq p\leq q\leq \infty$ and $k\in \mathbb{N}$. Then,  for any function $u\in L^p$ and $\lambda_{1}>0$, it holds that
\begin{equation}\nonumber
\left\{
\begin{aligned}
&{\rm{Supp}}~ \mathcal{F}(u) \subset \{\xi\in\mathbb{R}^{d}~| ~|\xi|\leq \lambda_{1} R\}\Rightarrow \|D^{k}u\|_{L^q}\lesssim\lambda_{1}^{k+d(\frac{1}{p}-\frac{1}{q})}\|u\|_{L^p},\\
&{\rm{Supp}}~ \mathcal{F}(u) \subset \{\xi\in\mathbb{R}^{d}~|~ \lambda_{1} r\leq |\xi|\leq \lambda_{1} R\}\Rightarrow \|D^{k}u\|_{L^{p}}\sim\lambda_{1}^{k}\|u\|_{L^{p}}.
\end{aligned}
\right.
\end{equation}
\end{Lemma}

The next lemma states the classical interpolation inequalities. 
\begin{Lemma}[\cite{Chemin,XinXu}]\label{Classical Interpolation}
Let $1 \leq  p, r, r_1, r_2 \leq \infty$.
\begin{itemize}
\item If $u\in \dot B_{p,r_1}^s \cap \dot B_{p,r_2}^{\tilde{s}} $ and $s\not = \tilde{s}$, then $u\in \dot B_{p,r}^{\theta s+(1-\theta)\tilde{s}}$  for all $\theta\in (0, 1)$ and
\begin{equation}\nonumber
\begin{aligned}
\|u\|_{\dot B_{p,r}^{\theta s+(1-\theta)\tilde{s}}}
\lesssim
\|u\|_{\dot B_{p,r}^{s}}^{\theta}
\|u\|_{\dot B_{p,r}^{\tilde{s}}}^{1-\theta}
\end{aligned}
\end{equation}
with $\frac{1}{r}=\frac{\theta}{r_1}+\frac{1-\theta}{r_2}$.
\item If $u\in \dot B_{p,\infty}^s \cap \dot B_{p,\infty}^{\tilde{s}} $ and $s<\tilde{s}$, then $u\in \dot B_{p,1}^{\theta s+(1-\theta)\tilde{s}}$  for all $\theta\in (0, 1)$ and
\begin{equation*}
\|u\|_{\dot B_{p,1}^{\theta s+(1-\theta)\tilde{s}}}
\leq
\frac{C}{\theta(1-\theta)(\tilde{s}-s)}\|u\|_{\dot B_{p,\infty}^{s}}^{\theta}
\|u\|_{\dot B_{p,\infty}^{\tilde{s}}}^{1-\theta}.
\end{equation*}
\end{itemize} 
\end{Lemma}

The following interpolation inequalities for high and low frequencies are also used in this paper.
\begin{cor}\label{Interpolation}
Let $s_1\leq s_2$, $q,r\in[1,+\infty]$, $\theta\in (0,1)$ and $1\leq\alpha_1\leq \alpha\leq \alpha_2\leq \infty$ such that $\frac{1}{\alpha}=\frac{\theta}{\alpha_1}+\frac{1-\theta}{\alpha_2}$, then
\begin{equation*}
\begin{aligned}
&\|u\|_{\tL_T^{\alpha}(\dB_{q,r}^{\theta s_1+(1-\theta)s_2})}^{\ell}
\leq
\Big(\|u\|_{\tL_T^{\alpha_1}(\dB_{q,r}^{s_1})}^{\ell}\Big)^{\theta}
\Big(\|u\|_{\tL_T^{\alpha_2}(\dB_{q,r}^{s_2})}^{\ell}\Big)^{1-\theta},\\
&\|u\|_{\tL_T^{\alpha}(\dB_{q,r}^{\theta s_1+(1-\theta)s_2})} ^{h}
\leq
\Big(\|u\|_{\tL_T^{\alpha_1}(\dB_{q,r}^{s_1})} ^{h}\Big)^{\theta}
\Big(\|u\|_{\tL_T^{\alpha_2}(\dB_{q,r}^{s_2})} ^{h}\Big)^{1-\theta}.
\end{aligned}
\end{equation*}
\end{cor}

The following lemma pertains to classical product laws.
\begin{Lemma}{\rm\cite{Chemin,BaratDanchin3}}\label{ClassicalProductLaw}
Let $s>0$ and $1\leq p,r\leq \infty$, then  we have
\begin{equation}\label{ClassicalProductLawEst1}
\begin{aligned}
&\|ab\|_{\dot{B}^{s}_{p,r}}
\lesssim
\|a\|_{L^\infty}\|b\|_{\dot{B}^{s}_{p,r}}
+\|b\|_{L^\infty}\|a\|_{\dot{B}^{s}_{p,r}}.
\end{aligned}
\end{equation}
For $d\geq 1$ and $-\min\{d/p, d/p'\}<s\leq d/p$ for $1/p+1/p'=1$, the following inequality holds{\rm:}
\begin{equation}\label{ClassicalProductLawEst2}
\begin{aligned}
\|ab\|_{\dot{B}^{s}_{p,1}}
\lesssim
\|a\|_{\dot{B}^{\frac{d}{p}}_{p,1}}\|b\|_{\dot{B}^{s}_{p,1}}.
\end{aligned}
\end{equation}
Finally, if $d\geq 1$ and $-\min\{d/p, d/p'\}\leq s< d/p$ for $1/p+1/p'=1$, then, we have
\begin{equation}\label{ClassicalProductLawEst3}
\begin{aligned}
\|ab\|_{\dot{B}^{s}_{p,\infty}}
\lesssim
\|a\|_{\dot{B}^{\frac{d}{p}}_{p,1}}\|b\|_{\dot{B}^{s}_{p,\infty}}.
\end{aligned}
\end{equation}
\end{Lemma}
We also show a new product law to handle some nonlinear terms in the proof of the uniqueness.

\begin{Lemma}\label{NonClassicalProLaw1}
Let $s_1>0$ and $2\leq p\leq 4$. Then,  it holds that
\begin{align}\label{newproduct}
\|ab\|_{\dB_{2,1}^{s_1}}^{h}
&\lesssim 
\|a\|_{\dB_{p,1}^{\frac{d}{p}}}\|b\|_{\dB_{2,1}^{s_1}}^{h}
+\|b\|_{\dB_{p,1}^{\frac{d}{p}}}\|a\|_{\dB_{2,1}^{s_1}}^{h}
+ 2^{(s_1-\frac{d}{2})J}\|b\|_{\dB_{p,1}^{\frac{d}{p}}}^{\ell} \|a\|_{\dB_{p,1}^{\frac{d}{p}}}^{\ell} +\|b\|_{\dB_{p,1}^{\frac{d}{p}}}
\|a\|_{\dB_{p,1}^{s_1-\frac{d}{2}+\frac{d}{p}}}^\ell.
\end{align}
\end{Lemma}

\begin{proof}
We use Bony's paraproduct decomposition for
two tempered distributions $a$ and $b$:
\begin{equation}\nonumber
\begin{aligned}
ab=\dot{T}_{a}b+\dot{R}[a,b]+\dot{T}_{b}a\quad \mbox{with}\quad
\dot{T}_{a}b\triangleq\sum_{j'\in \Z} \dot S_{j'-1}a \ddjj b\quad\mbox{and}\quad
\dot{R}[a,b]\triangleq\sum_{|j'-j''|\leq 1} \dot\Delta_{j''} a \ddjj b.
\end{aligned}
\end{equation}
First, we bound $\dot{T}_{a}b$. It is clear that
\begin{equation}\nonumber
\begin{aligned}
\|\dot{T}_{a} b\|_{\dB_{2,1}^{s_1}}^{h}
\leq
\sum_{j\geq J-1\atop |j-j'|\leq 1}2^{s_1 j} \|\dot S_{j'-1} a \ddj \ddjj b\|_{L^2}
+\sum_{j\geq J-1\atop |j-j'|\leq 4} 2^{s_1 j}\|[\ddj, \dot S_{j'-1} a] \ddjj b\|_{L^2}.
\end{aligned}
\end{equation}
The embedding $\dot B_{p,1}^{\frac{d}{p}}\hookrightarrow L^{\infty}(\mathbb{R}^d)$ leads to
\begin{equation}\nonumber
\begin{aligned}
\sum_{j\geq J-1\atop |j-j'|\leq 1} 2^{s_1j}\|\dot S_{j'-1} a \ddj \ddjj b\|_{L^2}
\lesssim
\|a\|_{L^\infty}\sum_{j\geq J-1} 2^{s_1j} \|\ddj b\|_{L^2}
\lesssim
\|a\|_{\dB_{p,1}^{\frac{d}{p}}}\|b\|_{\dB_{2,1}^{s_1}}^{h}.
\end{aligned}
\end{equation}
Note that
\begin{equation}\nonumber
\begin{aligned}
\sum_{j\geq J-1\atop |j-j'|\leq 4} 2^{s_1 j}\|[\ddj, \dot S_{j'-1} a] \ddjj b\|_{L^2}\lesssim\Big(\sum\limits_{j'\geq J-1\atop |j-j'|\leq 4}
+ \sum\limits_{J -5 \leq j'\leq J-2\atop |j-j'|\leq 4} \Big)2^{s_1j'}
\|[\ddj, \dot{S}_{j'-1} a] \Delta_{j'} b\|_{L^2}.
\end{aligned}
\end{equation}
By the commutator estimate in \cite[Lemma 2.97]{Chemin}, we have 
\begin{equation}\nonumber
\begin{aligned}
\sum\limits_{j'\geq J-1\atop |j-j'|\leq 4}2^{s_1j'}
\|[\ddj, \dot{S}_{j'-1} a] \Delta_{j'} b\|_{L^2}
&\lesssim
\sum\limits_{j'\geq J-1\atop |j-j'|\leq 4}
\Big(2^{s_1j'}\|\Delta_{j'} b\|_{L^2}\Big)
\Big(2^{-j'}\|\nabla \dot{S}_{j'-1} a\|_{L^{\infty}}\Big)\\
&\lesssim
\|\nabla a\|_{\dB_{\infty,1}^{-1}}
\|b\|_{\dB_{2,1}^{s_1}}^{h}
\lesssim
\|a\|_{\dB_{p,1}^{\frac{d}{p}}}
\|b\|_{\dB_{2,1}^{s_1}}^{h}.
\end{aligned}
\end{equation}
Similarly, as $\frac{2p}{p-2}\geq p$ due to $2\leq p\leq 4$, one has
\begin{equation}\nonumber
\begin{aligned}
&\sum\limits_{J -5 \leq j'\leq J-2\atop |j-j'|\leq 4}2^{s_1j'}
\|[\ddj, \dot{S}_{j'-1} a] \Delta_{j'} b\|_{L^2}\\
&\quad\lesssim
2^{(s_{1}-\frac{d}{2})J}
\sum\limits_{J -5 \leq j'\leq J-2}
\Big(2^{\frac{d}{p} j'}\|\Delta_{j'} b\|_{L^p}\Big)
\Big(2^{(\frac{d}{2}-\frac{d}{p}-1) j'}\|\nabla \dot{S}_{j'-1} a\|_{L^{\frac{2p}{p-2}}}\Big)\\
&\quad\lesssim
2^{(s_1-\frac{d}{2})J}
\|b\|_{\dB_{p,1}^{\frac{d}{p}}}^{\ell} \|a\|_{\dB_{\frac{2p}{p-2},1}^{\frac{d}{2}-\frac{d}{p}-1}}^{\ell}\lesssim 2^{(s_1-\frac{d}{2})J}
\|b\|_{\dB_{p,1}^{\frac{d}{p}}}^{\ell} \|a\|_{\dB_{p,1}^{\frac{d}{p}}}^{\ell}.
\end{aligned}
\end{equation}
Hence, it follows that
\begin{equation}\nonumber
\begin{aligned}
\|\dot{T}_{a} b\|_{\dB_{2,1}^{s_1}}^{h} \lesssim \|a\|_{\dB_{p,1}^{\frac{d}{p}}}\|b\|_{\dB_{2,1}^{s_1}}^{h}+\|b\|_{\dB_{p,1}^{\frac{d}{p}}}^{\ell} \|a\|_{\dB_{p,1}^{\frac{d}{p}}}^{\ell}.
\end{aligned}
\end{equation}
The term $\|\dot{T}_{b} a\|_{\dB_{2,1}^{s_1}}^{h}$ can be treated similarly. Finally, the classical remainder estimates (see \cite[Theorem 2.85]{Chemin}) imply
\begin{equation}\nonumber
\begin{aligned}
\|\dot{R}[a,b]\|_{\dB_{2,1}^{s_1}}^{h}
&\leq
\|\dot{R}[a^h,b]\|_{\dB_{2,1}^{s_1}}^{h}
+\|\dot{R}[a^\ell,b]\|_{\dB_{2,1}^{s_1}}^{h}\\
&\lesssim
\|b\|_{\dB_{p,1}^{\frac{d}{p}}}
\|a\|_{\dB_{2,1}^{s_1}}^h
+\|b\|_{\dB_{\frac{2p}{p-2},1}^{\frac{d}{2}-\frac{d}{p}}}
\|a\|_{\dB_{p,1}^{s_1-\frac{d}{2}+\frac{d}{p}}}^\ell\\
&\lesssim \|b\|_{\dB_{p,1}^{\frac{d}{p}}}
\|a\|_{\dB_{2,1}^{s_1}}^h
+\|b\|_{\dB_{p,1}^{\frac{d}{p}}}
\|a\|_{\dB_{p,1}^{s_1-\frac{d}{2}+\frac{d}{p}}}^\ell.
\end{aligned}
\end{equation}
This completes the proof of Lemma \ref{NonClassicalProLaw1}.
\end{proof}

%

Next, we introduce classical estimates for the composition of functions.
\begin{Lemma}[\cite{Chemin,DanchinXu}]\label{DifferComposition}
Assume $d\geq 1$ and $F(m)$ is a smooth function such that $F(0) = 0$. For any $s>0$, $p,r\in[1,\infty]$ and real-valued function $m$ in $\dB_{p,r}^{s}\cap L^{\infty}$, $F(m)$ belongs to $\dB_{p,r}^{s}$ and fulfills
\begin{align*}
\|F(m)\|_{\dB_{p,r}^{s}}
\leq C_{m}
\|m\|_{\dB_{p,r}^{s}},
\end{align*}
where $C_{m}>0$ denotes a constant dependent on $\|m\|_{L^{\infty}}$, $F'$, $s$, $p$, $r$ and $d$.


\end{Lemma}

We have the following corollary.

\begin{cor}\label{cor0.1}
Assume that $F(m)$ is a smooth function satisfying $F'(0)=0.$ Let $1\leq p\leq \infty$. For any couple $(m_1,m_2)$ of functions in $\dB_{p,1}^{s}\cap L^\infty$, there exists a constant $C_{m_1,m_2}>0$ depending on  $F''$ and $\|(m_1,m_2)\|_{L^{\infty}}$ such that 
\begin{itemize}
\item Let $-\min\{d/p,d(1-1/p)\}<s\leq d/p$ and $1\leq r\leq \infty$. Then,  we have
\begin{align}\label{corIneq1}
\|F(m_1)-F(m_2)\|_{\dB_{p,r}^s}\leq C \|(m_1,m_2)\|_{\dB_{p,1}^{\frac{d}{p}}} \|m_1-m_2\|_{\dB_{p,r}^s}.
\end{align}
\item In the limiting case $r=\infty$, for any  $-\min\{d/p,d(1-1/p)\}\leq s< d/p$, it holds that 
\begin{align}\label{corIneq3}
\|F(m_1)-F(m_2)\|_{\dB_{p,\infty}^{s}}\leq C \|(m_1,m_2)\|_{\dB_{p,1}^{\frac{d}{p}}} \|m_1-m_2\|_{\dB_{p,\infty}^{s}}.
\end{align}
\end{itemize}
\end{cor}
\begin{proof}
From
$$
F(m_1)-F(m_2)=(m_1-m_2)\int_0^1 F'(m_1 + \tau (m_2-m_1))\ d\tau,
$$
\eqref{ClassicalProductLawEst2}, Lemma \ref{DifferComposition} and the embedding $\dot{B}^{\frac{d}{p}}_{p,1}\rightarrow L^{\infty}(\mathbb{R}^{d})$, we have
\begin{align*}
\|F(m_1)-F(m_2)\|_{\dB_{p,r}^s}
&\lesssim
\|m_1-m_2\|_{\dB_{p,r}^s}\sup_{\tau\in[0,1]}
\|f'(m_1 + \tau (m_2-m_1))\|_{\dB_{p,1}^{\frac{d}{p}}}\\
&\lesssim
\|m_1-m_2\|_{\dB_{p,r}^s}\|(m_1,m_2)\|_{\dB_{p,1}^{\frac{d}{p}}}.
\end{align*}
This yields \eqref{corIneq1}. In order to prove \eqref{corIneq3}, one can follow a similar argument replacing \eqref{ClassicalProductLawEst2} by \eqref{ClassicalProductLawEst3}.
\end{proof}


We now consider the following Cauchy problem of the parabolic equations:
\begin{equation}\label{Heat}
\left\{
\begin{aligned}
&\partial_t u- \sum_{i=1}^{d}\frac{\partial}{\partial x_{i}}(A_{i} \frac{\partial}{\partial x_{i}}  u) =F,\\
&u(0, x)=u_0(x),
\end{aligned}
\right. \quad \quad  t>0,\quad x\in\mathbb{R}^{d},
\end{equation}
where the unknown is $u=u(x,t)\in\mathbb{R}^{n}$.
\begin{Lemma}\label{HeatRegulEstprop}
Let $\varepsilon>0$, $d,n\geq1$, $s_1,s_2\in\mathbb{R}$, $1\leq \rho_1, \rho_2, p_1, p_2\leq\infty$ and $T>0$ be given time, and $J_\varepsilon$ be the threshold between low and high frequencies. Assume $u_{0}^{\ell}\in\dot{B}^{s_1}_{p_1,1}$, 
$u_0^{h}\in\dot{B}^{s_2}_{p_2,1}$, $F^{\ell}\in  \tL^{\rho_1}_T(\dot{B}^{s_1-2+2/\rho_1}_{p_1,1})$ and $F^{h}\in  \tL^{\rho_2}_T(\dot{B}^{s_2-2+2/\rho_2}_{p_2,1})$. If $u$ is a solution to the Cauchy problem \eqref{Heat}, then, for all
$\tilde{\rho}_1\in [\rho_1,\infty]$ and $\tilde{\rho}_2\in [\rho_2,\infty]$, $u$ satisfies
\begin{equation}\nonumber
\begin{aligned}
&\|u\|_{\tL^{\tilde{\rho}_1}_{T}(\dot{B}^{s_1+\frac{2}{\tilde{\rho}_1}}_{p_1,1})}^{\ell}
\leq C\Big(\|u_{0}\|_{\dot{B}^{s_1}_{p_1,1}}^{\ell}
+\|F\|_{\tL^{\rho_1}_{T}(\dot{B}^{s_1-2+\frac{2}{\rho_1}}_{p_1,1})}^{\ell}\Big),
\end{aligned}
\end{equation}
and
\begin{equation}\nonumber
\begin{aligned}
\|u\|_{\tL^{\tilde{\rho}_2}_{T}(\dot{B}^{s_2+\frac{2}{\tilde{\rho}_2}}_{p_2,1})}^{h}
\leq C\Big(\|u_{0}\|_{\dot{B}^{s_2}_{p_2,1}}^{h}
+\|F\|_{\tL^{\rho_2}_{T}(\dot{B}^{s_2-2+\frac{2}{\rho_2}}_{p_2,1})}^{h}\Big),
\end{aligned}
\end{equation}
where $C>0$ is a constant independent of $T$ and $\varepsilon$.
\end{Lemma}
\begin{proof}
Note that from \eqref{Heat}, $\dot{\Delta}_{j}u$ can be represented by
\begin{equation*}
\begin{aligned}
\dot \Delta_j u(t)= {\rm{e}}^{t \sum\limits_{i=1}^{d}\frac{\partial}{\partial x_{i}}(A_{i} \frac{\partial}{\partial x_{i}})}\dot \Delta_j u_{0}
+\int_0^t {\rm{e}}^{\sum\limits_{i=1}^{d}\frac{\partial}{\partial x_{i}}(A_{i} \frac{\partial}{\partial x_{i}})(t-\tau)}\dot\Delta_j F(\tau)
\ d\tau.
\end{aligned}
\end{equation*}
As in \cite[Lemma 2.4]{Chemin}, one can show the localized semigroup estimate:
\begin{equation*}
\begin{aligned}
\|\dot \Delta_j u(t)\|_{L^p}
\lesssim
{\rm{e}}^{-  \tilde{a} 2^{2j}t}\|\dot\Delta_j u_{0}\|_{L^p}
+\int_0^t {\rm{e}}^{-  \tilde{a} 2^{2j}(t-\tau)}\|\dot\Delta_j F(\tau)\|_{L^p}\ d\tau
\end{aligned}
\end{equation*}
for some constant $\tilde{a}>0$. Setting $p=p_1$ and applying Young's inequality, we get
\begin{equation*}
\begin{aligned}
\|\dot \Delta_j u\|_{L_T^{\tilde{\rho}_1}L^{p_1}}
\lesssim
\Big(\frac{1-{\rm{e}}^{-  \tilde{a}\tilde{\rho}_1T 2^{2j}}}{  \tilde{a}\tilde{\rho}_12^{2j}} \Big)^{\frac{1}{\tilde{\rho}_1}}
\|\dot\Delta_j u_{0}\|_{L^{p_1}}
+\Big(\frac{1-{\rm{e}}^{-  \tilde{a}\rho_1'T 2^{2j}}}{  \tilde{a} \rho_1' 2^{2j}} \Big)^{\frac{1}{\rho_1'}}
\|\dot\Delta_j F\|_{L_T^{\rho_1}L^{p_1}}
\end{aligned}
\end{equation*}
with $\frac{1}{\rho_1'}=1+\frac{1}{\tilde{\rho}_1}-\frac{1}{\rho_1}$. Summing the above inequalities over $j\leq J_{\var}$, we have
\begin{equation*}
\begin{aligned}
\|u\|_{L_T^{\tilde{\rho}_1}(\dot B_{p_1,1}^{s_1+\frac{2}{\tilde{\rho}_1}})}^{\ell}
\lesssim \|u_0\|_{\dot B_{p_1,1}^{s_1}}^{\ell}+\|F\|_{\tL^{\rho_1}_{t}(\dot{B}^{s_1-2+\frac{2}{\rho_1}}_{p_1,1})}^{\ell}.
\end{aligned}
\end{equation*}
The second estimate is similar.
\end{proof}

We also study the Cauchy problem of the damped equation
\begin{equation}\label{damped}
\left\{
\begin{aligned}
&\partial_t u+\frac{1}{\var^2} u =F,\\
&u(0, x)=u_0(x),
\end{aligned}
\right. \quad\quad \quad t>0,\quad  x\in\mathbb{R}^{d}.\\
\end{equation}
By direct computations, we have the following lemma.
\begin{Lemma}\label{maximaldamped}
Let $d,n\geq1$, $s\in\mathbb{R}$, $s_1,s_2\in\mathbb{R}$, $1\leq \rho_1, \rho_2, p_1, p_2\leq\infty$, $T>0$ be a given time, and $J_\varepsilon$ be the threshold between low and high frequencies. Assume $u_{0}^{\ell}\in\dot{B}^{s_1}_{p_1,1}$, $u_{0}^{h}\in\dot{B}^{s_2}_{p_2,1}$, $F^{\ell}\in  \tL^{\rho_1}_T(\dot{B}^{s_1}_{p_1,1})$ and $F^{h}\in  \tL^{\rho_1}_T(\dot{B}^{s_2}_{p_2,1})$. If $u$ is a solution to the Cauchy problem \eqref{damped}, then, for all
$\tilde{\rho}_1\in [\rho_1,\infty]$ and $\tilde{\rho}_2\in [\rho_2,\infty]$, it holds that
\begin{equation}\nonumber
\begin{aligned}
&\var^{-\frac{2}{\tilde{\rho}_1}}\|u\|_{\tL^{\tilde{\rho}_1}_{T}(\dot{B}^{s_1}_{p_1,1})}^{\ell}
\leq C\Big(
\|u_{0}\|_{\dot{B}^{s_1}_{p_1,1}}^{\ell}+\var^{2-\frac{2}{\rho_1}}\|F\|_{\tL^{\rho_1}_{T}(\dot{B}^{s_1}_{p_1,1})}^{\ell}\Big)
\end{aligned}
\end{equation}
and
\begin{equation}\nonumber
\begin{aligned}
\var^{-\frac{2}{\tilde{\rho}_2}}\|u\|_{\tL^{\tilde{\rho}_2}_{T}(\dot{B}^{s_2}_{p_2,1})}^{h}
\leq C\Big(\|u_{0}\|_{\dot{B}^{s_2}_{p_2,1}}^{h}
+\var^{2-\frac{2}{\rho_1}}\|F\|_{\tL^{\rho_2}_{T}(\dot{B}^{s_2}_{p_2,1})}^{h}\Big),
\end{aligned}
\end{equation}
where $C>0$ is a constant independent of $T$ and $\var$.
\end{Lemma}

\subsection{New estimates of composition functions in hybrid Besov spaces}

We develop some new composition estimates for functions in $L^p$-$L^2$ hybrid Besov spaces, which generalize the previous related estimates in \cite{Chen1,BaratShou2,XuZhang} 
 and play a key role in our nonlinear analysis. We denote by $J$ the general threshold between low and high frequencies (not necessarily defined by \eqref{Jvar}).

\begin{Lemma} \label{NewSmoothlow}
Let $s>0$, $\sigma\in\mathbb{R}$, and $p\geq \max\{1,\frac{2d}{d+2}\}$. Then, for any smooth function $F(m)$ satisfying $F(0)=F'(0)=0$, there is a constant $C_{m}>0$ depending only on $\|m\|_{L^{\infty}}$, $s$, $\sigma$ and $d$ such that
\begin{equation}\label{LemNewSmoothlow}
\begin{aligned}
\|F(m)\|_{\dot{B}^{s}_{p,1}}^{\ell}
&\leq C_{m}\big(\|m^\ell\|_{\dot B_{p,1}^{\frac{d}{p}} } +\|m\|_{\dot B_{2,1}^{\frac{d}{2}} }^h  \big)\|m\|_{\dot{B}^{s}_{p,1}}^{\ell}\\
&\quad+
C_{m}2^{(s-\sigma+\frac{d}{2}-\frac{d}{p})J}\big(2^{J}\|m^\ell\|_{\dot B_{p,1}^{\frac{d}{p}-1}}+ 
\|m\|_{\dot B_{2,1}^{\frac{d}{2}}}^h\big)
\|m\|_{\dot B_{2,1}^{\sigma}}^{h}.
\end{aligned}
\end{equation}
\end{Lemma}
\begin{proof}
For simplicity, we consider the case where $F(m)$ is scalar-valued; the extension to the general vector-valued case follows analogously. As in \cite[Theorem 2.61]{Chemin}, we use $F(0)=F'(0)=0$ to decompose $F(m)$ as
\begin{equation}\label{com1}
\begin{aligned}
F(m)
&=\sum_{k'\in\mathbb{Z}}F(\dot{S}_{k'+1}m)-F(\dot{S}_{k'}m) =\sum_{k'\in\mathbb{Z}} \dot{\Delta}_{k'}m M_{k'}
=\dot{\Delta}_{k'} m(\dot S_{k'-1} m+ \dot\Delta_{k'}m) \widetilde{M}_{k'}
\end{aligned}
\end{equation}
with
\begin{equation}\nonumber
\begin{aligned}
\widetilde{M}_{k'}\triangleq\int_0^1 M_{k'}'(\tau (\dot S_{k'-1} m+ \dot\Delta_{k'}m)) \ d\tau
\quad \mbox{and}\quad
M_{k'}\triangleq\int_{0}^{1} F'(\dot{S}_{k'}m+\tau\dot{\Delta}_{k'}m)\ d\tau.
\end{aligned}
\end{equation}
Thanks to Bernstein's inequality and Leibniz's formula, we have
\begin{equation}\label{SmoothFIneq2}
\begin{aligned}
&\|\dot{\Delta}_k (\dot{\Delta}_{k'}m \widetilde{M}_{k'}\dot S_{k'-1} m)\|_{L^{p^1}}+\|\dot{\Delta}_k (\dot{\Delta}_{k'}m \widetilde{M}_{k'} \dot{\Delta}_{k'} m)\|_{L^{p^1}}\\
&\quad\lesssim
2^{(k'-k){|\beta|}}\|\dot{\Delta}_{k'}m \|_{L^{p^2}}\|m\|_{L^{p^3}}(1+\|m\|_{L^\infty})^{|\beta|}
\end{aligned}
\end{equation}
with $\frac{1}{p^3}+\frac{1}{p^2}=\frac{1}{p^1}$ and $\beta\in \N^n$.

In order to justify \eqref{LemNewSmoothlow}, we decompose the low-frequency regime as
\begin{equation}\nonumber
\begin{aligned}
& \Omega_{\ell \ell}\triangleq\{(k,k')~|~k'<k\leq J\},\quad \Omega_{\ell m}\triangleq\{(k,k')~|~k\leq k'\leq J\},\quad \Omega_{\ell h}\triangleq\{(k,k')~|~k\leq J<k'\}.
\end{aligned}
\end{equation}
 Applying \eqref{SmoothFIneq2} with $|\beta|=[s]+1$ and $(p^1,p^2,p^3)=(\infty,p,\infty)$, we obtain
\begin{equation}\label{LemEst2qqq}
\begin{aligned}
&\sum_{ \Omega_{\ell \ell}}2^{ks}\|\dot{\Delta}_k (\dot{\Delta}_{k'}m (\dot S_{k'-1} m+\dot\Delta_{k'}m ) \widetilde{M}_{k'})\|_{L^p}\\
&\quad\lesssim
(1+\|m\|_{L^\infty})^{[s]+1}\|m\|_{L^{\infty}} \sum_{k'< J}2^{k's}\|\dot{\Delta}_{k'}m \|_{L^p}
\sum_{k> k'}2^{(k-k')(s-[s]-1)}\\
&\quad\lesssim(1+\|m\|_{L^\infty})^{[s]+1}
\big(\|m^\ell\|_{\dot B_{p,1}^{\frac{d}{p}} } +\|m\|_{\dot B_{2,1}^{\frac{d}{2}} }^h  \big)\|m\|_{\dot{B}^{s}_{p,1}}^{\ell}.
\end{aligned}
\end{equation}
Similarly, it follows from \eqref{SmoothFIneq2} with $|\beta|=0$ and $(p^1,p^2,p^3)=(\infty,p,\infty)$ that
\begin{equation}\label{LemEst1qqq}
\begin{aligned}
&\sum_{ \Omega_{\ell m}}2^{ks}\|\dot{\Delta}_k (\dot{\Delta}_{k'}m (\dot S_{k'-1} m+\dot\Delta_{k'}m ) \widetilde{M}_{k'})\|_{L^p}\\
&\quad\lesssim
\|m\|_{L^{\infty}} \sum_{ k' \leq J}2^{k's}\|\dot{\Delta}_{k'}m \|_{L^p}\sum_{k\leq k'}2^{(k-k')s}
\lesssim
\big(\|m^\ell\|_{\dot B_{p,1}^{\frac{d}{p}} } +\|m\|_{\dot B_{2,1}^{\frac{d}{2}} }^h  \big)\|m\|_{\dot{B}^{s}_{p,1}}^{\ell}.
\end{aligned}
\end{equation}

We analyze the part of the regime $ \Omega_{\ell h}$ as follows.
\begin{itemize}
    \item Case 1: $\max\{1,\frac{2d}{d+2}\}\leq p< 2$ and $s\leq \sigma$.
\end{itemize}
By \eqref{SmoothFIneq2} with $|\beta|=0$ and $(p^1,p^2,p^3)=(p,\frac{2p}{2-p}, 2)$,  we have
\begin{equation}\label{cases}
\begin{aligned}
&\sum_{\Omega_{\ell h}}2^{ks}\|\dot{\Delta}_k (\dot{\Delta}_{k'}m (\dot S_{k'-1} m+\dot\Delta_{k'}m ) \widetilde{M}_{k'})\|_{L^p}\\
&\quad\lesssim
\|m\|_{L^{\frac{2p}{2-p}}} \sum_{k'> J} 2^{k's}\|\dot{\Delta}_{k'}m \|_{L^2} \sum_{k'>k}2^{(k-k')s}
\lesssim
\|m\|_{L^{\frac{2p}{2-p}}}\|m\|_{\dot B_{2,1}^{s}}^{h}.
\end{aligned}
\end{equation}
Noticing that $\frac{d}{p}\leq d-\frac{d}{p}< \frac{d}{2}$ due to $\max\{1,\frac{2d}{d+2}\}\leq p< 2$, we conclude from the low-high frequency decomposition, \eqref{HLEst} and embedding inequalities that
\begin{equation}\label{ffff3}
\begin{aligned}
\|m\|_{L^{\frac{2p}{2-p}}}\lesssim \|m\|_{\dot{B}^{d-\frac{d}{p}}_{2,1}} &\lesssim \|m^{\ell}\|_{\dot{B}^{\frac{d}{2}}_{p,1}}+\|m\|_{\dot{B}^{d-\frac{d}{p}}_{2,1}}^{h}\lesssim 2^{(\frac{d}{2}-\frac{d}{p}+1)J}\|m^{\ell}\|_{\dot{B}^{\frac{d}{p}-1}_{p,1}}+2^{(\frac{d}{2}-\frac{d}{p})J}\|m\|_{\dot{B}^{\frac{d}{2}}_{2,1}}^{h}.
\end{aligned}
\end{equation}
Therefore, by \eqref{HLEst}, \eqref{cases}, \eqref{ffff3} and the fact that $s\leq \sigma$, we obtain 
\begin{equation}\nonumber
\begin{aligned}
&\sum_{\Omega_{\ell h}}2^{ks}\|\dot{\Delta}_k (\dot{\Delta}_{k'}m (\dot S_{k'-1} m+\dot\Delta_{k'}m ) \widetilde{M}_{k'})\|_{L^p}
\lesssim
2^{(s-\sigma+\frac{d}{2}-\frac{d}{p})J} (2^{J}\|m^{\ell}\|_{\dot{B}^{\frac{d}{p}-1}_{p,1}}+\|m\|_{\dot{B}^{\frac{d}{2}}_{2,1}}^{h}) \|m\|_{\dot B_{2,1}^{\sigma}}^{h}.
\end{aligned}
\end{equation}

\begin{itemize}
\item Case 2: $p\geq 2$ and $s\leq \sigma$.
\end{itemize}
In this case, one deduces from Bernstein's inequality, \eqref{HLEst} and \eqref{SmoothFIneq2} with $|\beta|=0$ and $(p^1,p^2,p^3)=(2,\infty, 2)$ that
\begin{equation*}
\begin{aligned}
&\sum_{\Omega_{\ell h}}2^{ks}\|\dot{\Delta}_k (\dot{\Delta}_{k'}m (\dot S_{k'-1} m+\dot\Delta_{k'}m ) \widetilde{M}_{k'})\|_{L^p}\\
&\quad\lesssim 2^{(\frac{d}{2}-\frac{d}{p})J}\sum_{\Omega_{\ell h}}2^{ks}\|\dot{\Delta}_k (\dot{\Delta}_{k'}m (\dot S_{k'-1} m+\dot\Delta_{k'}m ) \widetilde{M}_{k'})\|_{L^2}\\
&\quad\lesssim 2^{(\frac{d}{2}-\frac{d}{p})J} \|m\|_{L^{\infty}} \|m\|_{\dot B_{2,1}^{s}}^{h}\\
&\quad\lesssim 2^{(s-\sigma+\frac{d}{2}-\frac{d}{p})J} (2^{J}\|m^{\ell}\|_{\dot{B}^{\frac{d}{p}-1}_{p,1}}+\|m\|_{\dot{B}^{\frac{d}{2}}_{2,1}}^{h}) \|m\|_{\dot B_{2,1}^{\sigma}}^{h}.
\end{aligned}
\end{equation*}

\begin{itemize}
\item Case 3: $s>\sigma$.
\end{itemize}
By the low-frequency cut-off and calculations in Cases 1-2, it holds that
\begin{equation*}
\begin{aligned}
&\sum_{\Omega_{\ell h}}2^{ks}\|\dot{\Delta}_k (\dot{\Delta}_{k'}m (\dot S_{k'-1} m+\dot\Delta_{k'}m ) \widetilde{M}_{k'})\|_{L^p}\\
&\lesssim
2^{(s-\sigma)J} \sum_{\Omega_{\ell h}}2^{k\sigma} \|\dot{\Delta}_k (\dot{\Delta}_{k'}m \widetilde{M}_{k'}\dot S_{k'-1} m)\|_{L^p}\lesssim 2^{(s-\sigma+\frac{d}{2}-\frac{d}{p})J} (2^{J}\|m^{\ell}\|_{\dot{B}^{\frac{d}{p}-1}_{p,1}}+\|m\|_{\dot{B}^{\frac{d}{2}}_{2,1}}^{h}) \|m\|_{\dot B_{2,1}^{\sigma}}^{h}.
\end{aligned}
\end{equation*}

Gathering the above three cases, we derive
\begin{equation}\label{abcdef}
\begin{aligned}
&\sum_{\Omega_{\ell h}}2^{ks}\|\dot{\Delta}_k (\dot{\Delta}_{k'}m (\dot S_{k'-1} m+\dot\Delta_{k'}m ) \widetilde{M}_{k'})\|_{L^p}\\
&\lesssim
2^{(s-\sigma+\frac{d}{2}-\frac{d}{p})J} (2^{J}\|m^{\ell}\|_{\dot{B}^{\frac{d}{p}-1}_{p,1}}+\|m\|_{\dot{B}^{\frac{d}{2}}_{2,1}}^{h}) \|m\|_{\dot B_{2,1}^{\sigma}}^{h}
\end{aligned}
\end{equation}
for all $s>0$, $\sigma\in\mathbb{R}$ and $p\geq \max\{1,\frac{2d}{d+2}\}$. Adding \eqref{LemEst2qqq}, \eqref{LemEst1qqq} and \eqref{abcdef} together, we arrive at
\begin{equation*}
\begin{aligned}
&\sum_{ k\leq J,~ k'\in\mathbb{Z}}2^{ks}\|\dot{\Delta}_k (\dot{\Delta}_{k'}m (\dot S_{k'-1} m+\dot\Delta_{k'}m ) \widetilde{M}_{k'})\|_{L^p}\\
&\quad\lesssim (1+\|m\|_{L^\infty})^{[s]+1}
\big(\|m^{\ell}\|_{\dot B_{p,1}^{\frac{d}{p}} } +\|m\|_{\dot B_{2,1}^{\frac{d}{2}} }^h  \big)\|m\|_{\dot{B}^{s}_{p,1}}^{\ell}+2^{(s-\sigma+\frac{d}{2}-\frac{d}{p})J}\big(2^{J}\|m^{\ell}\|_{\dot B_{p,1}^{\frac{d}{p}-1}}+ 
\|m\|_{\dot B_{2,1}^{\frac{d}{2}}}^h\big)
\|m\|_{\dot B_{2,1}^{\sigma}}^{h}.
\end{aligned}
\end{equation*}
Similar computations yield
\begin{equation*}
\begin{aligned}
&\sum_{ k\leq J,~ k'\in\mathbb{Z}}2^{ks}\|\dot{\Delta}_k (\dot{\Delta}_{k'}m (\dot S_{k'-1} m+\dot\Delta_{k'}m ) M_{k'})\|_{L^p}\\
&\quad\lesssim (1+\|m\|_{L^\infty})^{[s]+1}
\big(\|m^{\ell}\|_{\dot B_{p,1}^{\frac{d}{p}} }+\|m\|_{\dot B_{2,1}^{\frac{d}{2}} }^h  \big)\|m\|_{\dot{B}^{s}_{p,1}}^{\ell}+2^{(s-\sigma+\frac{d}{2}-\frac{d}{p})J}\big(2^{J}\|m^{\ell}\|_{\dot B_{p,1}^{\frac{d}{p}-1}}+ 
\|m\|_{\dot B_{2,1}^{\frac{d}{2}}}^h\big)
\|m\|_{\dot B_{2,1}^{\sigma}}^{h}.
\end{aligned}
\end{equation*}
These above two estimates, combined with \eqref{com1}, imply  \eqref{LemNewSmoothlow}.
\end{proof}

\begin{Lemma} \label{NewSmoothhigh}
Let  $s>0$, $\sigma\in\mathbb{R}$, $1\leq p\leq 4$ for $d=1$ and $1\leq p\leq \min\{4,\frac{2d}{d-2}\}$ for $d\geq 2$. For any smooth function $F(m)$ satisfying $F(0)=F'(0)=0$, there is a constant $C_{m}>0$ depending only on $\|m\|_{L^{\infty}}$, $s$, $\sigma$ and $d$ such that
\begin{equation}\label{LemNewSmoothhigh2}
\begin{aligned}
\|F(m)\|_{\dot{B}^{s}_{2,1}}^{h}
&\leq C_{m}
\big(\|m^{\ell}\|_{\dot B_{p,1}^{\frac{d}{p}} } +\|m\|_{\dot B_{2,1}^{\frac{d}{2}} }^h\big)\|m\|_{\dot{B}^{s}_{2,1}}^{h}\\
&\quad+C_{m} 2^{(s-\sigma+\frac{d}{p}-\frac{d}{2})J}
\big(2^{J}\|m^{\ell}\|_{\dot B_{p,1}^{\frac{d}{p}-1}}+\|m\|_{\dot B_{2,1}^{\frac{d}{2}}}^h \big)
\|m\|_{\dot B_{p,1}^{\sigma}}^{\ell}.
\end{aligned}
\end{equation}
\end{Lemma}
\begin{proof}
As in Lemma \ref{NewSmoothlow}, we decompose the high-frequency regime as
\begin{equation}\nonumber
\begin{aligned}
&\Omega_{hh}\triangleq\{(k,k')~|~k'> k\geq J\},\quad \Omega_{hm}\triangleq\{(k,k')~|~k\geq k'\geq J\},\quad \Omega_{h\ell}\triangleq\{(k,k')~|~k\geq J>k'\}.
\end{aligned}
\end{equation}
Recall that $F(m)$ satisfies \eqref{com1}. By applying \eqref{SmoothFIneq2} with $|\beta|=0$ and $(p^1,p^2,p^3)=(2,\infty,2)$, we have
\begin{equation}\label{LemNewSmoothEst1}
\begin{aligned}
&\sum_{\Omega_{hh}}2^{ks}\|\dot{\Delta}_k (\dot{\Delta}_{k'}m (\dot S_{k'-1} m+\dot\Delta_{k'}m ) \widetilde{M}_{k'})\|_{L^2}\\
&\quad\lesssim
\|m\|_{L^{\infty}} \sum_{k'> J}2^{k's}\|\dot{\Delta}_{k'}m \|_{L^2}\sum_{k'> k}2^{(k-k')s}
\lesssim
\big(\|m^{\ell}\|_{\dot B_{p,1}^{\frac{d}{p}} } +\|m\|_{\dot B_{2,1}^{\frac{d}{2}} }^h\big)\|m\|_{\dot{B}^{s}_{2,1}}^{h}.
\end{aligned}
\end{equation}
Next, we apply \eqref{SmoothFIneq2} with $|\beta|=[s]+1$ and $(p^1,p^2,p^3)=(2,\infty,2)$ to obtain
\begin{equation}\label{LemNewSmoothEst2}
\begin{aligned}
&\sum_{\Omega_{hm}}2^{ks}\|\dot{\Delta}_k (\dot{\Delta}_{k'}m (\dot S_{k'-1} m+\dot\Delta_{k'}m ) \widetilde{M}_{k'})\|_{L^2}\\
&\lesssim
(1+\|m\|_{L^\infty})^{[s]+1}\|m\|_{L^{\infty}} \sum_{k\geq k'\geq J}2^{k's}\|\dot{\Delta}_{k'}m \|_{L^2}\sum_{k\geq k'}2^{(k-k')(s-[s]-1)}\\
&\lesssim(1+\|m\|_{L^\infty})^{[s]+1}
\big(\|m^{\ell}\|_{\dot B_{p,1}^{\frac{d}{p}} } +\|m\|_{\dot B_{2,1}^{\frac{d}{2}} }^h\big)\|m\|_{\dot{B}^{s}_{2,1}}^{h}.
\end{aligned}
\end{equation}
For the regime $\Omega_{h\ell}$,  we consider the following three cases.
\begin{itemize}
    \item Case 1: $2\leq p<4$ for $d=1$, $2\leq p\leq \min\{\frac{2d}{d-2},4\}$ for $d\geq2$ and $s\geq \sigma$.
\end{itemize}
Applying \eqref{SmoothFIneq2} with $|\beta|=[s]+1$ and $(p^1,p^2,p^3)=(2,\frac{2p}{p-2},p)$, we have 
\begin{equation*}
\begin{aligned}
&\sum_{\Omega_{h\ell}}2^{ks}\|\dot{\Delta}_k (\dot{\Delta}_{k'}m (\dot S_{k'-1} m+\dot\Delta_{k'}m ) \widetilde{M}_{k'})\|_{L^2}\\
&\quad\lesssim
(1+\|m\|_{L^\infty})^{[s]+1} \|m\|_{L^{\frac{2p}{p-2}}} \sum_{k'<J} 2^{k's}\|\dot{\Delta}_{k'}m \|_{L^p} \sum_{k>k'}2^{(k-k')(s-[s]-1)}\\
&\quad\lesssim (1+\|m\|_{L^\infty})^{[s]+1} 2^{(s-\sigma)J}\|m\|_{L^{\frac{2p}{p-2}}} \|m\|_{\dot B_{p,1}^{\sigma}}^{\ell}.
\end{aligned}
\end{equation*}
Since $p$ satisfies $\frac{2p}{p-2}\geq p\geq 2$ and $\frac{2d}{p}-\frac{d}{2}\geq \frac{d}{p}-1$, the low-high frequency decomposition together with embedding inequalities and \eqref{HLEst} implies
\begin{equation}\nonumber
\begin{aligned}
\|m\|_{L^{\frac{2p}{p-2}}}\lesssim \|m\|_{\dot{B}^{0}_{\frac{2p}{p-2},1}}&\lesssim \|m^{\ell}\|_{\dot{B}^{\frac{2d}{p}-\frac{d}{2}}_{p,1}}+\|m\|_{\dot{B}^{\frac{d}{p}}_{2,1}}^{h}\lesssim 2^{(\frac{d}{p}-\frac{d}{2}+1)J}\|m^{\ell}\|_{\dot{B}^{\frac{d}{p}-1}_{p,1}}+2^{-(\frac{d}{2}-\frac{d}{p})J}\|m\|_{\dot{B}^{\frac{d}{2}}_{p,1}}^{h}.
\end{aligned}
\end{equation}
Therefore, one has
\begin{equation}\label{LemNewSmoothEst5}
\begin{aligned}
&\sum_{\Omega_{h\ell}}2^{ks}\|\dot{\Delta}_k (\dot{\Delta}_{k'}m (\dot S_{k'-1} m+\dot\Delta_{k'}m ) \widetilde{M}_{k'})\|_{L^2}\\
&\lesssim (1+\|m\|_{L^\infty})^{[s]+1} 2^{(s-\sigma+\frac{d}{p}-\frac{d}{2})J}
\big(2^{J}\|m^{\ell}\|_{\dot B_{p,1}^{\frac{d}{p}-1}}+\|m\|_{\dot B_{2,1}^{\frac{d}{2}}}^h \big)
\|m\|_{\dot B_{p,1}^{\sigma}}^{\ell}.
\end{aligned}
\end{equation}

\begin{itemize}
    \item Case 2: $1\leq p\leq 2$ and $s\geq \sigma$.
\end{itemize}

In this case, one has $s+\frac{d}{p}-\frac{d}{2}\geq \sigma$. Taking advantage of Bernstein's inequality and \eqref{SmoothFIneq2} with $|\beta|=[s+\frac{d}{p}-\frac{d}{2}]+1$ and $(p^1,p^2,p^3)=(p,\infty,p)$, we have
\begin{equation*}
\begin{aligned}
&\sum_{\Omega_{h\ell}}2^{ks}\|\dot{\Delta}_k (\dot{\Delta}_{k'}m (\dot S_{k'-1} m+\dot\Delta_{k'}m ) \widetilde{M}_{k'})\|_{L^2}\\
&\quad\lesssim \sum_{\Omega_{h\ell}}2^{k(s+\frac{d}{p}-\frac{d}{2})}\|\dot{\Delta}_k (\dot{\Delta}_{k'}m (\dot S_{k'-1} m+\dot\Delta_{k'}m ) \widetilde{M}_{k'})\|_{L^p}\\
&\quad\lesssim
(1+\|m\|_{L^\infty})^{[s+\frac{d}{p}-\frac{d}{2}]+1} \|m\|_{L^{\infty}} \sum_{k'<J} 2^{k'(s+\frac{d}{p}-\frac{d}{2})}\|\dot{\Delta}_{k'}m \|_{L^p} \sum_{k>k'}2^{(k-k')(s+\frac{d}{p}-\frac{d}{2}-[s+\frac{d}{p}-\frac{d}{2}]-1)}\\
&\quad\lesssim (1+\|m\|_{L^\infty})^{[s+\frac{d}{p}-\frac{d}{2}]+1} 2^{(s-\sigma+\frac{d}{p}-\frac{d}{2})J}\big(2^{J}\|m^{\ell}\|_{\dot B_{p,1}^{\frac{d}{p}-1}}+\|m\|_{\dot B_{2,1}^{\frac{d}{2}}}^h \big) \|m\|_{\dot B_{p,1}^{\sigma}}^{\ell}.
\end{aligned}
\end{equation*}

\begin{itemize}
    \item Case 3: $s<\sigma$.
\end{itemize}
In this case, the high-frequency cut-off together with the above estimates in Cases 1-2 implies
\begin{equation*}
\begin{aligned}
&\sum_{\Omega_{h\ell}}2^{ks}\|\dot{\Delta}_k (\dot{\Delta}_{k'}m (\dot S_{k'-1} m+\dot\Delta_{k'}m ) \widetilde{M}_{k'})\|_{L^2}\\
&\quad\lesssim
2^{(s-\sigma)J} \sum_{\Omega_{h\ell}}2^{k\sigma} \|\dot{\Delta}_k (\dot{\Delta}_{k'}m \widetilde{M}_{k'}\dot S_{k'-1} m)\|_{L^2}\\
&\quad\leq C_{m} 2^{(s-\sigma+\frac{d}{p}-\frac{d}{2})J} \big(2^{J}\|m^{\ell}\|_{\dot B_{p,1}^{\frac{d}{p}-1}}+\|m\|_{\dot B_{2,1}^{\frac{d}{2}}}^h \big)  \|m\|_{\dot B_{p,1}^{\sigma}}^{\ell}.
\end{aligned}
\end{equation*}

From the estimates in the above three cases, it holds that
\begin{equation}\label{sgdgeee}
\begin{aligned}
&\sum_{\Omega_{h\ell}}2^{ks}\|\dot{\Delta}_k (\dot{\Delta}_{k'}m (\dot S_{k'-1} m+\dot\Delta_{k'}m ) \widetilde{M}_{k'})\|_{L^2}\\
&\leq C_{m} 2^{(s-\sigma+\frac{d}{p}-\frac{d}{2})J} \big(2^{J}\|m^{\ell}\|_{\dot B_{p,1}^{\frac{d}{p}-1}}+\|m\|_{\dot B_{2,1}^{\frac{d}{2}}}^h \big)  \|m\|_{\dot B_{p,1}^{\sigma}}^{\ell}
\end{aligned}
\end{equation}
for all $s, \sigma$ and $p$ given in Proposition \ref{NewSmoothhigh}.

Adding \eqref{LemNewSmoothEst1}, \eqref{LemNewSmoothEst2} and \eqref{sgdgeee}
together, we end up with \eqref{LemNewSmoothhigh2}, which completes the proof of Proposition \ref{NewSmoothhigh}.

\end{proof}

\subsection{Proof of Theorem \ref{Thm0}}


In this subsection we give the proof of Theorem \ref{Thm0} on the global well-posedness of the Cauchy problem for the viscous conservation law \eqref{Thm2uheat}. 

\noindent
\textbf{\emph{Proof of Theorem \ref{Thm0}}.} Since the local well-posedness can be shown by the standard linearization iteration process, we omit the proof of local well-posedness for brevity and focus on giving the necessary {\it{a priori}} estimates. By virtue of Lemma \ref{HeatRegulEstprop} to \eqref{Thm2uheat}, we have
\begin{equation}\label{uheatLinerEst1}
\begin{aligned}
&\|u^*\|_{\tL^{\infty}_{t}(\dot{B}^{\frac {d}{p}-1}_{p,1}\cap \dot{B}^{\frac{d}{p}}_{p,1})}
+\|u^*\|_{\tL^1_{t}(\dot{B}^{\frac{d}{p}+1}_{p,1}\cap \dot{B}^{\frac{d}{p}+2}_{p,1})}
\lesssim \|u^*_0\|_{\dot{B}^{\frac {d}{p}-1}_{p,1}\cap \dot{B}^{\frac{d}{p}}_{p,1}}
+\|f(u^*)\|_{\tL_t^1(\dot{B}^{\frac{d}{p}}_{p,1}\cap \dot{B}^{\frac{d}{p}+1}_{p,1})} .
\end{aligned}
\end{equation}
Now, we assume $\|u^{*}\|_{L^{\infty}_{t}(L^{\infty})}\lesssim 1$. For $i=1,2,...,d$, due to $f_{i}(0)=\frac{\partial}{\partial_{u_{k}}}f_{i}(0)=0$, there exists a smooth function $\widetilde{f}_{i}(u)$ such that $f_{i}(u)=\widetilde{f}_{i}(u)u$ and $\widetilde{f}_{i}(0)=0$. Thus, making use of \eqref{ClassicalProductLawEst2} and Lemma \ref{DifferComposition}, we deduce for any $1\leq p\leq \infty$ that
\begin{equation}\label{uheatNonLinerAdditionalLowFreq}
\begin{aligned}
\|f_{i}(u^*)\|_{\tL_t^1(\dot{B}^{\frac {d}{p}}_{p,1})}\lesssim \|\widetilde{f}_{i}(u)\|_{\tL_t^2(\dot{B}^{\frac {d}{p}}_{p,1})}\|u\|_{\tL_t^2(\dot{B}^{\frac {d}{p}}_{p,1})}\lesssim \|u\|_{\tL_t^2(\dot{B}^{\frac {d}{p}}_{p,1})}^2.
\end{aligned}
\end{equation}
Similarly, it holds by \eqref{ClassicalProductLawEst1}, Lemma \ref{DifferComposition} and the embedding $\dot{B}^{\frac{d}{p}}_{p,1} \hookrightarrow L^{\infty}(\mathbb{R}^{d})$ that
\begin{equation}\label{uheatNonLinerAdditionalLowFreq1}
\begin{aligned}
\|f_{i}(u^*)\|_{\tL_t^1(\dot{B}^{\frac{d}{p}+1}_{p,1})}
&\lesssim \|\widetilde{f}_{i}(u)\|_{\tL_t^2(\dot{B}^{\frac{d}{p}}_{p,1})}\|u\|_{\tL_t^2(\dot{B}^{\frac {d}{p}+1}_{p,1})}+\|\widetilde{f}_{i}(u)\|_{\tL_t^2(\dot{B}^{\frac{d}{p}+1}_{p,1})}\|u\|_{\tL_t^2(\dot{B}^{\frac {d}{p}}_{p,1})}\\
&\lesssim \|u\|_{\tL_t^2(\dot{B}^{\frac {d}{p}}_{p,1})}\|u\|_{\tL_t^2(\dot{B}^{\frac {d}{p}+1}_{p,1})}.
\end{aligned}
\end{equation}
Inserting \eqref{uheatNonLinerAdditionalLowFreq}-\eqref{uheatNonLinerAdditionalLowFreq1} into \eqref{uheatLinerEst1} and taking advantage of the interpolation in Lemma \ref{Classical Interpolation}, we obtain
\begin{equation}\nonumber
\begin{aligned}
&\|u^*\|_{\tL^{\infty}_{t}(\dot{B}^{\frac {d}{p}-1}_{p,1}\cap \dot{B}^{\frac{d}{p}}_{p,1})}
+\|u^*\|_{\tL^1_{t}(\dot{B}^{\frac{d}{p}+1}_{p,1}\cap \dot{B}^{\frac{d}{p}+2}_{p,1})}\\
&\quad\lesssim
\|u^*_0\|_{\dot{B}^{\frac {d}{p}-1}_{p,1}\cap \dot{B}^{\frac{d}{p}}_{p,1}}+\Big(\|u^*\|_{\tL^{\infty}_{t}(\dot{B}^{\frac {d}{p}-1}_{p,1}\cap \dot{B}^{\frac{d}{p}}_{p,1})}
+\|u^*\|_{\tL^1_{t}(\dot{B}^{\frac{d}{p}+1}_{p,1}\cap \dot{B}^{\frac{d}{p}+2}_{p,1})}\Big)^2.
\end{aligned}
\end{equation}
Thence, by a standard bootstrap argument, one can prove
\begin{equation}\label{uheatLinerEst2}
\begin{aligned}
&\|u^*\|_{\tL^{\infty}_{t}(\dot{B}^{\frac {d}{p}-1}_{p,1}\cap \dot{B}^{\frac{d}{p}}_{p,1})}
+\|u^*\|_{\tL^1_{t}(\dot{B}^{\frac{d}{p}+1}_{p,1}\cap \dot{B}^{\frac{d}{p}+2}_{p,1})}\lesssim
\|u^*_0\|_{\dot{B}^{\frac {d}{p}-1}_{p,1}\cap \dot{B}^{\frac{d}{p}}_{p,1}}\quad \text{for all}\quad t>0.
\end{aligned}
\end{equation}
This, together with the local well-posedness, shows the global existence of a solution $u^*$ to the Cauchy problem of System \eqref{Thm2uheat} associated with the initial data $u^*_0$. The property $u^{*}\in C(\mathbb{R}_{+};\dot{B}^{\frac{d}{p}-1}_{p,1}\cap\dot{B}^{\frac{d}{p}}_{p,1})$ follows a  similar argument as in \cite[p.42]{Danchinnote}. With the aid of  \eqref{Thm2vdarcy}, \eqref{uheatNonLinerAdditionalLowFreq}, \eqref{uheatNonLinerAdditionalLowFreq1} and \eqref{uheatLinerEst2}, we recover the information on $\bv^*$ as follows:
\begin{equation*}
\begin{aligned}
&\|\bv^*\|_{\tL^1_{t}(\dot{B}^{\frac{d}{p}}_{p,1}\cap \dot{B}^{\frac{d}{p}+1}_{p,1})}
\lesssim
\|u^*\|_{\tL^1_{t}(\dot{B}^{\frac{d}{p}+1}_{p,1}\cap \dot{B}^{\frac{d}{p}+2}_{p,1} )}+
\|f(u^*)\|_{\tL_t^1(\dot{B}^{\frac{d}{p}}_{p,1}\cap \dot{B}^{\frac{d}{p}+1}_{p,1})}
\lesssim
\|u^*_0\|_{\dot{B}^{\frac {d}{p}-1}_{p,1}\cap \dot{B}^{\frac{d}{p}}_{p,1}}.
\end{aligned}
\end{equation*}
\qed


\vspace{7ex}

(T. Crin-Barat)\par\nopagebreak
\noindent\textsc{Université Paul Sabatier,  Institut de Math\'ematiques de Toulouse, Route de Narbonne 118, 31062 Toulouse Cedex 9, France.}

Email address: {\tt timoth\'ee.crin-barat@math.univ-toulouse.fr}

\vspace{3ex}

(L.-Y. Shou)\par\nopagebreak
\noindent\textsc{School of Mathematical Sciences and Mathematical Institute, Nanjing Normal University, Nanjing, 210023, P. R. China}

Email address: {\tt shoulingyun11@gmail.com}

\vspace{3ex}

(J. Zhang)\par\nopagebreak
\noindent\textsc{School of Mathematics and Information Science, Shandong Technology and Business University, Yantai, Shandong, 264005,P. R. China}

Email address: {\tt zhangjz\_91@nuaa.edu.cn}

\end{document}